\documentclass[10pt,oneside,reqno]{amsart}
\usepackage[utf8]{inputenc}
\usepackage[UKenglish]{babel}

\usepackage{algorithm,algorithmic}

\usepackage{cite}
\usepackage{hyperref}
\usepackage{verbatim}
\usepackage{mathtools}
\usepackage{graphicx}
\usepackage{geometry}
\usepackage{color}

\theoremstyle{definition}
\newtheorem{anytheorem}{Theorem}[section] 
\newtheorem{definition}[anytheorem]{Definition}
\newtheorem{theorem}[anytheorem]{Theorem}

\newtheorem{lemma}[anytheorem]{Lemma}
\newtheorem{remark}[anytheorem]{Remark}
\newtheorem{example}[anytheorem]{Example}
\newtheorem{assumption}[anytheorem]{Assumption}
\allowdisplaybreaks

\DeclareMathOperator*{\argmin}{arg\,min}
\DeclareMathOperator*{\esssup}{ess\,sup}
\DeclareMathOperator*{\tr}{Tr}

\newcommand{\E}{\ensuremath{\mathbb{E}}}

\begin{document}

\title{The Modified MSA, a Gradient Flow and Convergence}

\author{Deven Sethi}
\address{School of Mathematics, University of Edinburgh}
\email{D.Sethi-1@sms.ed.ac.uk}

\author{David \v{S}i\v{s}ka}
\address{School of Mathematics, University of Edinburgh and Vega Protocol, Gibraltar}
\email{D.Siska@ed.ac.uk}

\date{\today}




\begin{abstract}
The modified Method of Successive Approximations (MSA) is an iterative scheme for approximating solutions to stochastic control problems in continuous time based on Pontryagin Optimality Principle which, starting with an initial open loop control, solves the forward equation, the backward adjoint equation and then performs a static minimization step. 
We observe that this is an implicit Euler scheme for a gradient flow system. 
We prove that appropriate interpolations of the iterates of the modified MSA converge to a gradient flow with rate $\tau$. 
We then study the convergence of this gradient flow as time goes to infinity. 
In the general (non-convex) case we prove that the gradient term itself converges to zero. 
This is a consequence of an energy identity which shows that the optimization objective decreases along the gradient flow.  
Moreover, in the convex case, when Pontryagin Optimality Principle provides a sufficient condition for optimality, we prove that the optimization objective converges at rate $\tfrac{1}{S}$ to its optimal value
and at exponential rate under strong convexity. 
The main technical difficulties lie in obtaining appropriate properties of the Hamiltonian (growth, continuity). 
These are obtained by utilising the theory of Bounded Mean Oscillation (BMO) martingales required for estimates on the adjoint Backward Stochastic Differential Equation (BSDE). 
\end{abstract}

\maketitle

\section{Introduction}
Stochastic control problems in continuous space and time have numerous applications in engineering, finance and economics. 
The two main approaches to solving such problems are to work with the Hamilton--Jacobi--Bellman (HJB) equation, see e.g. Krylov~\cite{krylov2008controlled}, Fleming and Soner~\cite{fleming2006controlled} and Pontryagin's Optimality Principle, see e.g. Pham~\cite{pham2009continuous} or Carmona~\cite{carmona2016lectures}. 
Both methods have been extensively studied over the last several decades. 
However, neither method yields closed form solutions for most problems of practical interest and hence one has to resort to numerical methods.

Whether one uses the HJB equation or Pontryagin's Optimality Principle one is faced with a nonlinear problem in continuous space and time and a choice: should one discretise first and then face a discrete but still nonlinear problem or linearise first and then discretise a linear problem? 
In this paper we focus on a method which linearises the problem arising when one builds an algorithm based on Pontryagin's optimality principle: namely the modified Method of Successive Approximations (MSA). 

Before going into details about this paper we give a brief overview of other related results. 
Our paper focuses on iterative methods for Pontryagin Optimality Principle so we discuss the methods based on the HJB only for completeness.
If one chooses to first discretise the HJB equation then one can use finite differences, see~\cite{dong2007rate, barles2002convergence, gyongy2009finite} and the references therein for convergence and error bounds. 
Such methods will typically suffer from the curse of dimensionality.
Indeed the term has been coined by Bellman as the problem arises in numerical solutions of dynamic optimization problems~\cite{bellman1957dynamic}.

If one linearises the HJB equation via iteration then this leads to the continuous space and time analogues of the classical value or policy iteration algorithms. 
At each step of the algorithm one has to solve (most likely approximate) a linear PDE and then in the case of policy iteration also update the Markovian control function by solving a static minimization problem. 
Recently Jentzen et al. have proved that the kind of linear PDEs that would arise can be approximated by neural networks without suffering from the curse of dimensionality, see e.g.~\cite{grohs2018aproof, jentzen2021proof, gonon2022uniform} and references therein. 
There are also many papers examining algorithms that provide good approximations of such linear PDEs in high dimensions, see e.g.~\cite{e2017deep, sirignano2018dgm, sabate2021unbiased}. 
The point is that it makes sense to first linearise the problem via iteration and then solve the arising linear PDE with whichever well-studied method is the most appropriate.
For results concerning convergence of algorithms which linearise the HJB equation via iteration for continuous space and time problems see~\cite{puterman1981onconvergence, jacka2017onpolicy, jacka2017coupling, maeda2017evaluation, kerimkulov2020exponential, dong2022randomized,huang2022convergence} and references therein.
Note that policy and value iteration algorithms go back to the very origins of stochastic control (or Markov decision processes). 
The original formulations of the finite-state-and-action space versions of the algorithms go back to Bellman, see~\cite{bellman1957dynamic}.

Let us now turn our attention to iterative methods based on the Pontryagin optimality principle. 
Indeed these are the main focus of this paper.
Under appropriate conditions (essentially on the growth and differentiability of the drift, diffusion and cost functions) this provides a necessary condition for optimality.
Under further convexity assumptions it also provides a sufficient condition for optimality, see e.g.~\cite[Ch. 6]{pham2009continuous} 
or~\cite[Ch. 4]{carmona2016lectures}.
We will now give a heuristic overview of the problem, the algorithm and our results.
All notation and assumptions will be stated precisely in Sections \ref{sec notation} and~\ref{Formulation of Problem}.
Our objective is to minimise the functional  
\begin{equation}\label{Gain_functional_intro}
J(\alpha)=\E\left[\int_0^{T}f(t,X_t^\alpha,\alpha_{t})dt+g(X_{T}^{\alpha})\right],
\end{equation}
over $\alpha\in L^{2}_{\text{prog}}([0,T]\times\Omega;\mathbb{R}^{p})$, which is the space of admissible open loop controls, 
subject to
\begin{align}
\label{controlled_SDE_intro}
dX_{t}^{\alpha}=b(t,X_{t}^{\alpha},\alpha_{t})dt+\sigma(t,X_{t}^{\alpha},\alpha_{t})dW_{t},\enspace t\in[0,T],\enspace X_0^{\alpha} = x.
\end{align}
We define the Hamiltonian as 
\begin{align*}
\mathcal{H}(t,x,y,z,a)=b(t,x,a)\cdot y + \sigma(t,x,a):z + f(t,x,a)\,.
\end{align*}
The adjoint equation is then given by the Backward Stochastic Differential Equation (BSDE)
\begin{align}\label{adjoint equation_intro}
dY_{t}^{\alpha} =-D_{x}\mathcal{H}(t,X_{t}^{\alpha},Y_{t}^{\alpha},Z_{t}^{\alpha},\alpha_{t})dt + Z_{t}^{\alpha}dW_{t},\enspace t\in[0,T],\enspace Y_{T}^{\alpha}=D_{x}g(X_{T}^{\alpha}).
\end{align}
Pontryagin Optimality Principle stipulates that if $\hat{\alpha}$ is an (at least locally) optimal control, and $\hat{X}, \hat{Y},\hat{Z}$ are the corresponding solutions to~\eqref{controlled_SDE_intro} and~\eqref{adjoint equation_intro} then for almost all $(t,\omega)$ we have
\begin{align}
\label{eq optimality cond intro}
\mathcal{H}(t,\hat X_{t},\hat Y_{t},\hat Z_{t},\hat \alpha_{t})=\inf_{a \in \mathbb R^p}\mathcal{H}(t,\hat X_{t},\hat Y_{t},\hat Z_{t},a)\,.
\end{align}
Under the assumption that $g$ is convex and $\mathcal H$ is convex in $(x,a)$  in all other parameters~\eqref{eq optimality cond intro} is now a sufficient condition: if $\hat \alpha$, $\hat X$, $\hat Y$, $\hat Z$ satisfy~\eqref{controlled_SDE_intro}, \eqref{adjoint equation_intro} and~\eqref{eq optimality cond intro} then $J(\hat \alpha) = \inf_{\alpha} J(\alpha)$.  

Based on this one can see why the MSA algorithm (Algorithm~\ref{alg mmsa} with $\tau=\infty$) is a natural iterative method.
At each step of the algorithm we solve the SDE, the adjoint BSDE 
and then perform a static optimization step. 
It has been observed already in~\cite{iakrylov1963on} 
that this algorithm is divergent if $\tau=\infty$ in many cases (even for deterministic control problems).
This can be corrected by addition of extra terms over which to perform the minimization which leads to the ``modified'' MSA. 
In our case this corresponds to $\tau<\infty$ but other modifications 
are possible, see~\cite{iakrylov1963on} and~\cite{kerimkulov2021modified} for the stochastic case.
We will show in Lemma~\ref{cost_functional_increment_bound} and Remark~\ref{rmk:cost_functional_increment_bound} that taking a sufficiently large modification ensures that the cost functional $J$ is decreasing with each iteration.

\begin{algorithm}[H]
\caption{Modified Method of Successive Approximations:}
\label{alg mmsa}
\begin{algorithmic}[1]
\STATE{Initialisation: take $\tau < \max\{\lambda,2C_{J, \mathcal H}\}^{-1}$, where $C_{J, \mathcal H}$ is the constant from Lemma~\ref{cost_functional_increment_bound} and $\lambda$ is the constant from Assumption \ref{new lambda convexity}.}
\STATE{Initialisation: make a guess of the control $\alpha^0 = (\alpha^0_t)_{t\in[0,T]}$.}
\REPEAT
\STATE{Solve 
\begin{align*}
dX_t^n &= b(t,X_t^n,\alpha^n_t)dt+\sigma(t,X_t^n,\alpha_t^n)dW_t,\enspace X_0^n=x,\\
dY_{t}^{n}&=-D_{x}\mathcal{H}(t,X_t^n,Y_t^n,Z_t^n,\alpha_t^{n})dt+Z_{t}^{n}dW_{t},\enspace Y_T^{n}=D_{x}g(X_{t}^{n}).
\end{align*} }
\STATE{For each $t\in[0,T]$ and $\omega\in\Omega$ update the control 
\begin{align}\label{gradient_requirement}
\alpha^{n+1}_{t}=\argmin_{a\in \mathbb R^p}\mathcal{H}(t,X_{t}^{n},Y_{t}^{n},Z_{t}^{n},a)+\frac{1}{2\tau}|a-\alpha_{t}^{n}|^2.
\end{align}}
\UNTIL{$0\leq J(\alpha^{n})-J(\alpha^{n+1})$ is sufficiently small.}
\RETURN $J(\alpha^{n})$.
\end{algorithmic}	
\end{algorithm}

First we show that for $s\in [0,S]$ and $t\in [0,T]$ the continuous time gradient flow system
\begin{equation}
\label{Gradient_flow_system_intro}
\left\{ 
\begin{split}
\frac{d}{ds}\alpha_{s,t} & =-D_{a}\mathcal{H}(t,X_{s,t},Y_{s,t},Z_{s,t},\alpha_{s,t}),\,\,\,\text{$\alpha_{0,t}=\alpha^{0}_{t}$},\\    
dX_{s,t} & =b(t,X_{s,t},\alpha_{s,t})dt + \sigma(t,X_{s,t},\alpha_{s,t})dW_{t}\,, \,\,\, X_{s,0} = x\,, \\
dY_{s,t} & = D_{x}\mathcal{H}\left(t,X_{s,t},Y_{s,t},Z_{s,t},\alpha_{s,t}\right)dt + Z_{s,t}dW_{t}\,,\,\,\, Y_{s,T} = D_{x}g(X_{s,T})\,,\\
\end{split}
\right.
\end{equation}
has a unique solution, see Theorem~\ref{thm:existence_GF}.
The time $t$ corresponds to the original stochastic control problem time while the time $s$, which is the gradient flow time, corresponds to the iteration number $n$ in Algorithm~\ref{alg mmsa}. 

The two main results of the paper are the following. 

We show that as $\tau \to 0$ 
the iterates of the modified MSA (appropriately interpolated) converge with rate $\tau$ to~\eqref{Gradient_flow_system_intro},
see Theorem~\ref{thm:convergence_rate}.
It is readily seen from the statement of the theorem that the rate is not uniform in $S>0$ unless the problem itself is sufficiently strongly convex.
Our second main result concerns the convergence (in appropriate sense) of~\eqref{Gradient_flow_system_intro} as the gradient flow time goes to infinity.
In Theorem~\ref{prop: L^2-convergence of D_aH} we show that in the non-convex case $D_a \mathcal H(X_s,Y_s,Z_s,\alpha_s) \to 0$ in $L^{2}_{\text{prog}}([0,T]\times\Omega;\mathbb{R}^{p})$.  
Under the convexity assumptions which ensure that Pontryagin Optimality Principle is a sufficient condition we prove that $J(\alpha_S) \to \inf_{\alpha}J(\alpha)$ with rate $1/S$.
Moreover, we show that with strong convexity, 
$J(\alpha_S)$ and $\alpha_S$ converge exponentially to $J(\alpha^*)$ and $\alpha^*$ respectively, where $\alpha^*$ represent the optimal control.
This is Theorem~\ref{thm: Convergence_of_value_fiunction}.

Let us now illustrate the main results of this paper by analogy with a finite dimensional case. 
To that end consider $F:\mathbb R^p \to \mathbb R$ and an iteration of the form 
\[
a^\tau_n = \argmin_a F(a) + \tfrac1{2\tau}|a-a^\tau_{n-1}|^2\,.
\]
This is the minimizing movement scheme, see e.g.~\cite{santambrogio2017euclidean}, which is a finite dimensional analogue of the modified MSA algorithm; certainly the connection to~\eqref{gradient_requirement} is clear. 
If $F$ is $\lambda$-convex then we can take $\tau>0$ small enough so that the arg min is indeed single valued. 
Employing the first order condition of calculus and the convexity of $a\mapsto F(a) + \tfrac1{2\tau}|a-a^\tau_{n-1}|^2$ we see that the $\argmin$ condition is equivalent to
\[
\tfrac{a_{n+1}^\tau - a_n^\tau}{\tau} = -D_a F(a_{n+1}^\tau)\,
\]
which is the implicit Euler scheme with timestep $\tau$ for $\frac{d}{ds}a_s = -D_a F(a_s)$. 
Interpolating the iterates one can show that these interpolations indeed converge to the ODE. 
In this paper the analogous result is Theorem~\ref{thm:convergence_rate}.
A key step is to obtain uniform (in $\tau$) a priori bounds (which we prove in our setting in Section~\ref{Existence of the Gradient Flow}).
For this we need fine estimates on the solutions of the BSDE~\eqref{adjoint equation_intro} which employ the theory of exponential martingales and bounded mean oscillation (BMO) martingales, see Section~\ref{sect:existence of solut}. 
Once the BSDE estimates are established we use these together with the uniform a priori bounds to obtain the convergence rate.

Still continuing with the finite dimensional analogy let us consider convergence of $\tfrac{d}{ds}a_s = -D_a F(a_s)$ as $s\to \infty$.
From the chain rule we get $\frac{d}{ds}F(a_s) = -|D_a F(a_s)|^2$. 
Integrating this we see that
\[
F(a_S) - F(a_0) = -\int_0^S |D_a F(a_s)|^2\,ds
\]
and hence 
\begin{align*}
\int_0^\infty |D_a F(a_s)|^2\,ds = F(a_0) - \lim_{S\to \infty} F(a_S) 
\leq F(a_0) - \inf_a F(a) < \infty 
\end{align*}
as long as $F$ is lower bounded. 
If it is also smooth enough (e.g. so that $D_a F$ is Lipschitz continuous) then $s\mapsto |D_a F(a_s)|^2$ is uniformly continuous on $[0,\infty)$ and hence $\lim_{s\to \infty} |D_a F(a_s)|^2 = 0$.
In this paper the analogous result about~\eqref{Gradient_flow_system_intro} is Theorem~\ref{prop: L^2-convergence of D_aH}. 
The greatest difficulty arises in proving the energy identity for~\eqref{Gradient_flow_system_intro}, which is Lemma~\ref{functional_decreases_along_flow}, and in obtaining the required uniform-in-time continuity.
Finally, let us consider the convex case where $F$ satisfies  
$F(a)-F(b) \geq D_a F(b)\cdot(a-b)$.
Then for any fixed $a^\ast \in \argmin_a F(a)$ we have, due to convexity, that 
\[
\tfrac{d}{ds}|a_s - a^\ast|^2 = -2(a_s - a^\ast)\cdot D_a F(a_s) \leq 2F(a^\ast) - 2F(a_s)\,.
\]
Integrating and observing that due to the energy identity $F(a_S) \leq F(a_s)$ whenever $S\geq s$ we have 
\begin{align*}
2S(F(a_S) - F(a^\ast))&\leq 2\int_0^S \big(F(a_s) - F(a^\ast)\big)\,ds \leq - \int_0^S \tfrac{d}{ds}|a_s - a^\ast|^2\,ds \\
&\leq |a_0 - a^\ast|^2 - |a_S - a^\ast|^2.
\end{align*}
Hence $0\leq F(a_S) - F(a^\ast) \leq \tfrac1{2S} |a_0 - a^\ast|^2$.
If $F$ was strongly convex such that it would satisfy $F(a)-F(b) \geq D_a F(b)\cdot(a-b) + \frac{\eta}{2}|a-b|^2$ for some $\eta>0$, then a suitably modified argument yields the exponential convergence.
In Theorem~\ref{thm: Convergence_of_value_fiunction} we prove an analogous result for~\eqref{Gradient_flow_system_intro}.
The key observation which enables this is that if the Pontryagin Optimality Principle holds as a sufficient condition then for any admissible $\alpha$ and $\beta$ we have
\begin{align}\label{eqn:convexity_notion}
J(\alpha) - J(\beta)\geq (D_a\mathcal{H}\left(X^{\beta},Y^{\beta},Z^{\beta},\beta\right),\alpha-\beta)\,.
\end{align}
Here $(\cdot, \cdot)$ denotes the inner product in $L^{2}_{\text{prog}}([0,T]\times\Omega;\mathbb{R}^{p})$.
See Lemma~\ref{lemma: convexity of cost}.

We will now look at related existing results. 
We already mentioned that the study of the MSA (in the deterministic setting) and the proposal to modify it to ensure some convergence goes back to~\cite{iakrylov1963on}.
Analogous results for the stochastic case are proved in~\cite{kerimkulov2021modified} under the assumption $D^2_x \sigma = 0$.
It is proved that the algorithm improves the optimization objective and that $L^2$ difference of the Hamiltonians evaluated at two consecutive iterates of the algorithm goes to zero.
This is effectively a discrete analogue of Theorem~\ref{prop: L^2-convergence of D_aH} here.
Moreover in~\cite{kerimkulov2021modified} it is shown that if the problem is convex in the sense that Pontryagin optimality principle is a sufficient condition for optimality and if it has a structure that separates the space and control space dependence of the drift, diffusion and cost functions, then the modified MSA converges at rate of $1/n$.
This is analogous to Theorem~\ref{thm: Convergence_of_value_fiunction} but requiring more restrictive assumptions. 
In~\cite{ji2021modified} the authors study the stochastic recursive optimal control problem. 
Moreover, they employ an optimality principle which goes back to~\cite{peng1990general} and removes the need for the space of controls to be convex (in our case $\mathbb R^p$).
Finally, by using BMO martingale theory they are able to drop the assumption $D^2_x \sigma = 0$. 
Following their lead, we employ the same technique in this paper.
Gradient systems similar to~\eqref{Gradient_flow_system_intro} have been studied, without making the connection to the modified MSA algorithm but considering relaxed (or more generally measure-valued) controls giving rise to gradient flows on Wasserstein space, see~\cite{hu2019meanode, jabir2019mean} for the non-stochastic entropy regularized case and~\cite{siska2020gradient} for the stochastic control entropy regularized case.
The main results there consist of a formulation of Pontryagin Optimality as a necessary condition in the measure-valued-controls case and then in proving an exponential rate of convergence to the optimal open loop and measure-valued control under the assumption that the costs are strongly convex and/or the entropy regularization sufficiently strong, relative to the regularity and growth of the drift and diffusion, so that the gradient flow is strongly dissipative.

In~\cite{reisinger2022linear} the authors consider what can be seen as a variant of the modified MSA algorithm: proximal policy gradient method. 
They identify various conditions (small time horizon $T$, large discount factor, running cost sufficiently convex or weak dependence of costs on state) for linear i.e. $1/n$ rate of convergence of the algorithm.
They also allow for the inclusion of non-differentiable but convex running cost $\ell$ which can play a role of e.g. enforcing control space constraints.
If $\ell = 0$ then the proximal update step which replaces~\eqref{gradient_requirement} is
\[
\alpha_t^{n+1} \in \argmin_{a \in \mathbb R^p} \left( D_a H(t,X^n_t,Y^n_t, Z^n_t, \alpha^n_t) + \tfrac1{2\tau}|a-\alpha_t^n|^2 \right)\,.
\]
Notice that if the argmin is single valued then this is equivalent to the explicit Euler scheme for the first component of~\eqref{Gradient_flow_system_intro} and thus it can be viewed as another discretization of the gradient flow system studied in this paper.

One could convincingly argue that the modified MSA algorithm is of little practical interest in the stochastic case.
One has to solve a BSDE at each step of the iteration and once the algorithm terminated it only yields  an estimate of the optimal cost functional $J$ for one fixed $x\in \mathbb R^d$ and an estimated close-to-optimal open loop control.
However, it is possible to replace open loop controls with Markov controls in the modified MSA algorithm, at least in the case when the control doesn't enter the diffusion coefficient. 
This rests on the observation that in the Markovian framework the solution to the backward equations have representation as $Y^n_t = u^n(t,X^n_t)$ and $Z^n_t = \sigma^{-1}(t,X^n_t)D_x u^n(t,X^n_t)$ where $u^n$ is a solution of a parabolic PDE. 
One can then replace the step solving BSDE in Algorithm~\ref{alg mmsa} with solving the appropriate linear PDE and parametrize the Hamiltonian update step by $(t,x)$ instead of $(t,\omega)$ thus obtaining Markov control function at each iteration step. 
This is described in much more detail (including numerical experiments) in~\cite{reisinger2021fast}. 
We have already discussed earlier that it is possible to solve linear PDEs efficiently even in high dimensions using machine learning methods.
So in fact the theoretical analysis of the modified MSA presented in this paper forms a component of convergence analysis of an algorithm which is of practical interest.

\textbf{Organisation of the paper:}
In the remainder of the introduction we will fix the notation, see Section~\ref{sec notation} and we will fully state the stochastic control problem under consideration including assumptions required to make it well posed, see Section~\ref{Formulation of Problem}. 
In Sections \ref{sect:existence} and~\ref{Intro: Convergence to GF} we will present the main results regarding the existence of solutions to the gradient flow and convergence of the modified MSA to the gradient flow~\eqref{Gradient_flow_system_intro} respectively. 
In Section~\ref{Intro: convergence to OC} results on the convergence of the gradient flow as $S\to \infty$ are presented.
In Section \ref{sect:existence of solut} we prove the results outlined in Section~\ref{sect:existence} .
Section~\ref{Existence of the Gradient Flow} is devoted to proving results announced in Section~\ref{Intro: Convergence to GF} while in Section~\ref{Sect: convergence to OC} we prove results announced in Section~\ref{Intro: convergence to OC}.
In Section~\ref{sec examples} we collect a few examples that fit the setting of the paper. 
In Appendix~\ref{sec appendix} we collect results and their proofs which are mostly standard or the technique appeared elsewhere and require perhaps only a minor modification to fit our setting; we do not consider these to be a main contribution of this work.

\subsection{Notation}
\label{sec notation}
For any $m\in \mathbb N$ and $x,y\in \mathbb R^m$ we use $x\cdot y$ to denote the dot product. 
For matrices $A, B\in \mathbb R^{m'\times m}$ we let $A : B := \text{trace}(A^\top B)$.
We use the notation $|\cdot|$ for a norm on any finite dimensional vector space.
We will use standard notation for Lebesgue and Bochner spaces. 
We will use $\|\cdot\|_X$ to denote the norm in any infinite dimensional Banach space $X$. 
In particular if $X$ is a complete separable Banach space then $L^2(0,S;X)$ is the space of Bochner integrable and $X$-valued functions such that $\|f\|_{L^2(0,S;X)}:=\left(\int_0^S \|f_s\|_X^2\,ds\right)^{1/2} <\infty$.
Let $C\left(\left[0,S\right];X\right)$ represent the space of continuous $X$-valued functions such that $\|f\|_{C\left(\left[0,S\right];X\right)}:=\sup_{s\in[0,S]}\left\|f_s\right\|_X<\infty$. 
The subspace of $C\left(\left[0,S\right];X\right)$ of absolutely continuous functions will be denoted by
$AC\left(\left[0,S\right];X\right)$.
For any infinite dimensional Hilbert space we will write $(\cdot,\cdot)$ to denote the inner product.
The constants $c, C >0$ appearing can change from line to line and can depend on the constants declared in Section~\ref{Formulation of Problem} and problem dimensions i.e. they can depend on $K$, $T$, $d$, $d'$, $p$.
They will always be independent of the control (i.e. typically $\alpha$), $n$, $s$, $S$, $\tau$.  
We will sometimes use subscript to explicitly declare the dependence.
Whenever we take supremum or infimum over $\alpha$ it is implicitly understood that this is over $\alpha \in L^2_{\text{prog}}([0,T]\times \Omega);\mathbb R^p)$ unless stated otherwise.
Finally, for $n\in\mathbb{N}$, we will use $\alpha^n$ to represent the $n$-th open loop control generated by the update step of the MMSA, see Algorithm~\ref{alg mmsa}.

There will be additional notation related to BMO martingales introduced in Section~\ref{sect:existence of solut}, see in particular~\eqref{notation_space_1}, \eqref{notation_norm} and Definition~\ref{def bmo}.
As it isn't needed to state the main results we do not introduce it here.

\subsection{Formulation of the Problem}
\label{Formulation of Problem}
Fix some $K>0$, $T\in[0,\infty)$ and let $(\Omega,\mathcal{F},\mathbb{F}=(\mathcal{F}_{t})_{t\geq 0},\mathbb{P})$ be a filtered probability space and $W=(W^{1},\dots,W^{d'})$ be a $d'$-dimensional $\mathbb{F}$-Weiner process. 
Assume that we are given measurable functions $(b,\sigma,f):[0,T]\times\mathbb{R}^{d}\times \mathbb{R}^{p}\rightarrow (\mathbb{R}^{d}\times\mathbb{R}^{d\times d'}\times\mathbb{R})$ and $g:\mathbb R^d \to \mathbb R$ and our objective is to minimise the functional  
\begin{equation}\label{Gain_functional}
J(\alpha)=\E\left[\int_0^{T}f(t,X_t^\alpha,\alpha_{t})dt+g(X_{T}^{\alpha})\right],
\end{equation}
over
$\alpha\in L^{2}_{\text{prog}}([0,T]\times\Omega;\mathbb{R}^{p})$, which denotes the space of 
progressively measurable (with respect to the filtration $\mathbb{F}$) processes such that $\E\left[\int_0^T\left|\alpha_t\right|^{2}dt\right]<\infty$.
We will call these controls admissible. 
The dynamics of the $\mathbb{R}^{d}$-valued process $X^{\alpha}$ is given by the controlled stochastic differential equation (SDE)
\begin{align}\label{controlled_SDE}
dX_{t}^{\alpha}=b(t,X_{t}^{\alpha},\alpha_{t})dt+\sigma(t,X_{t}^{\alpha},\alpha_{t})dW_{t},\enspace t\in[0,T],\enspace X_0^{\alpha} = x.
\end{align}

\begin{assumption}\label{Assumptions_Old}
Let $(b,f,\sigma):[0,T]\times\mathbb{R}^{d}\times \mathbb{R}^{p}\rightarrow (\mathbb{R}^{d},\mathbb{R},\mathbb{R}^{d\times d'})$ satisfy the following conditions
\begin{enumerate}
\item For any $a\in \mathbb R^p$ the functions $b,\sigma$ and $f$ are jointly continuous in $(t,x)$.
\item For all $(t,x,a)\in[0,T]\times\mathbb{R}^{d}\times \mathbb{R}^{p}$, 
\begin{align*}
|D_{x}b(t,x,a)|+|D_{x}\sigma(t,x,a)|\leq K\text{ and }|b(t,x,a)|+|\sigma(t,x,a)|\leq K(1+|x|+|a|).
\end{align*}
\item For all $(t,x,a)\in[0,T]\times\mathbb{R}^d\times\mathbb{R}^p$
\begin{align*}
|f(t,x,a)|\leq K\left(1+|x|^2+|a|^{2} \right),\enspace |g(x)|\leq K(1+|x|^{2}).
\end{align*}
\end{enumerate}
\end{assumption}
This assumption and the mean value theorem imply that for any $x,x'\in\mathbb{R}^d$ and for any $(t,a)\in[0,T]\times\mathbb{R}^p$ we have
\begin{align*}
|b(t,x,a)-b(t,x',a)|+|\sigma(t,x,a)-\sigma(t,x',a)|\leq K|x-x'|.
\end{align*} 
\begin{theorem}\label{2nd-moments-controlled-SDEs}
Let Assumptions \ref{Assumptions_Old} hold. 
Then \eqref{controlled_SDE} has a unique solution and
there exists a constant $c\geq 0$ depending only on $K,d,d',p$ and $T$ such that
\begin{align*}
\sup_{\alpha\in L^2_{\text{prog}}\left([0,T]\times\Omega;\mathbb{R}^p\right)}\E\left[\sup_{0\leq t\leq T}\left|X_t^{\alpha}\right|^{2}\right]\leq c\left(1+|x|^2\right).
\end{align*}
\end{theorem}
The proof of this result is classical and can be found in \cite{krylov2008controlled}.
From Assumption \ref{Assumptions_Old} and Theorem \ref{2nd-moments-controlled-SDEs} we have that for each admissible control
\begin{align*}
|J(\alpha)|\leq&\E\left[\int_0^T|f(t,X_t^{\alpha},\alpha_t)|dt+|g(X_T^{\alpha})|dt\right]<\infty.
\end{align*}
Let the Hamiltonian $\mathcal{H}:[0,T]\times\mathbb{R}^{d}\times\mathbb{R}^{d}\times\mathbb{R}^{d\times d'}\times \mathbb{R}^{p}\rightarrow \mathbb{R}$ be given by
\begin{align}
\mathcal{H}(t,x,y,z,a)=b(t,x,a)\cdot y + \sigma(t,x,a):z + f(t,x,a).
\end{align}
For each admissible control $\alpha$, the adjoint equation is given by the following BSDE
\begin{align}\label{adjoint equation}
dY_{t}^{\alpha} =-D_{x}\mathcal{H}(t,X_{t}^{\alpha},Y_{t}^{\alpha},Z_{t}^{\alpha},\alpha_{t})dt + Z_{t}^{\alpha}dW_{t},\enspace t\in[0,T],\enspace Y_{T}^{\alpha}=D_{x}g(X_{T}^{\alpha}).
\end{align}
\begin{assumption}\label{Ass_2nd_derivative_sapce}
For all $(t,x,a)\in[0,T]\times\mathbb{R}^{d}\times \mathbb{R}^{p}$,
\begin{align*}
|D^2_{x}b(t,x,a)|+|D^2_{x}\sigma(t,x,a)|+|D^2_xf(t,x,a)|+|D_x^2g(x)|\leq& K.
\end{align*}
\end{assumption}
\begin{definition}[Solution to Adjoint Equation]\label{Solution_to_Adjoint}
An adapted  $\mathbb{R}^{d}\times\mathbb{R}^{d\times d'}$-valued process $\left(Y^{\alpha},Z^{\alpha}\right)$ such that 
\begin{align}\label{square_int_sol}
\E\left[\sup_{0\leq t\leq T }|Y_t^{\alpha}|^2+\int_0^T|Z_t^{\alpha}|^2dt\right]<\infty,
\end{align}
and such that it satisfies~\eqref{adjoint equation} is called a solution to~\eqref{adjoint equation}.
\end{definition}
If follows from Assumptions \ref{Ass_2nd_derivative_sapce} that for any $\alpha\in L^{2}_{\text{prog}}([0,T]\times\Omega;\mathbb{R}^p)$ the driver of the adjoint equation (i.e. the term $(t,y,z)\mapsto D_x\mathcal{H}(t,X_t^{\alpha},y,z,\alpha_{t})$) is Lipschitz in $(y,z)$ uniformly in $(t,\omega)$. 
The well-posedness of~\eqref{adjoint equation} is classical, see~\cite{pardoux1990adapted} or \cite[Theorem 4.3.1]{zhang2017backward}.

\subsection{Existence of Solutions to the the Gradient Flow system}
\label{sect:existence}
In this section we state precisely the main result concerning the existence and uniqueness of solutions to the gradient flow system \eqref{Gradient_flow_system_intro}
as well as the assumptions.
\begin{assumption}\label{additional_assumptions}
For all $(t,x,a) \in [0,T]\times \mathbb R^d \times \mathbb{R}^{p}$ we have
\begin{align*}
&|D_a b(t,x,a)| + |D_a \sigma(t,x,a)| \leq K\,,\\
&|D_x D_a b(t,x,a)|+|D_{x}D_{a}\sigma(t,x,a)|+|D_a D_x f(t,x,a)|\leq K\,.
\end{align*}
\end{assumption}
\begin{assumption}
\label{ass: D-xf and D_x g}
For any $(t,x,a)\in[0,T]\times\mathbb{R}^{d}\times \mathbb{R}^{p}$,
$
|D_{x}f(t,x,a)|+|D_{x}g(x)|\leq K.
$
\end{assumption}
\begin{assumption}\label{ass: one for Gateuax}
For any $(t,x,a)\in[0,T]\times\mathbb{R}^{d}\times\mathbb{R}^{p}$, 
\begin{align*}
|D_a f(t,x,a)|+|D_{x}f(t,x,a)|\leq K(1+|x|+|a|).
\end{align*}
\end{assumption}
\begin{assumption}
\label{assumption_for_uncontrolled_vol} 
For all $(t,x,a)\in[0,T]\times\mathbb{R}^{d}\times \mathbb{R}^{p}$,
\begin{align*}
|D_{a}^2\sigma(t,x,a)| + |D_a^2 b(t,x,a)|+|D_a^2 f(t,x,a)|\leq K.
\end{align*}
\end{assumption}
\begin{assumption}\label{ass:D_a^2sig}
For all $(t,x,a)\in[0,T]\times\mathbb{R}^{d}\times\mathbb{R}^{p}$ we have $D_a^2\sigma(t,x,a)=0$.
\end{assumption}
\begin{definition}[Solution to the Gradient Flow System \eqref{Gradient_flow_system_intro}]\label{Defn_solu_G_F}
The quadruple $(\alpha, X, Y, Z)$ is a solution to the gradient flow system \eqref{Gradient_flow_system_intro} if 
$\alpha\in AC\left(0,S;L^2_{\text{prog}}\left([0,T]\times\Omega;\mathbb{R}^p\right)\right)$,
 $X,Y,Z\in L^{2}\left(0,S;L^{2}_{\text{prog}}([0,T]\times\Omega;\mathbb{R}^{m})\right)$ with $m$ being $d$, $d$ and $d\times d'$ respectively and if
and for all $s\in [0,S]$, for all $t\in [0,T]$ and $\mathbb P$-almost surely,
\begin{align*}
X_{s,t} & = x+\int_0^tb(r,X_{s,r},\alpha_{s,r})dr+\int_0^t\sigma(r,X_{s,r},\alpha_{s,r})dW_{r}\,,\\
Y_{s,t} & = D_xg(X_{s,T})+\int_t^TD_x\mathcal{H}(r,X_{s,r},Y_{s,r},Z_{s,r},\alpha_{s,r})dr-\int_t^TZ_{s,r}dW_{r}
\end{align*}
and for almost every $s\in [0,S]$ we have
\begin{align}\label{eqn:GF}
\tfrac{d}{ds}\alpha_{s}=D_a\mathcal{H}(t,X_{s},Y_{s},Z_{s},\alpha_{s}),\enspace\alpha_{0}=\alpha^0
\end{align}
understood in the $L^2_{\text{prog}}\left([0,T]\times\Omega;\mathbb{R}^p\right)$ sense.
\end{definition}
\begin{theorem}[Existence and uniqueness of gradient flow system]
\label{thm:existence_GF}
Let Assumptions \ref{Assumptions_Old}, \ref{Ass_2nd_derivative_sapce}, \ref{additional_assumptions}, \ref{ass: D-xf and D_x g} , \ref{ass: one for Gateuax}, \ref{assumption_for_uncontrolled_vol} and \ref{ass:D_a^2sig} hold.
Let $S>0$.
Then there gradient flow system \eqref{Gradient_flow_system_intro} admits a unique solution in the sense of Definition \ref{Defn_solu_G_F}.
\end{theorem}

\subsection{Convergence of the Modified MSA to a Gradient Flow}\label{Intro: Convergence to GF} 
A key observation is that the cost functional decreases along the sequence of controls generated by iterates of the MMSA, see Remark~\ref{rmk:cost_functional_increment_bound} below which is itself a consequence of the following lemma.
\begin{lemma}\label{cost_functional_increment_bound}
Let Assumptions \ref{Assumptions_Old}, \ref{Ass_2nd_derivative_sapce}, \ref{additional_assumptions} and \ref{ass: D-xf and D_x g} be in hold.
Let $\varphi$ and $\theta$ be arbitrary admissible controls. Then there exists a constant $C_{J,\mathcal H}>0$ such that
\begin{align}\label{cost_functional_bound}
J(\varphi)-J(\theta)
\leq&\E\left[\int_{0}^{T}\mathcal{H}(t,X_t^{\theta},Y_{t}^{\theta},Z_{t}^{\theta},\varphi_{t})-\mathcal{H}(t,X_t^{\theta},Y_{t}^{\theta},Z_{t}^{\theta},\theta_{t})dt\right]\\
&+C_{J,\mathcal H}\E\left[\int_{0}^{T}|\varphi_{t}-\theta_t|^{2}dt\right]\nonumber.
\end{align}
\end{lemma}

This lemma was originally proved in \cite{kerimkulov2021modified} but under the assumption $D^2_x \sigma = 0$.
We will relax this following the BMO technique demonstrated in~\cite{ji2021modified}, see Section \ref{proof:cost_functional_increment_bound}. 

\begin{assumption}
\label{new lambda convexity}
There exists a constant $\lambda$ such that for any $(t,x,y,z)\in[0,T]\times\mathbb{R}^{d}\times\mathbb{R}^{d}\times\mathbb{R}^{d\times d'}$ the map
$
a \mapsto \mathcal{H}(t,x,y,z,a)+\tfrac{\lambda}{2}|a|^{2}    
$
is convex. 
\end{assumption}

The use of equality in~\eqref{gradient_requirement} in Algorithm~\ref{alg mmsa} is justified by Lemma~\ref{Well-defined-gradien-proposition} and Remark~\ref{Well-Posedness-of-gradient-condition} which shows that for $\tau > \lambda^{-1}$ the $\argmin$ in (\ref{gradient_requirement}) is single-valued.
Moreover, it follows from Lemma~\ref{cost_functional_increment_bound} that the sequence $(J(\alpha^n))_n$ is decreasing, see Remark \ref{rmk:cost_functional_increment_bound} below.
This justifies the stopping criteria for Algorithm~\ref{alg mmsa}. 

\begin{remark}\label{rmk:cost_functional_increment_bound}
Provided that $\frac{1}{2\tau}\geq C_{J,\mathcal{H}}$ we have $J(\alpha^{n+1})\leq J(\alpha^n)$ for any $n\in \mathbb N$.
Indeed, from Lemma \ref{cost_functional_increment_bound} with $\varphi=\alpha^{n+1}$ and $\theta=\alpha^{n}$ we have 
\begin{align*}
J(\alpha^{n+1})-J(\alpha^{n})\leq&\E\int_0^T \left[\mathcal{H}\left(t,X^{n}_{t},Y^{n}_{t},Z^{n}_{t},\alpha^{n+1}_t\right)-\mathcal{H}\left(t,X^{n}_{t},Y^{n}_{t},Z^{n}_{t},\alpha^{n}_t\right)\right]dt\\
&+C_{J,\mathcal{H}}\E\int_0^T\left|\alpha_t^{n+1}-\alpha_t^{n}\right|^{2}dt
\end{align*}
and from \eqref{gradient_requirement} we have 
\begin{align*}
\E\int_0^T\mathcal{H}\left(t,X^{n}_{t},Y^{n}_{t},Z^{n}_{t},\alpha^{n+1}_t\right)dt+\frac{1}{2\tau}\E\int_0^T|\alpha_t^{n+1}-\alpha_t^n|^2\,dt \\
\leq \E\int_0^T\mathcal{H}\left(t,X^{n}_{t},Y^{n}_{t},Z^{n}_{t},\alpha^{n}_t\right)\,dt\,.
\end{align*}
Hence $J(\alpha^{n+1})-J(\alpha^{n}) \leq 0$.
\end{remark}

Note that the first order optimality condition of calculus implies that~\eqref{gradient_requirement} is equivalent to the implicit Euler scheme
\begin{equation}\label{implicit_eulers}
\tfrac{\alpha_{t}^{n}-\alpha_{t}^{n-1}}{\tau}=-D_{a}\mathcal{H}(t,X_{t}^{n-1},Y_{t}^{n-1},Z_{t}^{n-1},\alpha_{t}^{n})\enspace\text{for each $(t,\omega)\in[0,T]\times\Omega$}. 
\end{equation}

Fix $S>0$. 
We will show that, after appropriate interpolation, the modified MSA algorithm converges to a gradient flow on $[0,S]$.
To that end we let $N_{\tau}\in \mathbb N$ and set $\tau = S / N_{\tau}$. 
The implicit Euler scheme~\eqref{implicit_eulers} leads to a sequence $\left(\alpha^{n}_{t}\right)_{n\in\mathbb{N}}$, 
where $\alpha^n \in L^{2}_{\text{prog}}([0,T]\times\Omega; \mathbb R^p)$ for each $n\in \mathbb \{1,\cdots,N_{\tau}\}$.
From these we define 
\begin{align}
\label{interp}
\begin{cases}
\hat{\alpha}^{\tau}_{s,t}=\alpha^{n-1}_t+\left(\frac{s-(n-1)\tau}{\tau}\right)\left(\alpha^{n}_t-\alpha^{n-1}_t\right)\,, \,\,\,t \in [0,T]\,, \,\,\, s\in ((n-1)\tau,n\tau]\,,\\
\alpha^{\tau,+}_{s,t}=\alpha^{n}_{t} \,, \,\,\, t \in [0,T]\,,\,\,\,t \in [0,T]\,, \,\,\, s\in ((n-1)\tau,n\tau]\\
\alpha^{\tau,-}_{s,t}=\alpha^{n-1}_{t} \,, \,\,\, t \in [0,T]\,,\,\,\,t \in [0,T]\,, \,\,\, s\in ((n-1)\tau,n\tau]
\end{cases}
\end{align}
and for $s=0$, we set $\hat{\alpha}^{\tau}_{0,t}=\alpha^{0}_{t}$, $\alpha^{\tau,+}_{0,t}=\alpha^{0}_{t}$ and $\alpha^{\tau,-}_{0,t}=\alpha^{0}_{t}$. 
Similarly, for the sequences $\left(X^{n}\right)_{n},\left(Y^{n}\right)_{n}$ and $\left( Z^{n}\right)_{n}$ we interpolate as follows.
For each $(t,\omega)\in[0,T]\times\Omega$ and $n\in\{1,\dots,N_{\tau}\}$ we define
\begin{align*}
X_{s,t}^{\tau,-}=X_{t}^{n-1},\enspace Y_{s,t}^{\tau,-}=Y_{t}^{n-1},\enspace Z_{s,t}^{\tau,-}=Z_t^{n-1}\enspace t \in [0,T]\,, \,\,\, s\in ((n-1)\tau,n\tau],
\end{align*}
and for $s=0$ we set $X_{0,t}^{\tau,-}=X_{t}^{0}$, $\enspace Y_{0,t}^{\tau,-}=Y_{t}^{0}$ and $\enspace Z_{s,t}^{\tau,-}=Z_t^0$.

We have introduced another time dimension: we now have $s\in [0,S]$ which denotes the gradient flow time.
There is still the time $t\in[0,T]$ corresponding to the original time in the stochastic control problem. 
For each $s \in [0,S]$ we have $\hat \alpha^\tau_s$, $\alpha^{\tau,+}_s$ and $\alpha^{\tau,-}_s$ in $L^{2}_{\text{prog}}([0,T]\times\Omega; \mathbb R^p)$.
That is, for each $s\in [0,S]$ we are getting an admissible open loop control for the original control problem.

The map $\hat{\alpha}^{\tau}:[0,S]\rightarrow L^{2}_{\text{prog}}([0,T]\times\Omega;\mathbb{R}^{p})\enspace:\enspace s\mapsto \hat{\alpha}_{s}^{\tau}$ is piecewise linear, continuous hence absolutely continuous, and therefore has a distributional derivative, satisfying the following: for $s\in((n-1)\tau,n\tau]$ and $t\in[0,T]$
\begin{equation}
\begin{split}
\label{discrete_derivative}
\tfrac{d}{ds}\hat{\alpha}_{s,t}^{\tau} & = \tfrac{\alpha_{t}^{n}-\alpha^{n-1}_{t}}{\tau}=-D_{a}\mathcal{H}(t,X_t^{n-1},Y_t^{n-1},Z_t^{n-1},\alpha_{t}^n)\\
& = -D_{a}\mathcal{H}(t,X_{s,t}^{\tau,-},Y_{s,t}^{\tau,-},Z_{s,t}^{\tau,-},\alpha_{s,t}^{\tau,+})
\end{split}
\end{equation}
Hence~\eqref{discrete_derivative} can be written as follows
\begin{align}\label{discrete_derivative_int}
\hat{\alpha}^\tau_{s,t}=\alpha^0_t-\int_0^sD_{a}\mathcal{H}(t,X_{r,t}^{\tau,-},Y_{r,t}^{\tau,-},Z_{r,t}^{\tau,-},\alpha_{r,t}^{\tau,+})dr
\end{align}
Our aim is to prove that the sequences $\left(\hat{\alpha}^{\tau}\right)_{\tau}$,$\left(\alpha^{\tau,+}\right)_{\tau}$, $\left(\alpha^{\tau,-}\right)_{\tau}$, $\left(X^{\tau,-}\right)_{\tau}$ $\left(Y^{\tau,-}\right)_{\tau}$ and $\left(Z^{\tau,-}\right)_{\tau}$ converge  as $\tau \to 0$ 
with rate $\tau$ to the solution of~\eqref{Gradient_flow_system_intro}.
To that end we need to impose the following assumption.
\begin{assumption}\label{compactness_of_orbits_assumption}
There exists a constant $K>0$ such that for any $t\in[0,T]$, $a\in \mathbb{R}^p$ and $x\in \mathbb{R}^{d}$, we have
$f(t,x,a)\geq -K + K^{-1}|a|^2$ and $g(x)\geq -K$.

\end{assumption}

\begin{theorem}[Convergence of Discrete Scheme]
\label{thm:convergence_rate}
Let Assumptions \ref{Assumptions_Old}, \ref{Ass_2nd_derivative_sapce}, \ref{additional_assumptions}, \ref{ass: D-xf and D_x g}, \ref{assumption_for_uncontrolled_vol}, \ref{new lambda convexity}, \ref{ass:D_a^2sig}, \ref{compactness_of_orbits_assumption}  hold. 
Let
$\alpha\in  AC\left(0,S;L^2_{\text{prog}}\left([0,T]\times\Omega;\mathbb{R}^p\right)\right)$ and $X, Y, Z$ in $L^{2}\left(0,S;L^{2}_{\text{prog}}([0,T]\times\Omega;\mathbb{R}^{m})\right)$ with $m=d$ for $X,Y$ and $m=d\times d'$ for $Z$
be the unique solution, in the sense of Definition \ref{Defn_solu_G_F}, to the gradient flow system \eqref{Gradient_flow_system_intro}.
Then
there exists a constant $C$ depending on 
$K$, $T$, $C_{\mathbb{H}^\infty,BMO}$, $C_{J,\mathcal{H}}$ and $\alpha^0$ such that
\begin{align*}
\left\|\hat{\alpha}^{\tau}-\alpha\right\|+\|\alpha^{\tau,+}-\alpha\|+\|\alpha^{\tau,-}-\alpha\|\leq CS^{\frac{1}{2}}e^{\frac{S}{2}(3+C+\lambda)}\tau,
\end{align*}
where $\|\cdot\|$ represents the supremum norm in $C\left(0,S;L^{2}_{\text{prog}}([0,T]\times\Omega;\mathbb{R}^{p})\right)$.
\end{theorem}

Convergence, with the same rate, of the sequences $\left(X^{\tau,-}\right)_\tau$, $\left(Y^{\tau,-}\right)_\tau$ and $\left(Z^{\tau,-}\right)_\tau$ is an immediate consequence of Theorem~\ref{thm:convergence_rate} and Lemma~\ref{lemma bsde stability}.

If the Hamiltonian is sufficiently strongly convex in $a$ uniformly in all other variables, that is, if Assumption~\ref{new lambda convexity} holds for some sufficiently negative $\lambda$, then the estimate is uniform $S$.
The proof of Theorem \ref{thm:convergence_rate} is given in Section \ref{Existence of the Gradient Flow}. 

\subsection{Convergence of the Gradient Flow}\label{Intro: convergence to OC}
The convergence we get depends on the convexity (or lack of) in the original control problem.
The energy identity given in Lemma \ref{functional_decreases_along_flow} will be fundamental to much of the analysis.
\begin{lemma}\label{functional_decreases_along_flow}
Assume that the quadruple $\left(\alpha_{s},X_{s},Y_{s},Z_{s}\right)$ is a solution to \eqref{Gradient_flow_system_intro}. Let  Assumptions \ref{Assumptions_Old}, \ref{additional_assumptions} and \ref{ass: one for Gateuax} hold and assume $\alpha^{0}\in L^{2}_{\text{prog}}([0,T]\times\Omega;\mathbb{R}^{p})$. 
Then for any $s\geq 0$
\begin{align}\label{derivative_along_flow}
\tfrac{d}{ds}J\left(\alpha_{s}\right)=-\left\|D_a\mathcal{H}(t,X_{s},Y_{s},Z_{s},\alpha_{s})\right\|_{L^{2}_{\text{prog}}\left([0,T]\times\Omega;\mathbb{R}^p\right)}^2\leq 0.
\end{align}
\end{lemma}
This Lemma will be proved in  Section \ref{Sect: convergence to OC}.
Whilst the proof of Lemma \ref{functional_decreases_along_flow} is deferred until Section \ref{Sect: convergence to OC} we provide a heuristic proof (liberally swapping limits and integrals). Assuming for the moment the functional is Gateaux differentiable then 
\begin{equation*}
\begin{split}
\tfrac{d}{ds}J(\alpha_{s}) 
& = \lim_{h\rightarrow 0}\tfrac{J(\alpha_{s+h})-J(\alpha_{s})}{h}\\
&= \lim_{h\rightarrow 0}\tfrac{1}{h}\int_0^1\tfrac{d}{d\delta}J(\alpha_{s+h}+\varepsilon(\alpha_{s+h}-\alpha_{s})+\delta(\alpha_{s+h}-\alpha_{s}))|_{\delta=0}d\varepsilon\,.    
\end{split}
\end{equation*}
If we now recognise the integrand as the Gateaux derivative at the point $\alpha_{s+h}+\varepsilon(\alpha_{s+h}-\alpha_{s})$ in the direction $\alpha_{s+h}-\alpha_{s}$ it follows (from Lemma \ref{Gateaux_Derivative}) that 
\begin{align*}
& \tfrac{d}{ds}J(\alpha_{s}) \\
& =\int_0^{1} \E\biggl[\int_{0}^{T} \!\!\lim_{h\rightarrow 0}D_{a}\mathcal{H}\left(t,X_{t}^{\varepsilon,h},Y_{t}^{\varepsilon,h},Z_{t}^{\varepsilon,h},\alpha_{s,t}+\varepsilon(\alpha_{s+h,t}-\alpha_{s,t})\right) 
\lim_{h\rightarrow 0}\tfrac{\left(\alpha_{s+h,t}-\alpha_{s,t}\right)}{h}dt\biggr]d\varepsilon,
\end{align*}
where 
\begin{align*}
X^{\varepsilon,h}_{t}:=X_{t}^{\alpha_{s,t}+\varepsilon\left(\alpha_{s+h,t}-\alpha_{s,t}\right)},\enspace Y^{\varepsilon,h}_{s,t}:=Y_{t}^{\alpha_{s,t}+\varepsilon\left(\alpha_{s+h,t}-\alpha_{s,t}\right)}\enspace Z^{\varepsilon,h}_{s,t}:=Z_{t}^{\alpha_{s,t}+\varepsilon\left(\alpha_{s+h,t}-\alpha_{s,t}\right)}.
\end{align*}
Assuming we can prove the required convergence for $D_a\mathcal{H}$ and substituting the expression for the gradient flow in \eqref{Gradient_flow_system_intro} the conclusion of Lemma~\ref{functional_decreases_along_flow} follows.
\begin{theorem}
\label{thm: Extending gradient flow}
Let $(\alpha,X,Y,Z)$ be a solution to \eqref{Gradient_flow_system_intro} on the interval $[0,S]$. 
Under Assumptions~\ref{Assumptions_Old}, \ref{additional_assumptions},~\ref{ass: one for Gateuax} and \ref{compactness_of_orbits_assumption}
we have that 
\begin{align*}
\sup_{s\in[0,S]} \|\alpha_s\|_{L^{2}_{\text{prog}}\left([0,T]\times\Omega;\mathbb{R}^p\right)}^2 \leq K(J(\alpha^0)+K(T+1)).
\end{align*}
Note that the constant on the right hand side is independent of $S$. 
Moreover, the solution of~\eqref{Gradient_flow_system_intro} can be extended to all $s\in [0,\infty)$.
\end{theorem}

\begin{theorem}
\label{prop: L^2-convergence of D_aH}
Let $(\alpha,X,Y,Z)$ be a solution to \eqref{Gradient_flow_system_intro} on $[0,\infty)$. 
If Assumptions \ref{Assumptions_Old}, \ref{additional_assumptions}, \ref{ass: one for Gateuax} and \ref{ass:D_a^2sig} hold. 
Let $\alpha^{0}\in L^{2}_{\text{prog}}([0,T]\times\Omega;\mathbb{R}^{p})$
then 
\begin{align*}
\lim_{s\rightarrow 0} \left\|D_{a}\mathcal{H}(X_s,Y_s,Z_s,\alpha_s)\right\|_{L^{2}_{\text{prog}}\left([0,T]\times\Omega;\mathbb{R}^p\right)}^2=0.
\end{align*}
\end{theorem}
\begin{theorem}
\label{thm: Convergence_of_value_fiunction}
Let $(\alpha_{s},X_{s},Y_{s},Z_{s})$ be a solution to \eqref{Gradient_flow_system_intro} on $[0,\infty)$. 
Let Assumptions \ref{Assumptions_Old}, \ref{additional_assumptions} and \ref{ass: one for Gateuax} hold and assume $\alpha^{0}\in L^{2}_{\text{prog}}([0,T]\times\Omega;\mathbb{R}^{p})$. Also assume that for any $(t,y,z)\in[0,T] \times \mathbb{R}^{d} \times \mathbb{R}^{d\times d'}$ the maps $
(x,a)\mapsto\mathcal{H}(t,x,y,z,a)$
and $x\mapsto g(x)$ are convex. 
Fix an arbitrary $\alpha^*\in \argmin_{\alpha\in L^{2}_{\text{prog}}([0,T]\times\Omega;\mathbb{R}^{p})}J(\alpha)$.
Then for any $S>0$
\begin{align*}
J(\alpha_S) - J(\alpha^*) \leq \frac1{S}\|\alpha_{0}-\alpha^{*}\|^2\,.
\end{align*}
If, additionally, there is $\eta>0$ such that for any $(t,y,z)\in[0,T]\times\mathbb{R}^d\times\mathbb{R}^{d\times d'}$ the map $(x,a)\mapsto\mathcal{H}(t,x,y,z,a) -\frac\eta 2|a|^2$ is convex, then for any $S>0$
\begin{align*}
J(\alpha_S)-J(\alpha^*)\leq\eta e^{-\eta S}\|\alpha^0-\alpha^*\|^2\text{ and }\|\alpha_S-\alpha^*\|^2\leq e^{-\eta S}\|\alpha^0-\alpha^*\|^2,
\end{align*}
where $\|\cdot\|=\|\cdot\|_{L^2_{\text{prog}}\left([0,T]\times\Omega;\mathbb{R}^p\right)}$.
\end{theorem}
Lemma~\ref{functional_decreases_along_flow} and Theorems~\ref{thm: Extending gradient flow}, \ref{prop: L^2-convergence of D_aH}, \ref{thm: Convergence_of_value_fiunction} will be proved in Section~\ref{Sect: convergence to OC}.

\section{Proof for Existence of Solutions to the Gradient Flow System}
\label{sect:existence of solut}

Existence and uniqueness will be shown by utilising Banach's Fixed Point Theorem. 
We first show the map $\varphi\mapsto D_a \mathcal H(X^\varphi, Y^\varphi, Z^\varphi, \varphi)$ is Lipschitz continuous, see Lemma \ref{lemma L2 cont of DaH}. 
The proof of Lemma \ref{lemma L2 cont of DaH} relies on the theory of BMO martingales
which allow us to derive a 
uniform in $\alpha$ bound for the processes $Y^\alpha$ and $Z^{\alpha}$ with respect to appropriate norms, see Lemma~\ref{technical_result_for_BSDEs}.

We now introduce the following notation.
For $1\leq j\leq\infty$, define
\begin{align}\label{notation_space_1}
\mathbb{H}_{k}^{j}:= L^{j}_{\text{prog}}\left([0,T]\times\Omega;\mathbb{R}^{k}\right).
\end{align}
In particular the space of admissible controls $L^{2}_{\text{prog}}\left([0,T]\times\Omega;\mathbb{R}^{p}\right)$ is now denoted $\mathbb{H}^{2}_{p}$, and for $X\in\mathbb{H}_{k}^{j}$, we define,
\begin{align}\label{notation_norm}
\left\|X\right\|_{\mathbb{H}_{k}^{j}}=\begin{cases}
\left(\E\left[\int_{0}^{T}\left|X_{t}\right|^{j}dt\right]\right)^{\frac{1}{j}}&\text{ if }1\leq j<\infty\\
\esssup_{(t,\omega)\in[0,T]\times\Omega}\left|X_{t}\right|&\text{otherwise}.
\end{cases}.
\end{align}

Let $M$ be an $\mathbb R$-valued stochastic process. 
We shall write $\langle M\rangle$ to denote its quadratic variation.
We will say $M\in\mathcal{H}_{\mathbb{F}}^{2}([0,T];\mathbb{R})$, if: $M$ is progressively measurable with respect to the filtration $\left(\mathcal{F}_t\right)_{t\geq 0}$,
$M$ has continuous samples paths, $M_{0}=0$ and $\|M\|_{\mathcal{H}^{2}_{\mathbb{F}}}:=\|\sqrt{\langle M\rangle_{T}}\|_{L^{2}}=\left(\E\left[\langle M\rangle_T\right]\right)^{\frac{1}{2}}<\infty$.
\begin{definition}
\label{def bmo}
With the supremum is taken over all stopping times $\tau\leq T$ let
\begin{align*}
\text{BMO}=\left\{M\in\mathcal{H}_{\mathcal{F}}^{2}([0,T];\mathbb{R}): \|M\|_{\text{BMO}}:=\sup_{\tau}\left\|\left(\E\left[\langle M\rangle_{T}-\langle M\rangle_{\tau}|\mathcal{F}_{\tau}\right]\right)^{\frac{1}{2}}\right\|_{\infty}<\infty\right\}\,.
\end{align*}
\end{definition}
For each admissible control $\alpha$ the solution to the adjoint equation $\left(Y^{\alpha},Z^{\alpha}\right)\in\mathbb{H}^{2}_{d}\times\mathbb{H}^{2}_{d\times d'}$. 
For $Z\in\mathbb{H}^{2}_{d\times d'}$ we say $Z\in\mathcal{K}(\mathbb{R}^{d\times d'})$ if   
\begin{align*}
\|Z\|_{\mathcal{K}}:=\sup_{\tau}\bigg\|\E\bigg[\int_{\tau}^{T}|Z_{t}|^{2}dt\bigg|\mathcal{F}_{\tau}\bigg]^{\frac{1}{2}}\bigg\|_{\infty}<\infty,
\end{align*} 
despite all finite dimensional norms being equivalent, it will be useful to set $
|Z|=(Z:Z)^{\frac{1}{2}}$.
The following lemma will be proved in Section~\ref{proof: technical_result_for_BSDEs}.
\begin{lemma}\label{technical_result_for_BSDEs}
Let $\left(Z^{\alpha}\right)^{ij}$ be the $(i,j)$-th component of the matrix $Z^{\alpha}$ and let $\big(\left(Z^{\alpha}\right)^{ij}\circ W^{j}\big)_t := \int_0^t \left(Z^{\alpha}\right)^{ij}_{t'}dW^{j}_{t'}$ for all $t\in [0,T]$.
Let Assumptions \ref{Assumptions_Old} and \ref{ass: D-xf and D_x g} hold and let $\alpha$ be an admissible control. Assume the pair $\left(Y^{\alpha},Z^{\alpha}\right)$ solves the adjoint equation (\ref{adjoint equation}). 
Then there is a constant $C_{\mathbb H^\infty,\,\text{BMO}}>0$ such that 
\begin{align*}
\sup_{\alpha}\|Y^{\alpha}\|_{\mathbb{H}^{\infty}_d} + \sup_{\alpha}\|Z^{\alpha}\|_{\mathcal{K}}+\sup_{1 \leq i\leq d } \sup_{1\leq j\leq d'}\sup_{\alpha}\|\left(Z^{\alpha}\right)^{ij}\circ W^{j}\|_{\text{BMO}}\leq  C_{\mathbb H^\infty,\,\text{BMO}}\,.
\end{align*}
\end{lemma}

To proceed we will need the following version of Fefferman's inequality which is proved in~\cite[Ch. 2.4]{kazamaki2006continuous}. 
\begin{theorem}\textbf{Fefferman's Inequality}\label{fefferman_inq}
Let $N$ be  a real valued $\mathbb{F}$-martingale such that $\E\left[\sqrt{\langle N\rangle}_{T}\right]<\infty$, and $M\in \text{BMO}$. Then 
\begin{align*}
\E\bigg[\int_{0}^{T}|d\langle M,N\rangle_{s}|\bigg]\leq\sqrt{2}\|M\|_{\text{BMO}}\E\left[|\langle N\rangle_{T}|^{\frac{1}{2}}\right]\,.
\end{align*}
\end{theorem}
\begin{lemma}\label{lem:fefferman ext}
Let $X$ be a continuous adapted process such that $\E\left[\sup_{0\leq t < T}|X_{t}|^{2}\right]<\infty$ and 
let $M$ be a BMO martingale. 
Then 
\begin{align*}
\E\left[\left\langle X\circ M \right\rangle_{T}\right]^{\frac{1}{2}}\leq\sqrt{2}\E\left[\sup_{0\leq t<T}|X_t|^{2}\right]^{\frac{1}{2}}\|M\|_{\text{BMO}}.
\end{align*}
\end{lemma}
\begin{proof}
The proof of this results follows Lemma 1.4 in \cite{delbaen2010harmonic}. 
If $\E\left[\langle X\circ M\rangle_{T}\right]^{\frac{1}{2}}=0$ then the result is trivially true and we proceed assuming this term is not zero.
From Fefferman's inequality we get
\[
\begin{split}
\E\left[\left\langle X\circ M\right\rangle_{T}\right]
& =\E\left[\int_0^TX_t^{2}d\langle M\rangle_{t}\right]=\E\left[\left\langle\int_0^\cdot X_t^2dM,M\right\rangle_{T}\right]\\
& \leq \sqrt{2}\E\left[\sqrt{\left\langle\int_0^\cdot X_t^2dM\right\rangle_{T}}\right]\|M\|_{\text{BMO}}\,.
\end{split}
\]
Hence 
\[
\begin{split}
\E\left[\left\langle X\circ M\right\rangle_{T}\right]
& \leq \sqrt{2}\E\left[\sqrt{\int_0^{T}X_t^{4}d\langle M\rangle_t}\right]\|M\|_{\text{BMO}}\\
& \leq\sqrt{2}\E\left[\sqrt{\sup_{0\leq t<T}|X_t|^2}\sqrt{\int_0^TX_t^2d\langle M\rangle_t}\right]\|M\|_{\text{BMO}}\\
& \leq \sqrt{2}\left(\E\left[\sup_{0\leq t<T}|X_t|^{2}\right]\right)^{\frac{1}{2}}\left(\E\left[\int_0^T X_t^2d\langle M\rangle_t\right]\right)^{\frac{1}{2}}\|M\|_{\text{BMO}}\,,
\end{split}
\]
where the last step employed H\"older's inequality.
Thus, finally, we get 
\[
\E\left[\left\langle X\circ M\right\rangle_{T}\right] 
\leq \sqrt{2}\left(\E\left[\sup_{0\leq t<T}|X_t|^{2}\right]\right)^{\frac{1}{2}}\E\left[\langle X\circ M\rangle_{T}\right]^{\frac{1}{2}}\|M\|_{\text{BMO}}\,.
\]
The required result follows from dividing by $\E\left[\langle X\circ M\rangle_{T}\right]^{\frac{1}{2}}$. 
\end{proof}
\begin{lemma}\label{lemma bsde stability}
Under Assumptions \ref{Assumptions_Old} and \ref{additional_assumptions}  it follows that there exists $c>0$ which depends only on $K,T,d$ and $d'$ such that if $\alpha, \alpha' \in L^2_\text{Prog}([0,T]\times\Omega; \mathbb R^p)$ then 
\begin{equation}\label{forward_bound}
\mathbb E \sup_{t\in [0,T]} |X_t^\alpha - X_t^{\alpha'}|^2 \leq c\|\alpha-\alpha'\|_{L^2_\text{Prog}([0,T]\times\Omega; \mathbb{R}^p)}^2.
\end{equation}
If, additionally, Assumptions \ref{Ass_2nd_derivative_sapce} and \ref{ass: D-xf and D_x g} hold then
\begin{equation}\label{backward_bounds}
\mathbb E \sup_{t\in [0,T]} |Y_t^\alpha - Y_t^{\alpha'}|^2 + \|Z^\alpha - Z^{\alpha'}\|_{L^2_\text{Prog}([0,T]\times\Omega; \mathbb{R}^p)}^2 \leq c\|\alpha-\alpha'\|_{L^2_\text{Prog}([0,T]\times\Omega; \mathbb{R}^p)}^2.
\end{equation}
\end{lemma}
\begin{proof}
See Section \ref{proof:lemma bsde stability}.
\end{proof}
\begin{lemma}
\label{lemma L2 cont of DaH}
Let Assumptions \ref{Assumptions_Old}, \ref{Ass_2nd_derivative_sapce}, \ref{additional_assumptions}, \ref{ass: D-xf and D_x g}, \ref{assumption_for_uncontrolled_vol} and \ref{ass:D_a^2sig} hold.
Then there exists a constant $C>0$ such that for arbitrary admissible controls $\varphi$ and $\theta$ 
\[
\|D_a \mathcal H(X^\varphi, Y^\varphi, Z^\varphi, \varphi) 
- D_a \mathcal H(X^\theta, Y^\theta, Z^\theta, \theta)\|_{L^2_{\text{prog}}([0,T]\times \Omega; \mathbb R^p)}^2 \leq C \|\theta-\varphi\|_{L^{2}_{\text{prog}}([0,T]\times\Omega;\mathbb{R}^p)}^2\,.
\]
\end{lemma}
\begin{proof}
First note that due to Assumptions~\ref{Assumptions_Old},~\ref{additional_assumptions} the fact $|D^2_a\sigma|=0$ and Lemma~\ref{lemma bsde stability} we have
\begin{align*}
\E\left[\sup_{0\leq t<T}|D_{a}\sigma(t,X_{t}^{\varphi},\varphi_t)-D_{a}\sigma(t,X_{t}^{\theta},\theta_{t})|^2\right]\leq&K\E\left[\sup_{0\leq t<T}|X_{t}^{\theta}-X_{t}^{\varphi}|^2\right]\\ 
\leq& cK\|\theta-\varphi\|_{L^{2}_{\text{prog}}([0,T]\times\Omega;\mathbb{R}^p)}\,.
\end{align*}
For any $j=1,\ldots,d'$ let $M = |Z^{\theta}| \circ W^{j}$.
Let $\mathcal X_t := D_{a}\sigma(t,X_{t}^{\varphi},\varphi)-D_{a}\sigma(t,X_{t}^{\theta},\theta_{t})$.
Due to Lemma~\ref{technical_result_for_BSDEs} we know that $M$ is a BMO martingale
 and hence due to Lemma~\ref{lem:fefferman ext} we have
\begin{equation}
\begin{split}
\label{eq continuity of DaH step}
\E\int_0^T|\mathcal X_t|^2 |Z_t^{\theta}|^2\,dt
& = \E\left\langle |\mathcal X_t|\circ M\right\rangle_T
\leq  2\|M\|_{\text{BMO}}^2\E\left[\sup_{0\leq t\leq T}\left|\mathcal X_t\right|^{2}\right]\\
& \leq 2K C_{\mathbb H^\infty,\,\text{BMO}}^2\E\left[\sup_{0\leq t\leq T}|X_t^{\theta}-X^{\varphi}_t|^2\right]\,.
\end{split}
\end{equation}
Moreover
\begin{align*}
\E&\left[\int_0^{T}\left|D_a\mathcal{H}(t,X_{t}^{\varphi},Y_{t}^{\varphi},Z_{t}^{\varphi},\varphi_t)-D_a\mathcal{H}(t,X_{t}^{\theta},Y_{t}^{\theta},Z_{t}^{\theta},\theta_t)\right|^{2}\right]\\
\leq&\E\left[\int_{0}^{T}|D_{a}b(t,X_{t}^{\varphi},\varphi)\cdot(Y_{t}^{\varphi}-Y_{t}^{\theta})|^2+|(D_{a}b(t,X_{t}^{\varphi},\varphi)-D_{a}b(t,X_{t}^{\theta},\theta_{t}))\cdot Y_{t}^{\theta}|^2dt\right]\\
&+\E\left[\int_{0}^{T}|D_{a}\sigma(t,X_{t}^{\varphi},\varphi):(Z_{t}^{\varphi}-Z_{t}^{\theta})|^2+|(D_{a}\sigma(t,X_{t}^{\varphi},\varphi)-D_{a}\sigma(t,X_{t}^{\theta},\theta_{t})): Z_{t}^{\theta}|^2dt\right]\\
&+K\E\left[\int_{0}^{T}\left|X_{t}^{\varphi}-X_{t}^{\theta}\right|^2dt\right]+K\E\left[\int_{0}^{T}\left|\varphi_{t}-\theta_{t}\right|^2dt\right]\\
\leq&C_{K,C_{\mathbb H^\infty,\,\text{BMO}}}\E\left[\int_0^T|\varphi_t-\theta_t|^2dt\right]+\E\left[\int_0^T|D_{a}\sigma(t,X_{t}^{\varphi},\varphi)-D_{a}\sigma(t,X_{t}^{\theta},\theta_{t})|^2|Z_t^{\theta}|^2dt\right].
\end{align*}
This together with~\eqref{eq continuity of DaH step} concludes the proof.
\end{proof}
\begin{proof}[Proof of Theorem \ref{thm:existence_GF}]
From Lemma \ref{lemma L2 cont of DaH} 
we know that the map
\[
L^2_{\text{prog}}\left([0,T]\times\Omega;\mathbb{R}^p\right) \ni \alpha \mapsto D_a \mathcal{H}(X^\alpha, Y^\alpha, Z^\alpha, \alpha) \in L^2_{\text{prog}}\left([0,T]\times\Omega;\mathbb{R}^p\right)\
\]
is Lipschitz continuous.
Define 
\[
C(0,S; L^2_\text{prog}\left([0,T]\times\Omega;\mathbb{R}^p\right) ) \ni \alpha \mapsto \Psi(\alpha) \in C(0,S; L^2_\text{prog}\left([0,T]\times\Omega;\mathbb{R}^p\right) ) 
\]
by 
\[
\Psi(\alpha)_s = \alpha^{0} - \int_0^s D_a \mathcal{H}(X^{\alpha_r}, Y^{\alpha_r}, Z^{\alpha_r}, {\alpha_r})\,dr\,,
\]
where, for each $r \in [0,S]$, we have $X^{\alpha_r}, Y^{\alpha_r}, Z^{\alpha_r}$ corresponding to the unique solutions of the forward~\eqref{controlled_SDE_intro} and backward~\eqref{adjoint equation_intro} equations respectively, with the control $\alpha_r \in L^2_{\text{prog}}\left([0,T]\times\Omega;\mathbb{R}^p\right)$.

We will show that $\Psi$ composed with itself sufficiently many times is a contraction in the Banach space $C(0,S; L^2_\text{prog}([0,T]\times\Omega;\mathbb{R}^p))$. 
To that end consider $\alpha, \alpha' \in C(0,S; L^2_\text{prog}\left([0,T]\times\Omega;\mathbb{R}^p\right) )$.
Note that
\[
\begin{split}
\Psi(\alpha)_s - \Psi(\alpha')_s =  \int_0^s -D_a \mathcal{H}(X^{\alpha_r}, Y^{\alpha_r}, Z^{\alpha_r}, {\alpha_r}) + D_a \mathcal{H}(X^{\alpha'_r}, Y^{\alpha'_r}, Z^{\alpha'_r}, {\alpha'_r})\,dr\,.
\end{split}
\]
Squaring this and using H\"older's inequality leads to 
\[
\begin{split}
|\Psi(\alpha)_s - \Psi(\alpha')_s|^2 \leq s \int_0^s |D_a \mathcal{H}(X^{\alpha_r}, Y^{\alpha_r}, Z^{\alpha_r}, {\alpha_r}) - D_a \mathcal{H}(X^{\alpha'_r}, Y^{\alpha'_r}, Z^{\alpha'_r}, {\alpha'_r})|^2\,dr\,.
\end{split}
\]
Integrating over $\Omega \times [0,T]$ and using Fubini's theorem we see that
\begin{align*}
\|\Psi(\alpha)_s &- \Psi(\alpha')_s\|_{L^2_\text{prog}([0,T]\times\Omega;\mathbb{R}^p)}^2 \\
&\leq S \int_0^s \|D_a \mathcal{H}(X^{\alpha_r}, Y^{\alpha_r}, Z^{\alpha_r}, {\alpha_r}) - D_a \mathcal{H}(X^{\alpha'_r}, Y^{\alpha'_r}, Z^{\alpha'_r}, {\alpha'_r})\|_{L^2_\text{prog}\left([0,T]\times\Omega;\mathbb{R}^p\right)}^2\,dr\,.
\end{align*}
Using Lemma~\ref{lemma L2 cont of DaH} we obtain, for all $s\in [0,S]$, that
\begin{align}
\label{contraction_1}
\|\Psi(\alpha)_s - \Psi(\alpha')_s\|_{L^2_\text{prog}([0,T]\times\Omega;\mathbb{R}^p)}^2 \leq c s \int_0^s \|\alpha_r - \alpha'_r\|_{L^2_\text{prog}([0,T]\times\Omega;\mathbb{R}^p)}^2\,dr\,.
\end{align}
So for any  $s\in [0,S]$ we have that
\[
\begin{split}
& \|\Psi(\Psi(\alpha))_s - \Psi(\Psi(\alpha'))_s\|_{L^2_\text{prog}([0,T]\times\Omega;\mathbb{R}^p)}^2 
\leq c s \int_0^s \|\Psi(\alpha)_{s_1} - \Psi(\alpha')_{s_1}\|_{L^2_\text{prog}([0,T]\times\Omega;\mathbb{R}^p)}^2\,ds_1 \\
& \leq (c s)^2 \int_0^s \int_0^{s_1} \|\alpha_{s_2} - \alpha'_{s_2}\|_{L^2_\text{prog}([0,T]\times\Omega;\mathbb{R}^p)}^2\,ds_2\,ds_1\,.
\end{split}
\]
Let $\Psi^n$ denote the composition of $\Psi$ with itself $n$ times. 
Then
\[
\begin{split}
\|\Psi^n(\alpha)_s & - \Psi^n(\alpha')_s\|_{L^2_\text{prog}([0,T]\times\Omega;\mathbb{R}^p)}^2\\ 
& \leq (cs)^n \int_0^s \int_0^{s_1} \cdots \int_0^{s_{n-1}} \|\alpha_{s_n} - \alpha'_{s_n}\|_{L^2_\text{prog}([0,T]\times\Omega;\mathbb{R}^p)}^2 \,ds_n \cdots ds_2 \, ds_1\,.\\
& \leq (cs)^n \sup_{r \in [0,s]} \|\alpha_{r} - \alpha'_{r}\|_{L^2_\text{prog}([0,T]\times\Omega;\mathbb R^p)}^2  \int_0^s \int_0^{s_1} \cdots \int_0^{s_{n-1}} \,ds_n \cdots ds_2 \, ds_1\\
& = \frac{(cs)^n s^n}{n!} \sup_{r \in [0,s]} \|\alpha_{r} - \alpha'_{r}\|_{L^2_\text{prog}([0,T]\times\Omega;\mathbb R^p)}^2\,.
\end{split}
\]
Hence, taking first $s=S$ on the right-hand side (which preserves the inequality), then taking the supremum in $s\in[0,S]$ on the left-hand side and finally taking square root of both sides, we obtain
\[
\sup_{s\in[0,S]} \|\Psi^n(\alpha)_s - \Psi^n(\alpha')_s\|_{L^2_\text{prog}([0,T]\times\Omega;\mathbb{R}^p)}
\leq
\bigg(\frac{c^n S^{2n}}{n!}\bigg)^{1/2}\sup_{s \in [0,S]} \|\alpha_{s} - \alpha'_{s}\|_{L^2_\text{prog}([0,T]\times\Omega;\mathbb R^p)}\,.
\]
Taking $n\in \mathbb N$ large enough shows the map is a contraction on $C(0,S; L^2_\text{prog}([0,T]\times\Omega;\mathbb{R}^p))$ which is a Banach space. 
Hence, by Banach's fixed point theorem there is a unique fixed point $\alpha \in C(0,S; L^2_\text{prog}([0,T]\times\Omega;\mathbb{R}^p))$
which satisfies for all $s\in [0,S]$ 
\[
\alpha_s = \alpha^{0} - \int_0^s D_a \mathcal H(X^{\alpha_r}, Y^{\alpha_r}, Z^{\alpha_r}, {\alpha_r})\,dr\,.
\]
We thus see that $s\mapsto \alpha_s \in L^2_\text{prog}([0,T]\times\Omega;\mathbb{R}^p))$ is absolutely continuous. This implies that $\frac{d}{ds}\alpha_s$ exists almost everywhere in $[0,S]$ and
\[
\tfrac{d}{ds}\alpha_s = -D_a \mathcal H(X^{\alpha_s}, Y^{\alpha_s}, Z^{\alpha_s}, {\alpha_s})\,\,\, \text{a.e. in} \,\,\, [0,S]\,, \alpha_0 = \alpha^0\,.
\]
Thus the proof is complete.
\end{proof}

\section{Proof of Convergence of Modified MSA to a Gradient Flow}\label{Existence of the Gradient Flow}

In order to prove the convergence of the Modified MSA to a gradient flow we need to ensure that (\ref{gradient_requirement}) is well defined, i.e. the $\argmin$ is single valued. 
\begin{lemma}
\label{Well-defined-gradien-proposition}
Let Assumption \ref{new lambda convexity} hold. Then 
for any $a'\in \mathbb{R}^{p}$ and for any $(t,x,y,z)$ the map 
\begin{align*}
a\mapsto\mathcal{H}(t,x,y,z,a)+\frac{\lambda}{2}\left|a-a'\right|^{2}
\end{align*}
is convex.
\end{lemma}
Lemma~\ref{Well-defined-gradien-proposition} will be proved in Section \ref{Well-Posedness-of-gradient-condition-proof}.
\begin{remark}\label{Well-Posedness-of-gradient-condition}\textbf{Well-Posedness of (\ref{gradient_requirement}):}
In order for equation (\ref{gradient_requirement}) to be well defined we need to make sure that the constant $\tau>0$ can be chosen independently of $(t,\omega)\in[0,T]\times\Omega$. 
To see why this is the case note that for any $(t,\omega)\in[0,T]\times\Omega$ we can pick a single $\lambda$ (thanks to Assumption \ref{new lambda convexity} and Lemma \ref{Well-defined-gradien-proposition}) such that the map 
\begin{align*}
a\mapsto\mathcal{H}\left(t,X_{t}^{n},Y_{t}^{n},Z_{t}^{n},a\right)+\tfrac{1}{2}\lambda\left|a-\alpha^{n}_{t}\right|,
\end{align*}
is convex. With $\tau<\lambda^{-1}$ we have strict convexity.
\end{remark}
Next we will show that the sequences 
$\left(\hat{\alpha}^{\tau}\right)_{\tau}$, $(\alpha^{\tau,+})_{\tau}$ and $(\alpha^{\tau,-})_{\tau}$ are bounded in $C\left([0,S];L^{2}_{\text{prog}}\left([0,T]\times\Omega;\mathbb{R}^{p}\right)\right)$ uniformly in $\tau$.

\begin{lemma}\label{prori estimate 2}
Let Assumptions~\ref{Assumptions_Old}, \ref{Ass_2nd_derivative_sapce}, \ref{additional_assumptions}, \ref{ass: D-xf and D_x g} and \ref{compactness_of_orbits_assumption} hold. Then there exists a constant $C>0$ which depends on $K,T$ and $J(\alpha^0)$ only  such that $\forall \tau>0$ and $\forall s\in[0,S]$
\begin{align*}
\|\hat{\alpha}^{\tau}_{s}\|_{L^{2}_{\text{prog}}\left([0,T]\times\Omega;\mathbb{R}^{p}\right)}+\|\alpha^{\tau,+}_{s}\|_{L^{2}_{\text{prog}}\left([0,T]\times\Omega;\mathbb{R}^{p}\right)}+\|\alpha^{\tau,-}_{s}\|_{L^{2}_{\text{prog}}\left([0,T]\times\Omega;\mathbb{R}^{p}\right)}\leq C.
\end{align*}
\end{lemma}
\begin{proof}
For each $\tau>0$, we know from Lemma \ref{cost_functional_increment_bound} and Remark \ref{rmk:cost_functional_increment_bound} that $J(\alpha^{\tau,-}_{s})$ decreases in $s$. 
Therefore, 
\begin{align*}
J(\alpha^{0})\geq J(\alpha_{s}^{\tau,-})=\E\left[\int_0^{T}f\left(t,X_{s,t}^{\tau,-},\alpha^{\tau,-}_{s,t}\right)dt+g(X_{s,t}^{\tau,-})\right],
\end{align*}
where $X_{s}^{\tau,-}$ is the solution of the controlled SDE corresponding to the control $\alpha^{\tau,-}_{s}$. Now using Assumption \ref{compactness_of_orbits_assumption} we have 
\begin{align*}
J(\alpha^{0})\geq\E\left[\int_0^{T}-K+K^{-1}\left|\alpha_{s,t}^{\tau,-}\right|^{2}dt-K\right]=-KT-K+K^{-1}\E\left[\int_0^T\left|\alpha_{s,t}^{\tau,-}\right|^{2}dt\right].
\end{align*}
Therefore
\begin{align*}
\sup_{\tau}\sup_{s\in[0,S]}\E\left[\int_0^T\left|\alpha_{s,t}^{\tau,-}\right|^{2}dt\right]\leq K\left(J(\alpha^{0})+K(T+1)\right).
\end{align*}
The bound for the sequence $(\alpha^{\tau,+})$ is identical and we omit the details. Using this inequality for any $s\in[0,S]$, we have that $s\in\left((n-1)\tau,n\tau\right]$ for some $n\in\{1,\dots,N_{\tau}\}$,
\begin{align*}
&\|\hat{\alpha}^{\tau}_{s}\|^{2}_{L^{2}_{\text{prog}}\left([0,T]\times\Omega;\mathbb{R}^p\right)}=\left\|\left(1-\tfrac{s-(n-1)\tau}{\tau}\right)\alpha^{n-1}_t+\left(\tfrac{s-(n-1)\tau}{\tau}\right)\alpha_t^{n}\right\|^{2}_{L^{2}_{\text{prog}}\left([0,T]\times\Omega;\mathbb{R}^p\right)}\\
&\leq \left(\left\|\left(1-\tfrac{s-(n-1)\tau}{\tau}\right)\alpha_t^{n-1}\right\|_{L^{2}_{\text{prog}}\left([0,T]\times\Omega;\mathbb{R}^p\right)}+\left\|\left(\tfrac{s-(n-1)\tau}{\tau}\right)\alpha_t^{n}\right\|_{L^{2}_{\text{prog}}\left([0,T]\times\Omega;\mathbb{R}^p\right)}\right)^{2}\\
&\leq \left(\left\|\alpha_t^{n-1}\right\|_{L^{2}_{\text{prog}}\left([0,T]\times\Omega;\mathbb{R}^p\right)}+\left\|\alpha_t^{n}\right\|_{L^{2}_{\text{prog}}\left([0,T]\times\Omega;\mathbb{R}^p\right)}\right)^{2}\\
&\leq 2\left(\left\|\alpha_t^{n-1}\right\|^{2}_{L^{2}_{\text{prog}}\left([0,T]\times\Omega;\mathbb{R}^p\right)}+\left\|\alpha_t^{n}\right\|^{2}_{L^{2}_{\text{prog}}\left([0,T]\times\Omega;\mathbb{R}^p\right)}\right)\leq 4\sup_{\tau}\sup_{s\in[0,S]}\left\|\alpha^{\tau,-}_{s}\right\|^{2}_{L^{2}_{\text{prog}}([0,T]\times\Omega;\mathbb{R}^p)}\\
&\leq 2^{2}K(J(\alpha^0)+K(T+1)).
\end{align*}
The result then follows from taking the supremum over $s\in[0,S]$.
\end{proof}
Before proving the convergence of the sequences $\left(\alpha^{\tau,+}\right)_{\tau}$, $\left(\alpha^{\tau,-}\right)_{\tau}$ and $\left(\hat{\alpha}^{\tau}\right)_{\tau}$ we need the following three Lemmas.

\begin{lemma}
\label{3_diff_controls}
Let Assumptions \ref{Assumptions_Old}, \ref{Ass_2nd_derivative_sapce}, \ref{additional_assumptions}, \ref{ass: D-xf and D_x g} ,\ref{assumption_for_uncontrolled_vol} and \ref{ass:D_a^2sig} hold.
Then there exists a constant $C$ depending on $K$, $T$ and $C_{\mathbb{H}^\infty,{BMO}}$ 
such that for any $\varphi, \phi, \theta\in L^2_{\text{prog}}\left([0,T]\times\Omega;\mathbb{R}^p\right)$  we have
\begin{align*}
\E\bigg[\int_0^T\bigg|D_a\mathcal{H}\left(t,X_t^\varphi,Y_t^\varphi,Z_t^\varphi,\phi_t\right)&-D_a\mathcal{H}(t,X_{t}^\theta,Y_{t}^\theta,Z_{t}^\theta,\theta_t)\bigg|^2dt\bigg]\\
&\leq C\left(\E\left[\int_0^T\left|\phi_t-\varphi_t\right|^2dt\right]+\E\left[\int_0^T\left|\varphi_t-\theta_t\right|^2dt\right]\right).
\end{align*}
\end{lemma}
\begin{proof}
Define 
\begin{align*}
I_1:=&\E\left[\int_0^T\left|D_a\mathcal{H}(t,X_t^\varphi,Y_t^\varphi,Z_t^\varphi,\phi_t)-D_a\mathcal{H}(t,X_t^\varphi,Y_t^\varphi,Z_t^\varphi,\varphi_t)\right|^2dt\right]\\
I_2:=&\E\left[\int_0^T\left|D_a\mathcal{H}(t,X_t^\varphi,Y_t^\varphi,Z_t^\varphi,\varphi_t)-D_a\mathcal{H}(t,X_{t}^\theta,Y_{t}^\theta,Z_{t}^\theta,\theta_t)\right|^2dt\right].
\end{align*}
From Lemma \ref{lemma L2 cont of DaH} we have $I_2\leq C\E\left[\int_0^T\left|\varphi_t-\theta_t\right|^2dt\right]$.
Moreover
\begin{align*}
I_1\leq&\E\left[\int_0^T\left|D_ab(t,X_t^\varphi,\phi_t)-D_ab(t,X_t^\varphi,\varphi_t)\right|^2\cdot\left|Y_t^\varphi\right|^2dt\right]\\
&+\E\left[\int_0^T\left|D_af(t,X_t^\varphi,\phi_t)-D_af(t,X_t^\varphi,\varphi_t)\right|^2dt\right]\\
\leq& K^2\left(1+C^2_{\mathbb{H}^\infty,BMO}\right)\E\left[\int_0^T\left|\phi_t-\varphi_t\right|^2dt\right].
\end{align*}
This completes the proof.
\end{proof}

\begin{lemma}
\label{lem:D_aH convex}
Let Assumption \ref{new lambda convexity} hold and assume the Hamiltonian is differentiable in $a$.
Then for any $(t,x,y,z)\in[0,T]\times\mathbb{R}^{d}\times\mathbb{R}^{d}\times\mathbb{R}^{d\times d'}$
\begin{align*}
\left(a-a'\right)\cdot\left(D_a\mathcal{H}(t,x,y,z,a)-D_a\mathcal{H}(t,x,y,z,a')\right)\geq-\lambda|a-a'|^2.
\end{align*}
\end{lemma}
\begin{proof}
From Assumption \ref{new lambda convexity} we see that for any $(t,x,y,z)\in[0,T]\times\mathbb{R}^{d}\times\mathbb{R}^{d}\times\mathbb{R}^{d\times d'}$ the following holds,
\begin{align*}
\mathcal{H}(t,x,y,z,a)+\tfrac{\lambda}{2}|a|^2-\mathcal{H}(t,x,y,z,a')-\tfrac{\lambda}{2}|a'|^2\geq& D_a\mathcal{H}(t,x,y,z,a')\cdot(a-a')\\
&+\lambda a\cdot(a-a').
\end{align*}
Now reversing the role of $a$ and $a'$ and adding both inequalities together gives the required result.
\end{proof}

\begin{lemma}
\label{lem:D_aH_bound}
Let Assumptions \ref{Assumptions_Old},   \ref{Ass_2nd_derivative_sapce}, \ref{additional_assumptions}, \ref{ass: D-xf and D_x g}, \ref{ass: one for Gateuax} and \ref{compactness_of_orbits_assumption} hold. 
For $\tau>0$ and $1\leq n\leq N_\tau$ let $X^n,\left(Y^n,Z^n\right)$ represent the 
solutions to the forward SDE \eqref{controlled_SDE} and backwards SDE \eqref{adjoint equation} corresponding to the control 
$\alpha^n$. Then there exists a constant $C$ depending on $T,K,d,d',\alpha^0$ and $C_{\mathbb{H}^\infty,BMO}$ such that for all $n\in\mathbb{N}$ we have
\begin{align*}
\E\left[\int_0^T\left|D_a\mathcal{H}\left(t,X_t^{n-1},Y_t^{n-1},Z_t^{n-1},\alpha_t^n\right)\right|^2dt\right]\leq C.
\end{align*}
\end{lemma}

\begin{proof}
From Lemma \ref{technical_result_for_BSDEs} and Assumption \ref{ass: one for Gateuax}
\begin{align*}
& \E\left[\int_0^T\left|D_a\mathcal{H}\left(X^{n-1},Y^{n-1},Z^{n-1},\alpha^n\right)\right|^2dt\right]\\
& \qquad \leq TK^2C_{\mathbb{H}^\infty,BMO}+K^2\E\left[\int_0^T\left|Z^{n-1}\right|^2dt\right]\\
& \qquad \qquad + 4K^2\left(T+\E\left[\int_0^T\left|X^{n-1}\right|^2dt\right]+\E\left[\int_0^T\left|\alpha^n\right|^2dt\right]\right).
\end{align*}
From Lemma \ref{prori estimate 2} we have have uniform in $n$ bound for the term $\E\left[\int_0^T|\alpha^n|^2dt\right]$. 
In order to bound the term $\E\left[\int_0^T\left|Z^{n-1}\right|^2dt\right]$ we use Lemma \ref{prori estimate 2} and Lemma \ref{lemma bsde stability} as follows:
\begin{align*}
&\E\left[\int_0^T\left|Z^{n-1}\right|^2dt\right]\leq C_{K,T,d,d'}\E\left[\int_0^T\left|\alpha^{n-1}-\alpha^0\right|^2dt\right]+2\E\left[\int_0^T\left|Z^{0}\right|^2dt\right]\\
&\leq C_{K,T,d,d'}\left(\E\left[\int_0^T\left|\alpha^{n-1}\right|^2dt\right]+\E\left[\int_0^T\left|\alpha^0\right|^2dt\right]\right)+2\E\left[\int_0^T\left|Z^{0}\right|^2dt\right]\leq C_{K,T,d,d',\alpha^0}.
\end{align*}
The term $\E\left[\int_0^T\left|X_t^{n-1}\right|^2dt\right]$ is uniformly bounded due to Lemma \ref{prori estimate 2} and Lemma \ref{lemma bsde stability}.
Therefore
\begin{align*}
\sup_{1\leq n\leq N_\tau}\E\left[\int_0^T\left|D_a\mathcal{H}\left(t,X_t^{n-1},Y_t^{n-1},Z_t^{n-1},\alpha_t^n\right)\right|^2dt\right]\leq C_{T,K,d,d',\alpha^0,C_{\mathbb{H}^\infty,BMO}}.
\end{align*}
\end{proof}

\begin{proof}[Proof of Theorem~\ref{thm:convergence_rate}]
Let $\hat{X}^\tau_{s}$ and $\left(\hat{Y}^\tau_s,\hat{Z}^\tau_s\right)$ be the 
solutions to the forward SDE \eqref{controlled_SDE} and backward SDE \eqref{adjoint equation} corresponding to the control $\hat{\alpha}^\tau_s$. 
From \eqref{discrete_derivative_int} for each $(t,\omega)$, we get
\begin{equation}
\begin{split}
\label{start}
|\hat{\alpha}^\tau_{s} & -\alpha_{s}|^2
= \int_0^s\tfrac{d}{du}\left(\left|\hat{\alpha}^\tau_{u}-\alpha_{u}\right|^2\right)du
=2\int_0^s\left(\hat{\alpha}^\tau_{u}-\alpha_{u}\right)\cdot\left(\tfrac{d}{du}\hat{\alpha}^\tau_u-\tfrac{d}{du}\alpha_{u}\right)du\\
& = 2\int_0^s\left(\hat{\alpha}^\tau_u-\alpha_u\right)\cdot
\left(D_a\mathcal{H}(X_{u},Y_{u},Z_{u},\alpha_u)-D_aH(X^{\tau,-}_{u},Y^{\tau,-}_{u},Z^{\tau,-}_u,\alpha^{\tau,+}_{u})\right)du\\
\end{split}
\end{equation}
Hence
\[
\begin{split}
|\hat{\alpha}^\tau_{s} & - \alpha_{s}|^2  = 2\int_0^s\left(\hat{\alpha}^\tau_u-\alpha_u\right)\cdot\left(D_a\mathcal{H}(X_{u},Y_{u},Z_{u},\alpha_u)
-D_a\mathcal{H}\left(X_u,Y_u,Z_u,\hat{\alpha}^\tau_u\right)
\right)\\
&+2\int_0^s\left(\hat{\alpha}^\tau_u-\alpha_u\right)
\cdot
\left(D_a\mathcal{H}\left(X_u,Y_u,Z_u,\hat{\alpha}^\tau_u\right)-D_a\mathcal{H}(\hat{X}_{u}^\tau,\hat{Y}_{u}^\tau,\hat{Z}_{u}^\tau,\hat{\alpha}_u^\tau)\right)du\nonumber\\
&+2\int_0^s\left(
\hat{\alpha}^\tau_u-\alpha_u\right)\cdot\left(D_a\mathcal{H}(\hat{X}_{u}^\tau,\hat{Y}_{u}^\tau,\hat{Z}_{u}^\tau,\hat{\alpha}_u^\tau)-D_aH(X^{\tau,-}_{u},Y^{\tau,-}_{u},Z^{\tau,-}_u,\alpha^{\tau,+}_{u})
\right)du
\end{split}
\]
From Young's inequality and Lemma~\ref{lem:D_aH convex} we get
\begin{align*}
\left|\hat{\alpha}^\tau_{s}-\alpha_{s}\right|^2\leq&\left(3+\lambda\right)\int_0^s\left|\hat{\alpha}^\tau_u-\alpha_u\right|^2du\\
&+\int_0^s\left|D_a\mathcal{H}(X_{u},Y_{u},Z_{u},\alpha_u)
-D_a\mathcal{H}\left(X_u,Y_u,Z_u,\hat{\alpha}^\tau_u\right)\right|^2du\nonumber\\
&+
\int_0^s\left|D_a\mathcal{H}(\hat{X}_{u}^\tau,\hat{Y}_{u}^\tau,\hat{Z}_{u}^\tau,\hat{\alpha}_u^\tau)-D_aH(X^{\tau,-}_{u},Y^{\tau,-}_{u},Z^{\tau,-}_u,\alpha^{\tau,+}_{u})\right|^2\nonumber,
\end{align*} 
Taking expectations, integrating from $0$ to $T$ and applying Fubini's Theorem gives
\begin{equation}
\label{eqn:helper}
\begin{split}
\E&\left[\int_0^T\left|\hat{\alpha}^\tau_s-\alpha_s\right|^2dt\right]\leq \left(3+\lambda\right)\int_0^s\E\left[\int_0^T\left|\hat{\alpha}^\tau_u-\alpha_u\right|^2dt\right]du+\int_0^s\left[I^{(1)}_u+I^{(2)}_u\right]du
\end{split}
\end{equation}
where
\begin{align*}
I^{(1)}_u:=&\E\left[\int_0^T\left|D_aH(X^{\tau,-}_{u},Y^{\tau,-}_{u},Z^{\tau,-}_u,\alpha^{\tau,+}_{u})-D_a\mathcal{H}(\hat{X}_{u}^\tau,\hat{Y}_{u}^\tau,\hat{Z}_{u}^\tau,\hat{\alpha}_u^\tau)\right|^2dt\right]\\
\leq&C\E\left[\int_0^T\left|\alpha^{\tau,+}_u-\alpha^{\tau,-}_u\right|^2dt\right]+C\E\left[\int_0^T\left|\alpha^{\tau,-}_u-\hat{\alpha}^\tau_u\right|^2dt\right]
\end{align*}
and
\begin{align*}
I^{(2)}_u:=&\E\left[\int_0^T\left|D_a\mathcal{H}(X_{u},Y_{u},Z_{u},\alpha_u)
-D_a\mathcal{H}\left(X_u,Y_u,Z_u,\hat{\alpha}^\tau_u\right)\right|^2du\right]\\
\leq& C\E\left[\int_0^T\left|\hat{\alpha}^\tau_{u}-\alpha_u\right|^2dt\right]
\end{align*}
and where $C$ is the constant from Lemma \ref{3_diff_controls} and depends on $K$, $T$ and $C_{\mathbb{H}^\infty,BMO}$. 
From this and \eqref{eqn:helper} we have 
\begin{align}
\label{eqn:helper_2}
\E\left[\int_0^T\left|\hat{\alpha}^\tau_s-\alpha_s\right|^2dt\right]\leq(3+\lambda+C)\int_0^s\E\left[\int_0^T\left|\hat{\alpha}^\tau_u-\alpha_u\right|^2dt\right]+\int_0^sI^{(1)}_udu. 
\end{align}
In order to bound $I^{(1)}_u$ we note that due to the definition of the interpolations \eqref{interp} for $s\in((n-1)\tau,n\tau]$, $\alpha^{\tau,+}_{s}=\alpha^n=\hat{\alpha}^\tau_{n\tau}$ and $\alpha^{\tau,-}_{s}=\alpha^{n-1}=\hat{\alpha}^\tau_{(n-1)\tau}$.
This lets us write
\begin{align}
\E\left[\int_0^T\left|\alpha^{\tau,+}_{s}-\alpha^{\tau,-}_{s}\right|^2dt\right] 
&=\E\left[\int_0^T\tau^2 \left| D_a\mathcal{H}\left(X^{n-1},Y^{n-1},Z^{n-1},\alpha^n \right) \right|^2 dt\right]\label{eqn:helper_3}\,,\\
\E\left[\int_0^T\left|\hat{\alpha}^{\tau}_{s}-\alpha^{\tau,-}_{s}\right|^2dt\right]
& \leq \E\left[\int_0^T \tau^2 \left|D_a\mathcal{H}\left(X^{n-1},Y^{n-1},Z^{n-1},\alpha^n\right)\right|^2 dt\right]\label{eqn:helper_4}.
\end{align}
Therefore, integrating \eqref{eqn:helper_3} and \eqref{eqn:helper_4}   from $0$ to $s$ we have 
\begin{align*}
&\int_0^S\E\left[\int_0^T\left|\alpha^{\tau,+}_{s}-\alpha^{\tau,-}_{s}\right|^2dt\right]+\E\left[\int_0^T\left|\hat{\alpha}^{\tau}_{s}-\alpha^{\tau,-}_{s}\right|^2dt\right]\,ds\\
&\leq 2\tau^2\sum_{n=1}^{N_\tau}\int_{(n-1)\tau}^{n\tau}\left|D_a\mathcal{H}\left(X^{n-1},Y^{n-1},Z^{n-1},\alpha^n\right)\right|^2\,ds\leq 2SC\tau^2,
\end{align*}
where the final inequality follows from Lemma \ref{lem:D_aH_bound} and $C$ is a constant depending on $T,K,d,d',\alpha^0,C_{\mathbb{H}^\infty,BMO}$. 
Therefore we have shown that $\int_0^sI^{(1)}_udu\leq 2CS\tau^2$.
Substituting this into \eqref{eqn:helper_2} we get that
\begin{align*}
\E\left[\int_0^T\left|\hat{\alpha}^\tau_s-\alpha_s\right|^2dt\right]&\leq\left(3+\lambda+C_{K,T,\alpha^0,C_{\mathbb{H}^\infty,BM0}}\right)\int_0^s\E\left[\int_0^T\left|\hat{\alpha}^\tau_u-\alpha_u\right|^2dt\right]du+2CS\tau^2,
\end{align*}
Gr\"{o}nwall's Lemma and taking supremums in $s\in[0,S]$ gives
\begin{align*}
\sup_{s\in[0,S]}\mathbb{E}\left[\int_0^T\left|\hat{\alpha}^\tau_s-\alpha_s\right|^2dt\right]ds\leq \tau^2 CSe^{S(3+C+\lambda)},
\end{align*}
which shows the required rate of convergence for the sequence $(\hat{\alpha}^\tau)_\tau$. 
The convergence of the sequences $(\alpha^{\tau,+})_\tau$ and $(\alpha^{\tau,-})_\tau$ 
follow from \eqref{eqn:helper_3}, \eqref{eqn:helper_4} and Lemma \ref{lem:D_aH_bound}. 
\end{proof}

\section{Convergence of the Gradient Flow}
\label{Sect: convergence to OC}
The main aim of this section is to prove Theorems~\ref{prop: L^2-convergence of D_aH} and~\ref{thm: Convergence_of_value_fiunction}. 
To do that we will first need to prove the energy identity which is Lemma~\ref{functional_decreases_along_flow}.
Next we prove a bound on the solutions of~\eqref{Gradient_flow_system_intro} which is uniform in $s\in [0,\infty)$ which will allow us to extend the gradient flow to $[0, \infty)$, this is Theorem~\ref{thm: Extending gradient flow}.

\begin{lemma}\label{Gateaux_Derivative}
Under Assumptions \ref{Assumptions_Old} and \ref{ass: one for Gateuax} the map $\alpha\mapsto J\left(\alpha\right)$ is Gateaux differentiable with 
\begin{align*}
\frac{d}{d\delta}J\left(\alpha+\delta(\alpha'-\alpha)\right)|_{\delta=0}=\E\left[\int_{0}^{T}D_{a}\mathcal{H}\left(t,X_{t}^{\alpha},Y_{t}^{\alpha},Z_{t}^{\alpha},\alpha_{t}\right)\cdot\left(\alpha'_{t}-\alpha_{t}\right)dt\right],
\end{align*}
where $\alpha,\alpha'\in\mathbb{H}_{p}^{2}$.
\end{lemma}
Lemma \ref{Gateaux_Derivative} is a classical result in stochastic control literature and gives us a notion of the Gateaux derivative of the cost functional. 
See e.g.~\cite[Lemma 4.8]{carmona2016lectures}.
\begin{proof}[Proof of Lemma \ref{functional_decreases_along_flow}]
For $\varepsilon\in(0,1)$ let $X^{\varepsilon}=X^{\beta_{t}+\varepsilon(\alpha_{t}-\beta_{t})}$, $Y^{\varepsilon}=Y^{\beta_{t}+\varepsilon(\alpha_{t}-\beta_{t})}$ and $Z^{\varepsilon}=Z^{\beta_{t}+\varepsilon(\alpha_{t}-\beta_{t})}$ and let $\alpha,\beta\in L^{2}([0,T]\times\Omega;\mathbb{R}^{p})$. Then from the Fundamental Theorem of Calculus and Lemma \ref{Gateaux_Derivative} we have
\begin{align*}
J(\alpha)-J(\beta)=&\int_{0}^{1}\frac{d}{d\delta}J(\beta+\varepsilon(\alpha-\beta) +\delta(\alpha-\beta)|_{\delta=0}d\varepsilon\\
=&\int_0^{1}\E\left[\int_{0}^{T}D_{a}\mathcal{H}\left(t,X_{t}^{\varepsilon},Y_{t}^{\varepsilon},Z_{t}^{\varepsilon},\beta_{t}+\varepsilon(\alpha_{t}-\beta_{t})\right)\cdot\left(\alpha_{t}-\beta_{t}\right)dt\right]d\varepsilon.
\end{align*} 
In what follows we apply the above observation with $\alpha_{t}=\alpha_{s+h,t}$ and $\beta_{t}=\alpha_{s,t}$ and introduce the following notation 
\begin{align}
X^{\varepsilon,h}_{s,t}:=&X_{t}^{\alpha_{s,t}+\varepsilon\left(\alpha_{s+h,t}-\alpha_{s,t}\right)}\text{ where }X^{\varepsilon,h}_{s,t}\rightarrow X_{s,t}\text{ in $\mathbb{H}^{2}_{d}$ as $h\rightarrow 0$,}\nonumber\\
Y^{\varepsilon,h}_{s,t}:=&Y_{t}^{\alpha_{s,t}+\varepsilon\left(\alpha_{s+h,t}-\alpha_{s,t}\right)}\text{ where }Y^{\varepsilon,h}_{s,t}\rightarrow Y_{s,t}\text{ in $\mathbb{H}^{2}_{d}$ as $h\rightarrow 0$}\\
Z^{\varepsilon,h}_{s,t}:=&Z_{t}^{\alpha_{s,t}+\varepsilon\left(\alpha_{s+h,t}-\alpha_{s,t}\right)}\nonumber\text{ where }Z^{\varepsilon,h}_{s,t}\rightarrow Z_{s,t}\text{ in $\mathbb{H}^{2}_{d\times d'}$ as $h\rightarrow 0$},
\end{align}
where the convergence results follow from Lemma \ref{lemma bsde stability}. 
Using this notation, and applying the above observation we have:
\begin{align*}
&\frac{d}{ds}J(\alpha_{s})=\lim_{h\rightarrow 0}\frac{J(\alpha_{s+h}) - J(\alpha_{s})}{h}\\
=&\lim_{h\rightarrow 0}\frac{1}{h} \!\int_0^{1} \!\!\E\bigg[\int_{0}^{T}\!\! D_{a}\mathcal{H}\left(t,X_{t}^{\varepsilon,h},Y_{t}^{\varepsilon,h},Z_{t}^{\varepsilon,h},\alpha_{s,t}+\varepsilon(\alpha_{s+h,t}-\alpha_{s,t})\right)\cdot \left(\alpha_{s+h,t}-\alpha_{s,t}\right)dt\bigg]d\varepsilon\\
=&\int_0^{1}\lim_{h\rightarrow 0} \E\bigg[\int_{0}^{T}D_{a}\mathcal{H}\left(t,X_{t}^{\varepsilon,h},Y_{t}^{\varepsilon,h},Z_{t}^{\varepsilon,h},\alpha_{s,t}+\varepsilon(\alpha_{s+h,t}-\alpha_{s,t})\right)\cdot \frac{\left(\alpha_{s+h,t}-\alpha_{s,t}\right)}{h}dt\bigg]d\varepsilon,
\end{align*}
where the final line is justified by bounding
\begin{align*}
\E\left[\int_{0}^{T}D_{a}\mathcal{H}\left(t,X_{t}^{\varepsilon,h},Y_{t}^{\varepsilon,h},Z_{t}^{\varepsilon,h},\alpha_{s,t}+\varepsilon(\alpha_{s+h,t}-\alpha_{s,t})\right)\cdot \frac{\left(\alpha_{s+h,t}-\alpha_{s,t}\right)}{h}dt\right]
\end{align*} uniformly in $h$ as follows: applying H{\"o}lder's inequality and the $L^2$ triangle inequality 
\begin{align*}
\E&\left[\int_{0}^{T}|D_{a}\mathcal{H}\left(t,X_{t}^{\varepsilon,h},Y_{t}^{\varepsilon,h},Z_{t}^{\varepsilon,h},\alpha_{s,t}+\varepsilon(\alpha_{s+h,t}-\alpha_{s,t})\right)\cdot\frac{ \left(\alpha_{s+h,t}-\alpha_{s,t}\right)}{h}|dt\right]\\
&\leq \left(I_{1}+I_2+I_3\right)^{\frac{1}{2}}\E\left[\int_{0}^{T}\left|\frac{\alpha_{s+h,t}-\alpha_{s,t}}{h}\right|^{2}dt\right]^{\frac{1}{2}},
\end{align*}
where 
\begin{align*}
I_1:=\E\left[\int_{0}^{T}|D_{a}b\left(t,X_{t}^{\varepsilon,h},\alpha_{s,t}+\varepsilon(\alpha_{s+h,t}-\alpha_{s,t})\right)\cdot Y_{t}^{\varepsilon,h}|^{2}dt\right]\leq TK^2C_{\mathbb H^\infty,\,\text{BMO}}^2,
\end{align*}
and 
\begin{align*}
I_2:=& \E\left[\int_{0}^{T}|D_{a}\sigma\left(t,X_{t}^{\varepsilon,h},\alpha_{s,t}+\varepsilon(\alpha_{s+h,t}-\alpha_{s,t})\right)\cdot Z_{t}^{\varepsilon,h}|^{2}dt\right]\\
\leq& K^2\E\left[\int_0^T\left|Z_{t}^{\varepsilon,h}\right|^2dt\right]\leq K^2\E\left[\int_0^T|Z_{s,t}|^2dt\right]+\delta,
\end{align*}
for some $\delta>0$ and $h$ sufficiently small. The final inequality above follows from the fact $Z_{t}^{\varepsilon,h}\rightarrow Z_{s,t}$ in $\mathbb{H}^2_{d\times d'}$ as $h\rightarrow 0$. 
Moreover
\begin{align*}
I_3:=&\E\left[\int_{0}^{T}|D_{a}f(t,X_{t}^{\varepsilon,h},\alpha_{s,t}+\varepsilon(\alpha_{s+h,t}-\alpha_{s,t}))|^{2}dt\right]\\
\leq& 2K^{2}\E\left[\int_{0}^{T}1+|X_{s,t}^{\varepsilon,h}|^{2}+|\alpha_{s,t}+\varepsilon(\alpha_{s+h,t}-\alpha_{s,t})|^{2}dt\right]
\end{align*}
and so finally 
\begin{align*}
I_3 & = 2K^2T+\E\int_0^T \left[ 2K^2|X_{s,t}^{\varepsilon,h}|^{2} + 4|\alpha_{s,t}|^2 + 4K^2\varepsilon^2|\alpha_{s+h,t}-\alpha_{s,t}|^{2}\right]\,dt\\
& \leq2K^2T+2K^2c(1+|x|^2)+4\E\left[\int_0^T|\alpha_{s,t}|^2dt\right]+4K^2\varepsilon^2h\int_{s}^{s+h}\E\left[\left|\alpha_{r,t}\right|^2\right]dr\\
& \leq 2K^2T+2K^2c(1+|x|^2)+4\E\left[\int_0^T|\alpha_{s,t}|^2dt\right]+4K^2\varepsilon^2h\|\alpha\|_{L^2(0,S;L^2_{\text{prog}}([0,T]\times\Omega;\mathbb{R}^p))},
\end{align*}
for $\delta>0$ and $h$ sufficiently small. 
The penultimate inequality follows from Theorem \ref{controlled_SDE_intro} and Fubini's Theorem.
This justifies the the interchanging of the limit and integral, also we note that
\begin{align*}
\bigg|\mathbb E\bigg[\int_{0}^{T}D_{a}\mathcal{H}&\left(t,X_{t}^{\varepsilon,h},Y_{t}^{\varepsilon,h},Z_{t}^{\varepsilon,h},\alpha_{s,t}+\varepsilon(\alpha_{s+h,t}-\alpha_{s,t})\right)\cdot \frac{\left(\alpha_{s+h,t}-\alpha_{s,t}\right)}{h}\\
&-D_{a}\mathcal{H}\left(t,X_{s,t},Y_{s,t},Z_{s,t},\alpha_{s,t}\right)\cdot\left(\frac{d}{ds}\alpha_{s,t}\right)dt\bigg]\bigg|\rightarrow 0, \enspace\left(h\rightarrow 0\right). 
\end{align*}
Therefore,
\begin{align*}
\tfrac{d}{ds}&J(\alpha_{s})\\
=&\int_0^{1}\lim_{h\rightarrow 0} \E\left[\int_{0}^{T}D_{a}\mathcal{H}\left(t,X_{t}^{\varepsilon,h},Y_{t}^{\varepsilon,h},Z_{t}^{\varepsilon,h},\alpha_{s,t}+\varepsilon(\alpha_{s+h,t}-\alpha_{s,t})\right)\cdot \tfrac{\left(\alpha_{s+h,t}-\alpha_{s,t}\right)}{h}dt\right]d\varepsilon\\
=&\int_0^1\E\left[\int_{0}^{T}D_{a}\mathcal{H}(t,X_{s,t},Y_{s,t},Z_{s,t},\alpha_{s,t})\cdot\left(\tfrac{d}{ds}\alpha_{s,t}\right)\right]d\varepsilon\\
=&-\E\left[\int_{0}^{T}|D_{a}\mathcal{H}(t,X_{s,t},Y_{s,t},Z_{s,t},\alpha_{s,t})|^{2}\right].
\end{align*}
The proof is thus complete.
\end{proof}

The gradient flow system~\eqref{Gradient_flow_system_intro} is defined for $s\in[0,S]$ we now want to show that it exists for $s\in [0,\infty)$ i.e. we want to prove Theorem \ref{thm: Extending gradient flow}.

\begin{proof}[Proof of Theorem \ref{thm: Extending gradient flow}]

From Lemma~\ref{functional_decreases_along_flow} we have that $\infty>J(\alpha^{0})\geq J(\alpha_{s,t})$.
From this and Assumption~\ref{compactness_of_orbits_assumption} we have that
\begin{align*}
J(\alpha^0)\geq&\E\left[\int_{0}^{T}f(t,X_{s,t},\alpha_{s,t})dt+g(X_{s,T})\right]\geq\E\left[\int_{0}^{T}-K + K^{-1}|\alpha_{s,t}|^{2}dt-K\right]\\
=&-K(T+1) + K^{-1}\E\left[\int_{0}^{T}|\alpha_{s,t}|^{2}dt\right].
\end{align*}
From this we see that $\|\alpha_{s}\|_{L^{2}_{\text{prog}}([0,T]\times\Omega;\mathbb{R}^p)}$ is bounded uniformly independently of $s$. 
Hence we can construct the gradient flow system on the interval $[S,2S]$ by starting with $\alpha_S$ as the initial condition $\alpha^0$ in Theorem~\ref{thm:convergence_rate}. 
By induction the solution exists on $[0,\infty)$.
\end{proof}

\begin{proof}[Proof of Theorem~\ref{prop: L^2-convergence of D_aH}]
From Lemma~\ref{functional_decreases_along_flow} we have 
\begin{align*}
J(\alpha_S)-J(\alpha_0)=\int_0^{S}\frac{d}{ds}J(\alpha_s)ds = -\int_0^S \left\|D_a\mathcal{H}(t,X_s,Y_s,Z_s,\alpha_s)\right\|_{\mathbb{H}^2_p}^{2} \,ds \leq 0. 
\end{align*}
Hence 
\[
\int_0^S \left\|D_a\mathcal{H}(t,X_s,Y_s,Z_s,\alpha_s)\right\|_{\mathbb{H}^2_p}^{2} \,ds = J(\alpha^0) - J(\alpha_S)\,.
\]
Taking the limit as $S\to \infty$ we have 
\[
\int_0^\infty \left\|D_a\mathcal{H}(t,X_s,Y_s,Z_s,\alpha_s)\right\|_{\mathbb{H}^2_p}^{2} \,ds = J(\alpha^0) - \lim_{S\to \infty }J(\alpha_S) \leq J(\alpha^0)  - \inf_{\alpha} J(\alpha) < \infty\,.
\]
Then from this and the uniform continuity of $s\mapsto \left\|D_a\mathcal{H}(t,X_s,Y_s,Z_s,\alpha_s)\right\|_{\mathbb{H}^2_p}^{2}$ (which follows from Lemma~\ref{lemma L2 cont of DaH}) we have that 
$
\lim_{s\to \infty}\left\|D_a\mathcal{H}(t,X_s,Y_s,Z_s,\alpha_s)\right\|_{\mathbb{H}^2_p}^{2} = 0
$.
\end{proof}

Before proving convergence of the cost functional along solutions to the gradient flow we will recall the following key step in the proof of Pontryagin Optimality Principle as a sufficient condition.

\begin{lemma}
\label{lemma: convexity of cost}
Let $\beta,\theta\in L^{2}_{\text{prog}}([0,T]\times\Omega;\mathbb{R}^{p})$ be arbitrary admissible controls and Assumption \ref{Assumptions_Old} and \ref{Ass_2nd_derivative_sapce} hold. 
Assume further that for any $(t,y,z)\in[0,T]\times\mathbb{R}^d\times\mathbb{R}^{d\times d'}$ the maps
\begin{align*}
(x,a)\mapsto\mathcal{H}(t,x,y,z,a)\text{ and }x\mapsto g(x)
\end{align*}
are convex.
Then 
\begin{align}\label{eqn:convexity of cost}
J(\beta)-J(\theta) \leq \left( D_a\mathcal{H}\left(X^{\beta},Y^{\beta},Z^{\beta},\beta\right), \beta-\theta\right)\,.
\end{align}
If, additionally, there is $\eta>0$ such that for any $(t,y,z)\in[0,T]\times\mathbb{R}^d\times\mathbb{R}^{d\times d'}$ the map $(x,a)\mapsto\mathcal{H}(t,x,y,z,a) -\frac\eta2|a|^2$ is convex then 
\begin{align*}
J(\beta)-J(\alpha)\leq\left(D_a \mathcal H(X^\beta,Y^\beta,Z^\beta,\beta),\beta-\theta\right) -\eta\|\theta-\beta\|_{L^2_{\text{prog}}( [0,T]\times \Omega;\mathbb R^p)}^2\,.
\end{align*}
\end{lemma}
\begin{proof}
Let $\beta$ and $\theta$ be as in the statement of the lemma. Then
\begin{align*}
J(\beta)-J(\theta)=\E\left[\int_0^Tf(t,X_t^{\beta},\beta_t)-f(t,X_t^{\theta},\theta_t)dt\right]+\E\left[g(X_T^{\beta})-g(X_T^{\theta})\right].
\end{align*}
From the convexity of $g$ we have for any $x,y\in\mathbb{R}^d$ that $g(y)-g(x)\leq D_xg(y)\cdot(y-x)$.
Writing $\Delta X_t := X_t^{\beta}-X_t^{\theta}$ and applying Ito's product rule we thus get
\begin{align*}
& \E\left[g(X_T^{\beta})-g(X_T^{\theta})\right]\leq \E\left[D_xg(X_T^{\beta})\cdot \Delta X_T\right]=\E\left[Y_{T}^{\beta}\cdot \Delta X_T\right]\\
& = \E\left[\int_0^TY_t^{\beta}\cdot\left(b(t,X_t^{\beta},\beta_t)-b(t,X_t^{\theta},\theta_t)\right)dt\right] -\E\left[\int_0^T \!\!\Delta X_t\cdot D_x\mathcal{H}(t,X_t^{\beta},Y_t^{\beta},Z_{t}^{\beta},\beta_{t})dt\right]\\
& \qquad \qquad +\E\left[\int_0^TZ_t^{\beta}:\left(\sigma(t,X_t^{\beta},\beta_t)-\sigma(t,X_t^{\theta},\theta_t)\right)dt\right],
\end{align*}
where we have used the properties of the solutions to the adjoint equation and controlled SDE to conclude the stochastic integrals equal zero. Also,
\begin{align*}
\E&\left[\int_0^Tf\left(t,X_t^{\beta},\beta_t\right)-f\left(t,X_t^{\theta},\theta_t\right)dt\right]\\
=&\E\left[\int_0^T\mathcal{H}\left(t,X_{t}^{\beta},Y_{t}^{\beta},Z_{t}^{\beta},\beta_t\right)-\mathcal{H}\left(t,X_{t}^{\theta},Y_{t}^{\beta},Z_{t}^{\beta},\theta_t\right)dt\right]\\
&-\E\left[\int_0^T\left(b(t,X_t^{\beta},\beta_t)-b(t,X_t^{\theta},\theta_t)\right)\cdot Y_t^{\beta}dt\right]\\
&-\E\left[\int_0^TZ_t^{\beta}:\left(\sigma(t,X_t^{\beta},\beta_t)-\sigma(t,X_t^{\theta},\theta_t)\right)dt\right].
\end{align*}
Combining these two results we have 
\begin{equation}\label{eqn:help_help}
\begin{split}
J(\beta)-J(\theta)
=&\E\bigg[\int_0^T\mathcal{H}\left(t,X_{t}^{\beta},Y_{t}^{\beta},Z_{t}^{\beta},\beta_t\right)-\mathcal{H}\left(t,X_{t}^{\theta},Y_{t}^{\beta},Z_{t}^{\beta},\theta_t\right)\\
&+D_x\mathcal{H}(t,X_t^{\beta},Y_t^{\beta},Z_{t}^{\beta},\beta_{t})\cdot\left(X_t^{\theta}-X_t^{\beta}\right)dt\bigg].
\end{split}
\end{equation}
In the setting where the map $(x,a)\mapsto\mathcal{H}(t,x,y,z,a)$ is convex uniformly in all other parameters we have
\begin{align*}
\mathcal{H}(t,x,y,z,a)-\mathcal{H}(t,x',y,z,a')+D_x\mathcal{H}(t,x,y,z,a)\cdot(x'-x)\leq D_a\mathcal{H}(t,x,y,z,a)\cdot(a-a')
\end{align*}
and hence
\begin{align}\label{eqn: helper_for_value_function_convergence}
J(\beta)-J(\theta)\leq\E\left[\int_0^TD_a\mathcal{H}\left(t,X_t^{\beta},Y_t^{\beta},Z_t^{\beta},\beta_t\right)\cdot\left(\beta_t-\theta_t\right)dt\right].
\end{align}
Under the additional convexity assumption we have 
\begin{align*}
\mathcal{H}(t,x,y,z,a)&-\mathcal{H}(t,x',y,z,a')-D_x\mathcal{H}(t,x,y,z,a)\cdot(x'-x)\\
&\leq D_aH(t,x,y,z,a)\cdot(a-a')-\eta|a-a'|^2.
\end{align*}
Therefore substituting the above into \eqref{eqn:help_help} we have  
\begin{align*}
J(\beta)-J(\theta)\leq\E\left[\int_0^TD_aH(X^\beta,Y^\beta,Z^\beta,\beta)\cdot(\beta-\theta) dt\right]-\eta\E\left[\int_0^T|\theta-\beta|^2dt\right].
\end{align*}
The proof is now complete.
\end{proof}
Using this notion of convexity we can now prove the cost functional converges with the rate outlined in Theorem \ref{thm: Convergence_of_value_fiunction}.
\begin{proof}[Proof of Theorem \ref{thm: Convergence_of_value_fiunction}]
We first prove the linear rate of convergence. 
In this proof let $\|\cdot\|$ denote the norm in $L^2_{\text{prog}}\left([0,T]\times\Omega;\mathbb{R}^p\right)$.
From~\eqref{Gradient_flow_system_intro} we have
\begin{equation}\label{helper_1_1}
\begin{split}
\tfrac{d}{ds}\|\alpha_{s}-\alpha^{*}\|^2=&-2\left(\alpha_{s}-\alpha^*, D_a\mathcal{H}\left(t,X_{s},Y_{s},Z_{s},\alpha_{s}\right)\right)
\leq 2\left(J(\alpha^*)-J(\alpha_{s})\right),
\end{split}
\end{equation}
where the final inequality follows from Lemma \ref{lemma: convexity of cost} applied to the controls $\alpha_{s}$ and $\alpha^*$. Now integrating this expression over the interval $[0,S]$ for some $S\in[0,\infty)$ gives
\begin{align*}
\|\alpha_{S}-\alpha^{*}\|^2-\|\alpha_{0}-\alpha^{*}\|^2\leq 2SJ(\alpha^*)-2\int_0^SJ(\alpha_s)\,ds.
\end{align*}
Therefore, 
\begin{align*}
2\int_0^SJ(\alpha_s)ds-2SJ(\alpha^*)
\leq \|\alpha_{0}-\alpha^{*}\|^2 - \|\alpha_{S}-\alpha^{*}\|^2.
\end{align*}
From Lemma \ref{functional_decreases_along_flow} we know that $J(\alpha_s)$ is decreasing allowing us to write 
\begin{align*}
2S\left(J(\alpha_{S})-J(\alpha^*)\right)
\leq \|\alpha_{0}-\alpha^{*}\|^2,
\end{align*}
giving us the required rate of convergence. 

Under the additional strong convexity assumption the same calculation as in \eqref{helper_1_1} gives
\begin{align*}
\tfrac{d}{ds}\left\|\alpha_s-\alpha^*\right\|^2\leq J(\alpha^*)-J(\alpha_s)-\eta\left\|\alpha_s-\alpha^*\right\|^2.
\end{align*}
Hence
\begin{align*}
0\leq e^{\eta s}J(\alpha_s)-e^{\eta s}J(\alpha^*)\leq-\tfrac{d}{ds}\left(e^{\eta s}\left\|\alpha_s-\alpha^*\right\|^2\right).
\end{align*}
Now integrating from $0$ to $S$ and using Lemma \ref{functional_decreases_along_flow}
\begin{equation}
\label{eqn:conv_to_OC}
\begin{split}
\left(J(\alpha_S)-J(\alpha^*)\right)\int_0^Se^{\eta s}ds
& \leq\int_0^Se^{\eta s}(J(\alpha_s)-J(\alpha^*))ds\\
&\leq\left\|\alpha_0-\alpha^*\right\|^2 - e^{\eta S}\left\|\alpha_S-\alpha^*\right\|^2.    
\end{split}
\end{equation}
Therefore 
\begin{align*}
J(\alpha_S)-J(\alpha^*)\leq\frac{\eta}{e^{\eta S}-1}\left\|\alpha_0-\alpha^*\right\|^2\leq \eta e^{-\eta S}\left\|\alpha_0-\alpha^*\right\|^2.
\end{align*}
Finally, from~\eqref{eqn:conv_to_OC}, we have
\begin{align*}
e^{\eta S}\left\|\alpha_S-\alpha^*\right\|^2\leq& -\left(J(\alpha_S)-J(\alpha^*)\right)\int_0^Se^{\eta s}ds+\left\|\alpha_0-\alpha^*\right\|^2\leq\left\|\alpha_0-\alpha^*\right\|^2.
\end{align*}
From this the conclusion follows.
\end{proof}
\section{Examples}
\label{sec examples}

In this short section we give a few examples of control problems that fit the assumptions of the paper.

\begin{example}[Modified Linear-Quadratic]
\label{convex example}
The linear quadratic control problem doesn't exactly fit in the framework of this paper because we cannot allow quadratic growth in $x$ of the cost functions outside of a compact set. 
Our modification reads as follows (we use the same notation as in \cite{carmona2016lectures} Chapter 4.3.1): let $A_{t},B_{t},\beta_t,C_{t},D_{t}$ and $\gamma_t$ be bounded by $K$ deterministic functions of time which take values in appropriate matrix/vector valued spaces $\mathbb{R}^{n}$, so that $b:[0,T]\times\mathbb{R}^{d}\times \mathbb{R}^{p}\rightarrow\mathbb{R}^{d}$ and  $\sigma:[0,T]\times\mathbb{R}^{d}\times \mathbb{R}^{p}\rightarrow\mathbb{R}^{d\times d'}$. 
We define $b$ and $\sigma$ as follows:
\begin{align*}
b(t,x,a)=A_t x + B_t a + \beta_t,\enspace \sigma(t,x,a)=C_t x+ D_t a + \gamma_t,
\end{align*}
We modify the running and terminal reward as follows:
\begin{align*}
f(t,x,a) = \begin{cases}
\frac{1}{2}L_t|x|^2 + \frac{1}{2}a^{T}M_ta & \text{ if }|x|\leq 1, \\
\frac{1}{2}L_t|x| + \frac{1}{2}a^{T}M_ta & \text{ if }|x| > 1,
\end{cases}\enspace g(x)=\begin{cases}
\frac{1}{2}N|x|^{2} & \, \text{if}\, \, \,|x|\leq 1, \\ \frac{1}{2}N |x|& \, \text{if}\, \, \, |x| > 1,
\end{cases}
\end{align*}
where we assume that $L_t$ and $M_t$ are deterministic with $L_t$ valued in $\mathbb R$, with $M_t$ valued in $\mathbb R^{p\times p}$ such that $0 \leq L_t \leq K$, $K^{-1} \leq M_t \leq K$ and where $N\in \mathbb R$ such that $N\geq 0$.

We check each of the assumptions hold: we start with Assumptions \ref{Assumptions_Old} and \ref{ass: D-xf and D_x g} we have 
\begin{align*}
|D_{x}b(t,x,a)|=|A_t|\leq K,\enspace |D_{x}\sigma(t,x,a)|=|C_t|\leq K,\enspace |D_{x}f(t,x,a)|=\begin{cases}
|L_t x| & \text{ if }|x|\leq 1 \\ 
|L_t| & \text{ if }|x|>1
\end{cases}
\end{align*}
so that $|D_{x}f(t,x,a)|\leq K$ and moreover
$|b(t,x,a)|+|\sigma(t,x,a)|\leq K(1+|x|+|a|)$ and clearly $f$ satisfies $|f(t,x,a)|\leq K(1+|x|^{2}+|a|^{2})$ and $|g(x)|\leq K(1+|x|^2)$. 
Also $|D_xg(x)|\leq K$. 
For Assumption \ref{Ass_2nd_derivative_sapce} we have $|D_{x}^{2}b|=|D_{x}^{2}\sigma|=0$, $|D_{x}^{2}f|\leq|L_t|<\infty$ and $|D_xg(x)|\leq|N_t|\leq K$. 
For Assumption \ref{additional_assumptions} we have $|D_{x}D_{a}b(t,x,a)|=|D_{x}D_{a}\sigma(t,x,a)|=|D_{x}D_{a}f(t,x,a)|=0$ and
\begin{align*}
|D_{a}b(t,x,a)|=|B_t|\leq K,\enspace |D_{a}\sigma(t,x,a)|=|D_t|\leq K.
\end{align*}
For Assumption \ref{assumption_for_uncontrolled_vol} we have 
\begin{align*}
|D^2_a b (t,x,a)|=|D^{2}_{a}\sigma(t,x,a)|=0,\enspace |D^{2}_{a}f(t,x,a)|=|D_{a}\left[M_t a\right]|=|M_t|\leq K.
\end{align*}
To check Assumption \ref{new lambda convexity} we note that the Hamiltonian is $\mathcal{H}(t,x,y,z,a)$
\begin{align*}
=\begin{cases}
(A_t x + B_t a + \beta_t)\cdot y + (C_t x+ D_t a + \gamma_t)\cdot z + \frac{1}{2}L_t|x|^2 + \frac{1}{2}a^{T}M_ta & \text {if } |x|\leq 1 \\
(A_t x + B_t a + \beta_t)\cdot y + (C_t x+ D_t a + \gamma_t)\cdot z + \frac{1}{2}L_t|x| + \frac{1}{2}a^{T}M_ta & \text{ if }|x|\geq 1
\end{cases}
\end{align*}
which satisfies the required convexity  with $\lambda = 0$. 
Finally for Assumptions \ref{ass: one for Gateuax} and  \ref{compactness_of_orbits_assumption} we note that
\begin{align*}
f(t,x,a)\geq \frac{1}{2}|M_{t}||a^{2}|\geq  K^{-1}|a|^{2},\enspace|D_{a}f(t,x,a)|\leq |M_t||a|\leq K(1+|x|+|a|).
\end{align*}
and $g(x)\geq 0$. 
\end{example}

\begin{example}[$\lambda$-convex Hamiltonian)]
\label{non-convex examples}
This is similar to Example~\ref{convex example} but with a different running reward and restrict ourselves to the one dimensional case. 
Let $f(t,x,a)$
\begin{align*}
=\begin{cases}
\frac{1}{2}L_t x^2 + \frac{1}{2}M_t\left( a^4 - a^2\right)1_{\{|a|\leq 1\}}+(a-1)^{2}1_{\{a>1\}}+(a+1)^{2}1_{\{a<-1\}} & \text{ if } |x|\leq 1\\
\frac{1}{2}L_t |x| + \frac{1}{2}M_t\left( a^4 - a^2\right)1_{\{|a|\leq 1\}}+(a-1)^{2}1_{\{a>1\}}+(a+1)^{2}1_{\{a<-1\}} & \text{ if } |x|> 1
\end{cases},
\end{align*}
and so we only need to check the assumptions for $f$, $D_{a}f$ since the others are the same as in Example \ref{convex example}. 
To that end note that 
\begin{align*}
D_{a}f(t,x,a)=\begin{cases}
\frac{1}{2}M_t (4a^3 - 2a) & \text{ if }|a|\leq 1 \\
2(a-1)&\text{ if }a>1\\
2(a+1)&\text{ if }a<-1
\end{cases}
\end{align*}
and
\begin{align*}
D_{a}^2f(t,x,a)=\begin{cases}
\frac{1}{2}M_t (12a^2 - 2)& \text{if }|a|\leq 1\\
2&\text{otherwise.}
\end{cases}
\end{align*}
Hence $|D_a f|$ and $|D_{a}^{2}f|$ satisfies the required boundedness conditions. Also $|D_{a}D_{x}f|=0$, and the lower growth condition in Assumption \ref{compactness_of_orbits_assumption} holds since $f(t,x,a)\geq - K + K^{-1}|a|^{2}$. The Hamiltonian is given by 
\begin{align*}
\mathcal{H}(t,x,y,z,a)=&(A_t x + B_t a + \beta_t)\cdot y + (C_t x+ D_t a + \gamma_t)\cdot z\\ 
&+ \frac{1}{2}L_t x^2 + \frac{1}{2}M_t\left( a^4 - a^2\right)1_{\{|a|\leq 1\}}+(a-1)^{2}1_{\{a>1\}}+(a+1)^{2}1_{\{a<-1\}}
\end{align*}
for $|x|\leq 1$ and 
\begin{align*}
\mathcal{H}(t,x,y,z,a)=&(A_t x + B_t a + \beta_t)\cdot y + (C_t x+ D_t a + \gamma_t)\cdot z\\ 
&+ \frac{1}{2}L_t |x| + \frac{1}{2}M_t\left( a^4 - a^2\right)1_{\{|a|\leq 1\}}+(a-1)^{2}1_{\{a>1\}}+(a+1)^{2}1_{\{a<-1\}}
\end{align*}
otherwise. Note that the required convexity breaks for $|a|\leq 1$, however if we add $\frac{\lambda}{2}|a|^{2}$ to the Hamiltonian we have strict convexity provide $\lambda>2$, and so Assumption \ref{new lambda convexity} is satisfied. The gradient flow reads as follows
\begin{align*}
\frac{d}{ds}\alpha_{s,t}=&- B_{t}Y_{s,t} - D_{t}Z_{s,t} - \frac{1}{2}M_t\left(4\alpha_{s,t}^3 - 2\alpha_{s,t}\right)1_{\{|\alpha_{s,t}|\leq 1\}}\\
&-2(\alpha_{s,t}-1)1_{\{\alpha_{s,t}>1\}}-2(\alpha_{s,t}+1)1_{\{\alpha_{s,t}<-1\}}
\end{align*}
with the initial condition $ \alpha_{0,t}=\alpha^{0}_{t}$ and where 
\begin{align*}
\begin{cases}
dX_{s,t}=&\left(A_t X_{s,t} + B_t \alpha_{s,t} + \beta_t\right) dt + \left(C_t X_{s,t}+ D_t \alpha_{s,t} + \gamma_t\right) dW_{t}, \\
dY_{s,t}=&-\biggl(A_t Y_{s,t} + C_t Z_{s,t} + L_t X_{s,t} 1_{\{|X_{s,t}|\leq 1\}}+\frac{1}{2}L_t\frac{X_{s,t}}{|X_{s,t}|}1_{\{|X_{s,t}|>1\}}\biggr)dt+Z_{s,t}dW_{t},\\ X_{0,t}=x,& Y_{s,T}=D_xg(X_{s,T}).
\end{cases}
\end{align*}
\end{example}

\begin{example}[Non-linear drift and diffusion]
\label{ex non-linear drift}
In Examples~\ref{convex example} and~\ref{non-convex examples} we have shown that our assumptions are satisfied with linear drift and diffusion structure (with costs being convex or $\lambda$-convex respectively). 
Here we show that the setting of this paper also covers non-linear drifts and diffusion coefficients, as long as the growth is controlled. 
Let
\begin{align*}
b(t,x,a)=L(x)A_tL(a),\enspace\sigma(t,x,a)=ax+ba+c\text{ where } L(x)=\frac{1}{1+e^{-x}},
\end{align*}
and $A_t$ is real-valued, deterministic satisfies $|A_t|\leq K$ and $a,b,c\in\mathbb{R}$.
Then we see that 
\begin{align*}
|D_xb(t,x,a)|\leq|A_t|\frac{e^{-x}}{(1+e^{-x})^{2}}\leq\frac{K}{1+e^{-x}}\leq K,
\end{align*}
and 
\begin{align*}
|D^{2}_{x}b(t,x,a)|&=|A_{t}|\left(\frac{2e^{-2x}}{(1+e^{-x})^{3}}-\frac{e^{-x}}{(1+e^{-x})^{2}}\right)\leq|A_t|\left(\frac{2e^{-2x}}{(1+e^{-x})^2} + \frac{e^{-x}}{1+e^{-x}}\right)\\
&\leq |A_t|\left(\frac{2}{e^{-2x}(1+2e^{-x}+e^{-2x})}+1\right)\leq 3K,
\end{align*}
also, 
\begin{align*}
|D_{xa}^{2}b(t,x,a)|\leq\left(\frac{e^{-x}}{(1+e^{-x})^2}\right)K\left(\frac{e^{-a}}{(1+e^{-a})^2}\right).
\end{align*}
Therefore Assumptions~\ref{Assumptions_Old}, \ref{Ass_2nd_derivative_sapce}, \ref{additional_assumptions}, ~\ref{ass: D-xf and D_x g}, \ref{assumption_for_uncontrolled_vol}, \ref{compactness_of_orbits_assumption} are satisfied as long as the cost functions are reasonable.
It remains to check Assumption ~\ref{new lambda convexity}. 
We will assume costs come from Example~\ref{convex example}.
\begin{align*}
\mathcal{H}(t,x,y,z,a)=\begin{cases}
L(x)A_tL(a)y + \sigma(t,x,a)z + \frac{1}{2}L_t|x|^2 + \frac{1}{2}a^{T}M_ta & |x|\leq 1\\ L(x)A_tL(a)y +  \sigma(t,x,a)z  + \frac{1}{2}L_t|x| + \frac{1}{2}a^{T}M_ta & \text{ if }|x|\geq 1\,,
\end{cases}
\end{align*}
which means that the map $a\mapsto\mathcal{H}(t,x,y,z,a)+\frac{\lambda}{2}|a|^{2}$ is convex.
\end{example}

\appendix

\section{Proofs of Technical Results}
\label{sec appendix}
\subsection{Proof of Lemma \ref{Well-defined-gradien-proposition}}\label{Well-Posedness-of-gradient-condition-proof}
\begin{proof}
From Assumption \ref{new lambda convexity} we have a $\lambda\geq 0$ such that for any $(t,x,y,z)$ the map 
\begin{align}\label{temp_1}
a\mapsto\mathcal{H}(t,x,y,z,a)+\frac{\lambda}{2}|a|^{2},
\end{align} 
is convex.
Note that for any $a, a', b\in \mathbb{R}^{p}$ we have
\begin{align*}
\mathcal{H}&(t,x,y,z,a)+\frac{\lambda}{2}|a-a'|^{2}=\left(\mathcal{H}(t,x,y,z,a)+\frac{\lambda}{2}|a|^{2}\right)-\frac{\lambda}{2}2a\cdot a'+\frac{\lambda}{2}|a'|^{2}\\
\geq&\left(\mathcal{H}(t,x,y,z,b)+\frac{\lambda}{2}|b|^{2}+\left[D_{a}\mathcal{H}(t,x,y,z,b)+\frac{\lambda}{2} 2b\right]\cdot(a-b) \right)-\frac{\lambda}{2}2a\cdot a'+\frac{\lambda}{2}|a'|^{2}\\
=&\mathcal{H}(t,x,y,z,b)+\frac{\lambda}{2}\left(|b|^{2}-2b\cdot a' + |a'|^{2}\right)+\left[D_{a}\mathcal{H}(t,x,y,z,b)+\frac{\lambda}{2} 2b\right]\cdot(a-b)+\frac{\lambda}{2}2a'\cdot (b-a)\\
=&\mathcal{H}(t,x,y,z,b)+\frac{\lambda}{2}|b-a'|^{2}+\left[D_{a}\mathcal{H}(t,x,y,z,b)+\lambda(b-a')\right]\cdot(a-b)\\
=&\left[\mathcal{H}(t,x,y,z,b)+\frac{\lambda}{2}|b-a'|^{2}\right]+D_{a}\left[\mathcal{H}(t,x,y,z,b)+\frac{\lambda}{2}|b-a'|^{2}\right]\cdot(a-b)
\end{align*}
which shows the desired convexity of the map $a\mapsto \mathcal{H}(t,x,y,z,a)+\frac{\lambda}{2}|a-a'|^{2}$ for arbitrary $a'\in \mathbb R^p$.
\end{proof}

\subsection{Proof of Lemma \ref{technical_result_for_BSDEs}}
\label{proof: technical_result_for_BSDEs}
\begin{remark}\label{Proving BMO}
Note that if we prove $Z\in\mathcal{K}(\mathbb{R}^{d\times d'})$, then it immediately follows that $\int_0^{T}Z^{i,j}dW_{j}=:Z^{i,j}\circ W_{j}\in BMO$ for all $i\in\{1,\dots,d\}$ and $j\in\{1,\dots,d'\}$. To see this, simply note that
\begin{align*}
\|(Z^{i,j}\circ W^{j})\|_{\text{BMO}}=&\sup_{\tau}\left\|\E\left[\langle Z^{i,j}\circ W^{j}\rangle_{T}-\langle Z^{i,j}\circ W^{j}\rangle_{\tau}|\mathcal{F}_{\tau}\right]\right\|_{\infty}=\sup_{\tau}\left\|\E\left[\int_{\tau}^{T}|Z_{t}^{i,j}|^{2}dt\bigg|\mathcal{F}_{\tau}\right]\right\|_{\infty}\\
\leq&\sup_{\tau}\left\|\E\left[\int_{\tau}^{T}|Z_{t}|^{2}dt\bigg|\mathcal{F}_{\tau}\right]\right\|_{\infty}=\|Z\|_{\mathcal{K}}<\infty.
\end{align*}
Therefore, in order to prove Lemma \ref{technical_result_for_BSDEs}, it is sufficient to prove that $\sup_{\alpha}\|Y^{\alpha}\|_{\mathbb{H}^{\infty}}<\infty$ and $\sup_{\alpha}\|Z^{\alpha}\|_{\mathcal{K}}<\infty$ only. 
\end{remark}
\begin{proof}
The bound $\sup_{\alpha}\|Y^{\alpha}\|_{\mathbb{H}^{\infty}}<\infty$ is proven in \cite{kerimkulov2021modified}, which in turn recognizes the adjoint equation as a linear BSDE then invokes Proposition 3.2 in \cite{harter2019stability}.

In order to get the BMO bound for the process $Z^{\alpha}$ we follow a similar strategy as in \cite{ji2021modified} where a similar bound for a different class of stochastic control problems. Applying It\^{o}'s Lemma to $|Y_{t}^{\alpha}|^{2}$, gives
\begin{align*}
d(|Y_{t}^{\alpha}|^{2})=&2(Y_{t}^{\alpha})^{\top}dY_{t}^{\alpha}+dY_{t}^{\alpha}dY_{t}^{\alpha}\\
=&-2(Y_{t}^{\alpha})^{\top}D_{x}\mathcal{H}(t,X_{t}^{\alpha},Y_{t}^{\alpha},Z_{t}^{\alpha},\alpha_{t})dt + 2(Y_{t}^{\alpha})^{\top}Z_{t}^{\alpha}dW_{t}+\tr((Z^{\alpha}_{t})^{\top}Z^{\alpha}_{t})dt\\
=&\left(|Z_{t}^{\alpha}|^{2}-2(Y_{t}^{\alpha})^{\top}D_{x}\mathcal{H}(t,X_{t}^{\alpha},Y_{t}^{\alpha},Z_{t}^{\alpha},\alpha_{t})\right)dt+(Y_{t}^{\alpha})^{\top}Z_{t}^{\alpha}dW_{t}.
\end{align*}
So in integral form we have 
\begin{align*}
|Y_{t}^{\alpha}|^{2}=|D_{x}g(X_{T}^{\alpha})|^{2}+2\int_{t}^{T}(Y_{s}^{\alpha})^{\top}D_{x}\mathcal{H}(s,X_{s}^{\alpha},Y_{s}^{\alpha},Z_{s}^{\alpha},\alpha_{s})ds-\int_{t}^{T}|Z_{s}^{\alpha}|^{2}ds-\int_{t}^{T}(Y_{t}^{\alpha})^{\top}Z_{t}^{\alpha}dW_{s}.
\end{align*}
Then rearranging and using Assumption \ref{Assumptions_Old} we have
\begin{align*}
\int_{\tau}^{T}|Z_{s}^{\alpha}|^{2}ds\leq& K^{2}+\int_{\tau}^{T}2(Y_{s}^{\alpha})^{\top}D_{x}\mathcal{H}(s,X_{s}^{\alpha},Y_{s}^{\alpha},Z_{s}^{\alpha},\alpha_{s})ds-\int_{\tau}^{T}(Y_{s}^{\alpha})^{\top}Z_{s}^{\alpha}dW_{s}\\
\leq&K^{2}+2\|Y^{\alpha}\|_{\mathbb{H}^{\infty}}\int_{\tau}^{T}|D_{x}b(s,X_{s}^{\alpha},\alpha_{s})\cdot Y_{s}^{\alpha}|+|D_{x}\sigma(s,X_{s}^{\alpha},\alpha_{s}):Z_{s}^{\alpha}| + |D_{x}f(s,X_{s}^{\alpha},\alpha_{s})|ds\\
&-\int_{\tau}^{T}(Y_{s}^{\alpha})^{\top}Z_{s}^{\alpha}dW_{s},
\end{align*}
where $\tau\leq T$ is a stopping time. If we now take conditional expectations $\E\left[\cdot\enspace\bigg|\mathcal{F}_{\tau}\right]$, and note that
\begin{align*}
\E\left[\int_{\tau}^{T}\left|\left(Y_{s}^{\alpha}\right)^{\top}Z_{s}^{\alpha}\right|^{2}ds\bigg|\right]\leq C_{\mathbb{H}^{\infty},\text{BMO}}^{2}\E\left[\int_{0}^{T}\left|Z_{s}^{\alpha}\right|^{2}ds\right]<\infty,
\end{align*}
since the pair $\left(Y^{\alpha},Z^{\alpha}\right)$ is a solution to the adjoint equation meaning the stochastic integral is a martingale and equals $0$ upon taking conditional expectations. 
It then follows that, 
\begin{align*}
\E&\left[\int_{\tau}^{T}|Z_{s}^{\alpha}|^{2}ds\bigg|\mathcal{F}_{\tau}\right]\leq C_{C_{\mathbb H^\infty,\,\text{BMO}},T,K}+2C_{\mathbb H^\infty,\,\text{BMO}}\E\left[\int_{\tau}^{T}|D_{x}\sigma(s,X_{s}^{\alpha},\alpha_{s}): Z_{s}^{\alpha}|ds\bigg|\mathcal{F}_{\tau}\right]\\
\leq&C_{C_{\mathbb H^\infty,\,\text{BMO}},T,K}+C_{\mathbb H^\infty,\,\text{BMO}}\E\left[\frac{1}{\lambda}\int_{\tau}^{T}\left|D_{x}\sigma(s,X_{s}^{\alpha},\alpha_{s})\right|^{2}dt+\lambda\int_{\tau}^{T}\left|Z_{s}^{\alpha}\right|^{2}dt\bigg|\mathcal{F}_{\tau}\right]\\
\leq& C_{C_{\mathbb H^\infty,\,\text{BMO}},T,K}+\lambda C_{\mathbb H^\infty,\,\text{BMO}}\E\left[\int_{\tau}^{T}\left|Z_{s}^{\alpha}\right|^{2}ds\bigg|\mathcal{F}_{\tau}\right].
\end{align*}
Picking $\lambda=\frac{1}{2C_{\mathbb H^\infty,\,\text{BMO}}}$, taking a supremum over all $\tau\in[0,T]$ and admissible controls, we have 
\begin{align*}
\sup_{\alpha}\sup_{\tau\leq T}\left\|\left(E\left[\int_{\tau}^{T}|Z_{s}^{\alpha}|^{2}ds|\mathcal{F}_{\tau}\right]^{\frac{1}{2}}\right)\right\|_{\infty}\leq C_{K,T,C_{\mathbb H^\infty,\,\text{BMO}},\lambda}\implies\sup_{i,j}\sup_{\alpha}\left\|\left(Z^{\alpha}\right)^{ij}\circ W_{j}\right\|_{\infty}<\infty.
\end{align*}	
The implication above follows from Remark \ref{Proving BMO}.
\end{proof}
\subsection{Proof of Lemma \ref{lemma bsde stability}}\label{proof:lemma bsde stability}
\begin{proof}
The proof for (\ref{forward_bound}) is an immediate consequence of Doob's Maximal Inequality for submartingale and It\^{o}'s Isometry. To see this,
\begin{align*}
\E\left[\sup_{t\in[0,T]}|X_{t}^{\alpha}-X_{t}^{\alpha'}|^{2}\right]\leq& 2\E\left[\sup_{t\in[0,T]}\left(\int_{0}^{t}|b(u,X_{u}^{\alpha},\alpha)-b(u,X_{u}^{\alpha'},\alpha')|du\right)^{2}\right] \\
&+ 2\E\left[\sup_{t\in[0,T]}\left|\int_{0}^{t}\sigma(u,X_{u}^{\alpha},\alpha)-\sigma(u,X_{u}^{\alpha'},\alpha')dW_{u}\right|^{2}\right]=:2I_{1}+2I_{2}.
\end{align*}
Considering $I_{1}$ and $I_{2}$ separately we have
\begin{align*}
I_{1}\leq& \E\left[\sup_{t\in[0,T]}t\left(\int_{0}^{t}|b(u,X_{u}^{\alpha},\alpha_{u})-b(u,X_{u}^{\alpha'},\alpha'_{u})|^{2}du\right)\right]\text{ (from $\frac{1}{2}-\frac{1}{2}$ H{\"o}lder's Inq)}\\
\leq&T\E\left[\int_0^{T}|b(u,X_{u}^{\alpha},\alpha_{u})-b(u,X_{u}^{\alpha'},\alpha'_{u})|^{2}dt\right]\leq 2TK^{2}\E\left[\int_{0}^{T}|X_{u}^{\alpha}-X_{u}^{\alpha'}|^{2}+|\alpha_{u}-\alpha_{u}'|^{2}\right],
\end{align*}
where the final inequality follows from sing Assumption \ref{additional_assumptions}. For $I_{2}$ set $M_{t}=\int_{0}^{t}\sigma(u,X_{u}^{\alpha},\alpha)-\sigma(u,X_{u}^{\alpha'},\alpha')dW_{u}$ which is a martingale, hence the process $(M_{t}^{2})_{t\in[0,T]}$ is a submartingale. So applying the Doob's martingale Inequality we have 
\begin{align*}
I_{2}\leq& 4 \E\left[\left(\int_{0}^{T}\sigma(u,X_{u}^{\alpha},\alpha_{u})-\sigma(u,X_{u}^{\alpha'},\alpha'_{u})dW_{u}\right)^{2}\right]=4\E\left[\int_{0}^{T}\left|\sigma(u,X_{u}^{\alpha},\alpha_{u})-\sigma(u,X_{u}^{\alpha'},\alpha'_{u})\right|^{2}du\right]\\
\leq&4K^{2}\E\left[\int_{0}^{T}|X_{u}^{\alpha}-X_{u}^{\alpha'}|^{2}+|\alpha_{u}-\alpha_{u}'|^{2}du\right]
\end{align*}
Combining both inequalities gives
\begin{align*}
\mathbb E \sup_{t\in [0,T]} |X_t^\alpha - X_t^{\alpha'}|^2 \leq& C_{T,K}\E\left[\int_{0}^{T}|X_{t}^{\alpha}-X_{t}^{\alpha'}|^{2}dt\right] +C_{T,K}\|\alpha-\alpha'\|_{\mathbb{H}_{p}^{2}}^2\\
\leq&C_{T,K}\int_{0}^{T}\E\left[\sup_{s\leq t }|X_{s}^{\alpha}-X_{s}^{\alpha'}|^{2}\right]dt +C_{T,K}\|\alpha-\alpha'\|_{\mathbb{H}_{p}^{2}}^2.
\end{align*}
From Gr\"{o}nwalll's Inequality $\E\left[\sup_{t\in[0,T]}\left|X_{t}^{\alpha}-X_{t}^{\alpha'}\right|^{2}\right]\leq C_{T,K}\|\alpha-\alpha'\|_{\mathbb{H}_{p}^{2}}^2$.
The bounds for the solutions to the adjoint equation are slightly more involved - this will be similar to Lemma 5.4 in \cite{siska2020gradient} (in that paper the authors bound by the integrated Wasserstein distance instead). Let $\alpha,\alpha'$ be arbitrary admissible controls, and note the following: if 
\begin{align*}
Y_{t}^{\alpha}=&Y_{T}^{\alpha}+\int_{t}^{T}D_{x}\mathcal{H}(s,X_{s}^{\alpha},Y_{s}^{\alpha},Z_{s}^{\alpha},\alpha_{s})ds-\int_{t}^{T}Z_{s}^{\alpha}dW_{s}
\end{align*}
then in differential notation $dY_{t}^{\alpha}=-D_{x}\mathcal{H}(t,X_{t}^{\alpha},Y_{t}^{\alpha},Z_{t}^{\alpha},\alpha_{t})dt + Z_{t}^{\alpha}dW_{t}$, $t\in[0,T]$ and $Y_T^\alpha = g(X_T^\alpha)$.
Bearing this in mind we first compute, using It\^{o}'s product rule and It\^{o}'s lemma, the following
\begin{align*}
d&\left(e^{\beta t}\left|Y_{t}^{\alpha}-Y_{t}^{\alpha'}\right|^{2}\right)=\beta e^{\beta t}\left|Y_{t}^{\alpha}-Y_{t}^{\alpha'}\right|^{2}dt + e^{\beta t}d\left(\left|Y_{t}^{\alpha}-Y_{t}^{\alpha'}\right|^{2}\right)\\
=&\beta e^{\beta t}\left|Y_{t}^{\alpha}-Y_{t}^{\alpha'}\right|^{2}dt + e^{\beta t}\left(2\left(Y_{t}^{\alpha}-Y_{t}^{\alpha'}\right)^{\top}d\left(Y_{t}^{\alpha}-Y_{t}^{\alpha'}\right)+d\left(Y_{t}^{\alpha}-Y_{t}^{\alpha'}\right)\cdot d\left(Y_{t}^{\alpha}-Y_{t}^{\alpha'}\right)\right),
\end{align*}
where $\beta\geq 1$ will be fixed later. Then substituting the following into the above equation,
\begin{align*}
d\left(Y_{t}^{\alpha}-Y_{t}^{\alpha'}\right) =& \left(-D_{x}\mathcal{H}(t,X_{t}^{\alpha},Y_{t}^{\alpha},Z_{t}^{\alpha},\alpha_{t}) + D_{x}\mathcal{H}(t,X_{t}^{\alpha'},Y_{t}^{\alpha'},Z_{t}^{\alpha'},\alpha'_{t})\right)dt+\left(Z_{t}^{\alpha}-Z_{t}^{\alpha'}\right)dW_{t}\\
d\left(Y_{t}^{\alpha}-Y_{t}^{\alpha'}\right)\cdot& d\left(Y_{t}^{\alpha}-Y_{t}^{\alpha'}\right) = \left|Z_{t}^{\alpha}-Z_{t}^{\alpha'}\right|^{2}dt
\end{align*}
we have 
\begin{align*}
d\left(e^{\beta t}\left|Y_{t}^{\alpha}-Y_{t}^{\alpha'}\right|^{2}\right)=&\beta e^{\beta t}\left|Y_{t}^{\alpha}-Y_{t}^{\alpha'}\right|^{2}dt+2e^{\beta t}\left(Y_{t}^{\alpha}-Y_{t}^{\alpha'}\right)^{\top}(Z_{t}^{\alpha}-Z_{t}^{\alpha'})dW_{t}+e^{\beta t}\left|Z_{t}^{\alpha}-Z_{t}^{\alpha'}\right|^{2}dt\\ 
+&2e^{\beta t}\left(Y_{t}^{\alpha}-Y_{t}^{\alpha'}\right)^{\top}\left(-D_{x}\mathcal{H}(t,X_{t}^{\alpha},Y_{t}^{\alpha},Z_{t}^{\alpha},\alpha_{t}) + D_{x}\mathcal{H}(t,X_{t}^{\alpha'},Y_{t}^{\alpha'},Z_{t}^{\alpha'},\alpha'_{t})\right)dt \\
=&e^{\beta t}\Bigl[\beta\left|Y_{t}^{\alpha}-Y_{t}^{\alpha'}\right|^{2}+\left|Z_{t}^{\alpha}-Z_{t}^{\alpha'}\right|^{2}+2\left(Y_{t}^{\alpha}-Y_{t}^{\alpha'}\right)^{\top}\Bigl(-D_{x}\mathcal{H}(t,X_{t}^{\alpha},Y_{t}^{\alpha},Z_{t}^{\alpha},\alpha_{t})\\
&+ D_{x}\mathcal{H}(t,X_{t}^{\alpha'},Y_{t}^{\alpha'},Z_{t}^{\alpha'},\alpha'_{t})\Bigr)\Bigr]dt + e^{\beta t}\left(Y_{t}^{\alpha}-Y_{t}^{\alpha'}\right)^{\top}\left(Z_{t}^{\alpha}-Z_{t}^{\alpha'}\right)dW_{t}.
\end{align*}
In integral form
\begin{align*}
&e^{\beta t}|Y_{t}^{\alpha}-Y_{t}^{\alpha'}|^{2}+\int_{t}^{T}e^{\beta s} \left|Z_{s}^{\alpha}-Z_{s}^{\alpha'}\right|^{2}ds=e^{\beta T}\left|Y_{T}^{\alpha}-Y_{T}^{\alpha'}\right|^{2}-2\int_t^{T}e^{\beta s} \left(Y_{s}^{\alpha}-Y_{s}^{\alpha'}\right)^{\top}\left(Z_{s}^{\alpha}-Z_{s}^{\alpha'}\right)dW_{s}\\
-&\int_t^{T}e^{\beta s}\Biggl[\beta\left|Y_{s}^{\alpha}-Y_{s}^{\alpha'}\right|^2+2\left(Y_{s}^{\alpha}-Y_{s}^{\alpha'}\right)^{\top}\Bigl(-D_{x}\mathcal{H}(s,X_{s}^{\alpha},Y_{s}^{\alpha},Z_{s}^{\alpha},\alpha_{s})+ D_{x}\mathcal{H}(s,X_{s}^{\alpha'},Y_{s}^{\alpha'},Z_{s}^{\alpha'},\alpha'_{s})\Bigr)\Biggr]ds.
\end{align*}
Since $Y^{\alpha},Y^{\alpha'},Z_{s}^{\alpha}$ and $Z_{s}^{\alpha'}$ are solutions to the adjoint equation, the integrand in the stochastic integral above is square integrable hence disappears upon taking (conditional) expectations. Therefore,
\begin{align*}
e^{\beta t}|Y_{t}^{\alpha}&-Y_{t}^{\alpha'}|^{2}+\E_{t}\left[\int_{t}^{T}e^{\beta s}\left|Z_{s}^{\alpha}-Z_{s}^{\alpha'}\right|^{2}ds\right]\\
&=e^{\beta T}\E_{t}\left[\left|D_{x}g(X_{T}^{\alpha})-D_{x}g(X_{T}^{\alpha'})\right|^{2}\right]-\E_{t}\left[\int_t^{T}e^{\beta s}\beta\left|Y_{s}^{\alpha}-Y_{s}^{\alpha'}\right|^2ds\right]\\
&+2\E_{t}\left[\int_{t}^{T}e^{\beta s}\left(Y_{s}^{\alpha}-Y_{s}^{\alpha'}\right)^{\top}\Bigl(D_{x}\mathcal{H}(s,X_{s}^{\alpha},Y_{s}^{\alpha},Z_{s}^{\alpha},\alpha_{s})- D_{x}\mathcal{H}(s,X_{s}^{\alpha'},Y_{s}^{\alpha'},Z_{s}^{\alpha'},\alpha'_{s})\Bigr)\Bigr]ds\right].
\end{align*}
Following what was done in \cite{siska2020gradient} we decompose the above integral as follows:
\begin{align*}
e^{\beta t}|Y_{t}^{\alpha}-Y_{t}^{\alpha'}|^{2}+&\E_{t}\left[\int_{t}^{T}\left|Z_{s}^{\alpha}-Z_{s}^{\alpha'}\right|^{2}ds\right]\leq K^2 e^{\beta T}\E_{t}\left[X_{T}^{\alpha}-X_{T}^{\alpha'}\right]-\beta\E_{t}\left[\int_{t}^{T}e^{\beta s}\left|Y_{t}^{\alpha}-Y_{t}^{\alpha'}\right|^{2}ds\right]\\
+&2\E_{t}\bigg[\int_{t}^{T}e^{\beta s} \left(Y_{s}^{\alpha}-Y_{s}^{\alpha'}\right)^{\top}\bigg(D_{x}b(s,X_{s}^{\alpha},\alpha_{s})\cdot\left(Y_{s}^{\alpha}-Y_{s}^{\alpha'}\right)\\
+&\left[D_{x}b(s,X_{s}^{\alpha},\alpha_{s})-D_{x}b(s,X_{s}^{\alpha'},\alpha'_{s})\right]\cdot Y_{s}^{\alpha'}+D_{x}\sigma(s,X_{s}^{\alpha},\alpha_{s}):\left(Z_{s}^{\alpha}-Z_{s}^{\alpha'}\right)\\
+&\left[D_{x}\sigma(s,X_{s}^{\alpha},\alpha_{s}):Z_{s}^{\alpha'}-D_{x}\sigma(s,X_{s}^{\alpha'},\alpha'_{s}):Z_{s}^{\alpha'}\right]+D_{x}f(s,X_{s}^{\alpha},\alpha_{s})-D_{x}f(s,X_{s}^{\alpha'},\alpha'_{s})\bigg)ds\bigg]\\
\leq&K^2e^{\beta T}\E_{t}\left[\left|X_{T}^{\alpha}-X_{T}^{\alpha'}\right|^{2}\right]-\beta\E_{t}\left[\int_{t}^{T}e^{\beta s}\left|Y_{t}^{\alpha}-Y_{t}^{\alpha'}\right|^{2}ds\right]+I_{1}+I_{2}+I_{3}+I_{4}+I_{5},
\end{align*}
where 
\begin{align*}
I_{1}&\leq\E_{t}\left[\int_{t}^{T}e^{\beta s}\left|Y_{s}^{\alpha}-Y_{s}^{\alpha'}\right|^{2}\left|D_{x}b(s,X_{s}^{\alpha},\alpha_{s})\right|ds\right]\leq K\E_{t}\left[\int_t^Te^{\beta s}\left|Y_{s}^{\alpha}-Y_{s}^{\alpha'}\right|^{2}ds\right]\\
I_{2}&=\E_{t}\left[\int_{t}^{T}e^{\beta s}\left(Y_{s}^{\alpha}-Y_{s}^{\alpha'}\right)^{\top}\left(D_{x}b(s,X_{s}^{\alpha},\alpha_{s})-D_{x}b(s,X_{s}^{\alpha'},\alpha'_{s})\right)Y_{s}^{\alpha}ds\right]\\
I_{3}&\leq\E_{t}\left[\int_{t}^{T}e^{\beta s}\left|Y_{s}^{\alpha}-Y_{s}^{\alpha'}\right|\left|Z_{s}^{\alpha}-Z_{s}^{\alpha'}\right|\left|D_{x}\sigma(s,X_{s}^{\alpha},\alpha_{s})\right|ds\right]\\
I_{4}&\leq\E_{t}\left[\int_{t}^{T}e^{\beta s}\left|Y_{s}^{\alpha}-Y_{s}^{\alpha'}\right|\left|Z_{s}^{\alpha'}\right|\left|D_{x}\sigma(s,X_{s}^{\alpha},\alpha_{s})-D_{x}\sigma(s,X_{s}^{\alpha'},\alpha'_{s})\right|ds\right]\\
I_{5}&=\E_{t}\left[\int_{t}^{T}e^{\beta s}\left|Y_{s}^{\alpha}-Y_{s}^{\alpha'}\right|\left|D_{x}f(s,X_{s}^{\alpha},\alpha_{s})-D_{x}f(s,X_{s}^{\alpha'},\alpha'_{s})\right|ds\right].
\end{align*}
We now bound each of these in turn. For this we will need a slight variation of Young's Inequality namely that for $a,b>0$, $ab\leq\frac{\lambda^{p}a^{p}}{p}+\frac{\lambda^{-q}b^{q}}{q}$ for every $0<\lambda<\infty$ and $p,q$ H{\"o}lder conjugates. In particular for $p=q=2$ we have this gives that for $0<\gamma<\infty$, $ab\leq\frac{1}{2}\left(\gamma a^2 + \gamma^{-1}b^2\right)$. For $I_{3}$,
\begin{align*}
I_{3}\leq& K\E_{t}\left[\int_{t}^{T}e^{\beta s}\left|Y_{s}^{\alpha}-Y_{s}^{\alpha'}\right|\left|Z_{s}^{\alpha}-Z_{s}^{\alpha'}\right|ds\right]\leq\frac{K}{2}\E_{t}\left[\int_{t}^{T}e^{\beta s}\left(\gamma^{-1}\left|Y_{s}^{\alpha}-Y_{s}^{\alpha'}\right|^{2}+\gamma\left|Z_{s}^{\alpha}-Z_{s}^{\alpha'}\right|^{2}\right)ds\right]
\end{align*}
For $I_{5}$, we proceed as follows
\begin{align*}
I_{5}\leq& K\E_{t}\left[\int_{t}^{T}e^{\beta s}\left|Y_{s}^{\alpha}-Y_{s}^{\alpha'}\right|\left|X_{s}^{\alpha}-X_{s}^{\alpha'}\right|+e^{\beta s}\left|Y_{s}^{\alpha}-Y_{s}^{\alpha'}\right|\left|\alpha_{s}-\alpha'_{s}\right|ds\right]\\
\leq&K\gamma\E_{t}\left[\int_{t}^{T}e^{\beta s}\left|Y_{s}^{\alpha}-Y_{s}^{\alpha'}\right|^{2}ds\right]+\frac{K\gamma^{-1}}{2}\E_{t}\left[\int_{t}^{T}e^{\beta s}\left|X_{s}^{\alpha}-X_{s}^{\alpha'}\right|^{2}ds\right]\\
&+ \frac{K\gamma^{-1}}{2} \E_{t}\left[\int_{t}^{T}e^{\beta s}\left|\alpha_{s}-\alpha'_{s}\right|^{2}ds\right]\\.
\end{align*}
For $I_{2}$ we proceed as follows
\begin{align*}\label{helper_bound}
I_{2}\leq& C_{\mathbb H^\infty,\,\text{BMO}}\E_{t}\left[\int_{t}^{T}e^{\beta s}\left|Y_{s}^{\alpha}-Y_{s}^{\alpha'}\right|\left|D_{x}b(s,X_{s}^{\alpha},\alpha_{s})-D_{x}b(s,X_{s}^{\alpha'},\alpha'_{s})\right|ds\right]\\
\leq&\frac{C_{\mathbb H^\infty,\,\text{BMO}}\gamma }{2}\E_{t}\left[\int_{t}^{T}e^{\beta s}\left|Y_{s}^{\alpha}-Y_{s}^{\alpha'}\right|^{2} ds\right]+\frac{C_{\mathbb H^\infty,\,\text{BMO}}\gamma^{-1}}{2} \E_{t}\left[\int_{t}^{T}e^{\beta s}\left|D_{x}b(s,X_{s}^{\alpha},\alpha_{s})-b(s,X_{s}^{\alpha'},\alpha'_{s})\right|^{2}ds\right]\\
\leq&\frac{C_{\mathbb H^\infty,\,\text{BMO}}\gamma }{2}\E_{t}\left[\int_{t}^{T}e^{\beta s}\left|Y_{s}^{\alpha}-Y_{s}^{\alpha'}\right|^{2} ds\right]+\frac{C_{\mathbb H^\infty,\,\text{BMO}}\gamma^{-1}}{2}K^2\E_{t}\left[\int_{t}^{T}e^{\beta s} \left|\alpha_{s}-\alpha'_{s}\right|^{2}ds\right]\\
&+\frac{C_{\mathbb H^\infty,\,\text{BMO}}\gamma^{-1}}{2}K^2\E_{t}\left[\int_{t}^{T}e^{\beta s} \left|X^{\alpha}_{s}-X^{\alpha'}_{s}\right|^{2}ds\right].
\end{align*}
Pick $\beta=\frac{C_{\mathbb H^\infty,\,\text{BMO}}\gamma}{2}+K+\frac{K\gamma^{-1}}{2}+K\gamma$ and for now leave $0<\gamma<\infty$ free. Combining everything we have so far 
\begin{align}
e^{\beta t} \left|Y_{t}^{\alpha}-Y_{t}^{\alpha'}\right|^{2}+&\left(1-\frac{K\gamma}{2}\right)\E_{t}\left[\int_{t}^{T}e^{\beta s} \left|Z_{s}^{\alpha}-Z_{s}^{\alpha'}\right|^{2}ds\right]\leq K^2e^{\beta T} \E_{t}\left[\left|X_{T}^{\alpha}-X_{T}^{\alpha'}\right|^{2}\right]\nonumber\\
+&\left(\frac{K\gamma^{-1}}{2} + \frac{D\gamma^{-1}K^2}{2} \right) \E_{t}\left[\int_{t}^{T}e^{\beta s}\left|\alpha_{s}-\alpha'_{s}\right|^{2}ds\right]\\
+&\left(\frac{K\gamma^{-1}}{2}+\frac{D\gamma^{-1}K^2}{2}\right)\E_{t}\left[\int_{t}^{T}e^{\beta s}\left|X_{s}^{\alpha}-X_{s}^{\alpha'}\right|^{2}ds\right]+I_{4}\nonumber.
\end{align}
We now consider $I_{4}$. From Lemma \ref{technical_result_for_BSDEs} we know that the process $\left(\left|Z_{t}^{\alpha}\right|\circ W^{j}\right)$ is a BMO martingale for any $1\leq j\leq d'$. In order to apply Theorem \ref{fefferman_inq} we want to check the process 
\begin{align*}
H_{t}:=\int_{0}^{t}e^{\beta s}\left|Y_{s}^{\alpha}-Y_{s}^{\alpha'}\right|\left|D_{x}\sigma(s,X_{s}^{\alpha},\alpha_{s})-D_{x}\sigma(s,X_{s}^{\alpha'},\alpha'_{s})\right|dW^{j}
\end{align*}
satisfies the following properties:
\begin{enumerate}
\item It is an $\mathbb{F}$-martingale. To check this it suffices to note that 
\begin{align*}
\E&\left[\int_{0}^{t}e^{2\beta s}\left|Y_{s}^{\alpha}-Y_{s}^{\alpha'}\right|^{2}\left|X_{s}^{\alpha}-X_{s}^{\alpha'}\right|^{2}+\left|Y_{s}^{\alpha}-Y_{s}^{\alpha'}\right|^{2}\left|\alpha_{s}-\alpha_{s}'\right|^{2}ds\right]<\infty,
\end{align*}
where we use the inequality $(a+b)^{2}\leq 2a^2+2b^2$ , the fact  $X^{\alpha}$, $X^{\alpha'}$ solve the controlled SDE and the bound in Lemma \ref{technical_result_for_BSDEs} for $Y^\alpha$ and $Y^{\alpha'}$.
\item Secondly we want to show $\E\left[\sqrt{\left\langle H\right\rangle_{T}}\right]<\infty$. To see this the calculation is identical to what was done in checking the $\mathbb{F}$-martingale requirement.
\end{enumerate}
We now apply Fefferman's Inequality to the BMO martingale $\left|Z^{\alpha}\right|\circ W^{j}$ and the It\^{o} integral $H_{t}$
\begin{align*}
I_{4}=&\E_{t}\left[\langle\left(\left|Z^{\alpha}\right|\circ W^{j}\right),H\rangle_{T}-\langle\left(\left|Z^{\alpha}\right|\circ W^{j}\right),H\rangle_{t}\right]\leq\E\left[\langle\left(\left|Z^{\alpha}\right|\circ W^{j}\right),H\rangle_{T}-\langle\left(\left|Z^{\alpha}\right|\circ W^{j}\right),H\rangle_{t}\right]\\
=&\E\left[\int_{t}^{T}d\langle\left(\left|Z^{\alpha}\right|\circ W^{j}\right),H\rangle_{t}\right]\leq\E\left[\int_{0}^{T}\left|d\langle\left(\left|Z^{\alpha}\right|\circ W^{j}\right),H\rangle_{t}\right|\right]\leq\sqrt{2}\|\left|Z^{\alpha}\right|\circ W^j\|_{\text{BMO}}\E\left[\langle H\rangle^{\frac{1}{2}}_{T}\right],
\end{align*}
where the final inequality follows from Fefferman's inequality. Evaluating the quadratic variation of $H$ gives
\begin{align*}
\langle H\rangle^{\frac{1}{2}}_{T}=\left(\int_{0}^{T}e^{2\beta s}\left|Y_{s}^{\alpha}-Y_{s}^{\alpha'}\right|^{2}\left|D_{x}\sigma(s,X_{s}^{\alpha},\alpha_{s})-D_{x}\sigma(s,X_{s}^{\alpha'},\alpha'_{s})\right|^{2} ds\right)^{\frac{1}{2}}.
\end{align*} 
This gives the following bound,
\begin{align*}
I_{4}\leq& C_{\mathbb H^\infty,\,\text{BMO}}\sqrt{2}\E\left[\left(\int_{0}^{T}e^{2\beta s}\left|Y_{s}^{\alpha}-Y_{s}^{\alpha'}\right|^{2}\left|D_{x}\sigma(s,X_{s}^{\alpha},\alpha_{s})-D_{x}\sigma(s,X_{s}^{\alpha'},\alpha'_{s})\right|^{2} ds\right)^{\frac{1}{2}}\right]\\
\leq&e^{\beta T} C_{\mathbb H^\infty,\,\text{BMO}}\sqrt{2}\E\left[\left(\sup_{0\leq t\leq T}\left|Y_{t}^{\alpha}-Y_{t}^{\alpha'}\right|^{2}\right)^{\frac{1}{2}}\left(\int_{0}^{T}\left|D_{x}\sigma(s,X_{s}^{\alpha},\alpha_{s})-D_{x}\sigma(s,X_{s}^{\alpha'},\alpha'_{s})\right|^{2} ds\right)^{\frac{1}{2}}\right]\\
\leq&\frac{e^{\beta T} C_{\mathbb H^\infty,\,\text{BMO}}\sqrt{2}}{2}\gamma\E\left[\sup_{0\leq t\leq T}\left|Y_{t}^{\alpha}-Y_{t}^{\alpha'}\right|^{2}\right]\\
&+\frac{e^{\beta T} C_{\mathbb H^\infty,\,\text{BMO}}\sqrt{2}}{2}\gamma^{-1}\E\left[\int_{0}^{T}\left|D_{x}\sigma(s,X_{s}^{\alpha},\alpha_{s})-D_{x}\sigma(s,X_{s}^{\alpha'},\alpha'_{s})\right|^{2} ds\right]\\
\leq&\frac{e^{\beta T} C_{\mathbb H^\infty,\,\text{BMO}}\sqrt{2}}{2}\gamma\E\left[\sup_{0\leq t\leq T}\left|Y_{t}^{\alpha}-Y_{t}^{\alpha'}\right|^{2}\right]\\
&+\frac{e^{\beta T} C_{\mathbb H^\infty,\,\text{BMO}}\sqrt{2}}{2}\gamma^{-1}K^{2}2\E\left[\int_{0}^{T}\left|X_{s}^{\alpha}-X_{s}^{\alpha'}\right|^{2}+\left|\alpha_{s}-\alpha_{s}'\right|^{2}ds\right]\\
\leq&\frac{e^{\beta T} C_{\mathbb H^\infty,\,\text{BMO}}\sqrt{2}}{2}\gamma\E\left[\sup_{0\leq t\leq T}\left|Y_{t}^{\alpha}-Y_{t}^{\alpha'}\right|^{2}\right]+C_{T,K,C_{\mathbb H^\infty,\,\text{BMO}},\gamma}\|\alpha-\alpha'\|^2_{\mathbb{H}_{p}^{2}}.
\end{align*}
Combining everything together we now have 
\begin{align}\label{test}
e^{\beta t} \left|Y_{t}^{\alpha}-Y_{t}^{\alpha'}\right|^{2}+&\left(1-\frac{K\gamma}{2}\right)\E_{t}\left[\int_{t}^{T}e^{\beta s} \left|Z_{s}^{\alpha}-Z_{s}^{\alpha'}\right|^{2}ds\right]\leq \E_{t}\left[\left|X_{T}^{\alpha}-X_{T}^{\alpha'}\right|^{2}\right]\nonumber\\
+&\left(\frac{K\gamma^{-1}}{2} + \frac{C_{\mathbb H^\infty,\,\text{BMO}}\gamma^{-1}K^2}{2} \right) \E_{t}\left[\int_{t}^{T}e^{\beta s}\left|\alpha_{s}-\alpha'_{s}\right|^{2}ds\right]\\
+&\left(\frac{K\gamma^{-1}}{2}+\frac{C_{\mathbb H^\infty,\,\text{BMO}}\gamma^{-1}K^2}{2}\right)\E_{t}\left[\int_{t}^{T}e^{\beta s}\left|X_{s}^{\alpha}-X_{s}^{\alpha'}\right|^{2}ds\right]\nonumber\\
+&\frac{e^{\beta T} C_{\mathbb H^\infty,\,\text{BMO}}\sqrt{2}}{2}\gamma\E\left[\sup_{0\leq t\leq T}\left|Y_{t}^{\alpha}-Y_{t}^{\alpha'}\right|^{2}\right]+C_{T,K,C_{\mathbb H^\infty,\,\text{BMO}},\gamma}\|\alpha-\alpha'\|^2_{\mathbb{H}_{p}^{2}}\nonumber.
\end{align}
If we now take $t=0$, pick $\gamma$ such that $0<1-\frac{K\gamma}{2}<1$, take expectations and apply the total law of probability to the inequality we have the following:
\begin{align*}
\left(1-\frac{K\gamma}{2}\right)&\E\left[\int_0^{T}e^{\beta s}\left|Z_{s}^{\alpha}-Z_{s}^{\alpha'}\right|^{2}ds\right]\\
\leq&C_{K,C_{\mathbb H^\infty,\,\text{BMO}},\gamma^{-1},T}\|\alpha-\alpha'\|_{\mathbb{H}_{p}^{2}}^{2}+\frac{e^{\beta T} C_{\mathbb H^\infty,\,\text{BMO}}\sqrt{2}}{2}\gamma\E\left[\sup_{0\leq t\leq T}\left|Y_{t}^{\alpha}-Y_{t}^{\alpha'}\right|^{2}\right],
\end{align*}
where we have used the bound from the first part of the proof. Also, pick $\gamma$ such that $0<1-\frac{e^{\beta T} C_{\mathbb H^\infty,\,\text{BMO}}\sqrt{2}}{2}\gamma<1$, then from (\ref{test}) if we instead take a supremum over $t\in[0,T]$ then an expectation we have
\begin{align*}
\left(1-\frac{e^{\beta T} C_{\mathbb H^\infty,\,\text{BMO}}\sqrt{2}}{2}\gamma\right)\left(\E\left[\sup_{t\in[0,T]}\left|Y_{t}^{\alpha}-Y_{t}^{\alpha'}\right|^{2}\right]\right)\leq C_{K,C_{\mathbb H^\infty,\,\text{BMO}},\gamma^{-1},T}\|\alpha-\alpha'\|_{\mathbb{H}_{p}^{2}}^{2},
\end{align*}
where again we have used the bound from the first part of the proof. So provided we pick $\gamma\in(0,\infty)$ such that 
\begin{align*}
\begin{cases}
0<1-\frac{K\gamma}{2}<1\\
0<1-\frac{e^{\beta T} C_{\mathbb H^\infty,\,\text{BMO}}\sqrt{2}}{2}\gamma<1
\end{cases}\implies\gamma<\min\left\{\frac{2}{K},\frac{\sqrt{2}e^{-\beta T}}{C_{\mathbb H^\infty,\,\text{BMO}}}\right\}, 
\end{align*}
we have shown, 
\begin{align*}
\E\left[\sup_{t\in[0,T]}\left|Y_{t}^{\alpha}-Y_{t}^{\alpha'}\right|^{2}\right]\leq C_{K,C_{\mathbb H^\infty,\,\text{BMO}},\gamma^{-1},T}\left(1-\frac{e^{\beta T} C_{\mathbb H^\infty,\,\text{BMO}}\sqrt{2}}{2}\gamma\right)^{-1}\|\alpha-\alpha'\|_{\mathbb{H}_{p}^{2}}^{2},
\end{align*}
and 
\begin{align*}
\E\left[\int_{0}^{T}\left|Z_{s}^{\alpha}-Z_{s}^{\alpha'}\right|^{2}ds\right]\leq\left(1-\frac{K\gamma}{2}\right)^{-1}C_{K,C,\gamma^{-1},T}\|\alpha-\alpha'\|_{\mathbb{H}_{p}^{2}}^{2}
\end{align*}
\end{proof}

\subsection{Proof of Lemma \ref{cost_functional_increment_bound}}\label{proof:cost_functional_increment_bound}
\begin{proof}
We start with the following inequality, which holds for any admissible controls $\varphi$ and $\theta$, which is proven in Lemma 2.3  \cite{kerimkulov2021modified} - the idea is to add and subtract relevant terms to the cost functional $J$ and then estimate accordingly, the details are omitted,
\begin{align*}
J(\varphi)-J(\theta)\leq&\E\left[\int_{0}^{T}\mathcal{H}(t,X_{t}^{\varphi},Y_{t}^{\theta},Z_{t}^{\theta},\varphi_{t})-\mathcal{H}(t,X_{t}^{\theta},Y_{t}^{\theta},Z_{t}^{\theta},\theta_{t})dt\right]\\
&-\E\left[\int_{0}^{T}\left(X_{t}^{\varphi}-X_{t}^{\theta}\right)^{\top}D_{x}\mathcal{H}(t,X_{t}^{\theta},Y_{t}^{\theta},Z_{t}^{\theta},\theta_{t})dt\right]+\frac{K}{2}\E\left[\left|X_{T}^{\varphi}-X_{T}^{\theta }\right|^{2}\right],
\end{align*} 
and then applying Taylor's theorem to the term, 
\begin{align*}
\mathcal{H}(t,X_{t}^{\varphi},Y_{t}^{\theta},Z_{t}^{\theta},\varphi_{t})=&\mathcal{H}(t,X_{t}^{\theta},Y_{t}^{\theta},Z_{t}^{\theta},\varphi_{t}) + (X_{t}^{\varphi}-X_{t}^{\theta})^{\top}D_{x}\mathcal{H}(t,X_{t}^{\theta},Y_{t}^{\theta},Z_{t}^{\theta},\varphi_{t})\\
&+\frac{1}{2}\left(X_{t}^{\varphi}-X_{t}^{\theta}\right)^{\top}D_{xx}^{2}\mathcal{H}(t,C_{t},Y_{t}^{\theta},Z_{t}^{\theta},\varphi_{t})\left(X_{t}^{\varphi}-X_{t}^{\theta}\right),
\end{align*}
for some $\left(C_{t}( \omega)\right)_{t\in[0,T]}$. 
This is where the authors in~\cite{kerimkulov2021modified} apply the assumption $D^2_x\sigma=0$ however since we no longer have this assumption we proceed as follows. 
\begin{align}\label{initial_bound}
J(\varphi)&-J(\theta)\leq\E\left[\int_0^T\mathcal{H}(t,X_{t}^{\theta},Y_{t}^{\theta},Z_{t}^{\theta},\varphi_{t})-\mathcal{H}(t,X_{t}^{\theta},Y_{t}^{\theta},Z_{t}^{\theta},\theta_{t})dt\right]+\frac{K}{2}\E\left[\left|X_{T}^{\varphi}-X_{T}^{\theta }\right|^{2}\right]\nonumber\\
-& \mathbb E\left[\int_0^{T}(X_{t}^{\varphi}-X_{t}^{\theta})^{\top}\left(D_{x}\mathcal{H}(t,X_{t}^{\theta},Y_{t}^{\theta},Z_{t}^{\theta},\theta_{t})-D_{x}\mathcal{H}(t,X_{t}^{\theta},Y_{t}^{\theta},Z_{t}^{\theta},\varphi_{t})\right)dt\right]\\
+&\frac{1}{2}\E\left[\int_{0}^{T}\left(X_{t}^{\varphi}-X_{t}^{\theta}\right)^{\top}D_{xx}^{2}\mathcal{H}(t,C_{t},Y_{t}^{\theta},Z_{t}^{\theta},\varphi_{t})\left(X_{t}^{\varphi}-X_{t}^{\theta}\right)dt\right]\nonumber.
\end{align}		
We now introduce the following notation: $b^{i}_{x_j}$ represents the $x_j$ derivative of the ith entry in the vector $b(t,x,a)$, and same for $\sigma^{i,j}_{x_k}$. 
\begin{align*}
\mathbb E & \left[\int_0^{T}(X_{t}^{\varphi}-X_{t}^{\theta})^{\top}\left(D_{x}\mathcal{H}(t,X_{t}^{\theta},Y_{t}^{\theta},Z_{t}^{\theta},\theta_{t})-D_{x}\mathcal{H}(t,X_{t}^{\theta},Y_{t}^{\theta},Z_{t}^{\theta},\varphi_{t})\right)dt\right]\\
=&\E\left[\int_{0}^{T}\sum_{j=1}^{d}\left(X_{t}^{\varphi}-X_{t}^{\theta}\right)_{j}\sum_{i=1}^{d}(b_{x_{j}}^{i}(t,X_{t}^{\theta},\theta_{t})-b_{x_{j}}^{i}(t,X_{t}^{\theta},\varphi_{t}))Y_{t}^{i}(\theta)dt\right]\\
&+\E\left[\int_{0}^{T}\sum_{k=1}^{d}(X_{t}^{\theta}-X_{t}^{\varphi})_{k}\sum_{j=1}^{d'}\sum_{i=1}^{d}\left(\sigma_{x_k}^{i,j}(t,X_{t}^{\theta},\theta_{t})-\sigma_{x_{k}}^{i,j}(t,X_{t}^{\theta},\varphi_{t})\right)\left(Z^{\theta}_t\right)^{i,j}dt\right]\\
&+\E\left[\int_{0}^{T}\sum_{k=1}^{d}(X_{t}^{\varphi}-X_{t}^{\theta})_{k}(f_{x_{k}}(t,X_{t}^{\varphi},\varphi_{t})-f_{x_{k}}(t,X_{t}^{\varphi},\theta_{t}))dt\right]\\
\leq&\|Y^{\theta}\|_{\mathbb{H}^{\infty}_{d}}\E\left[\int_0^{T}\sum_{j=1}^{d}\left(\frac{(X_{t}^{\varphi}-X_{t}^{\theta})_{j}^{2}}{2}+\frac{\left(\sum_{i=1}^{d}(b^{i}_{x_{j}}(t,X_{t}^{\theta},\theta_{t})-b^{i}_{x_{j}}(t,X_{t}^{\theta},\varphi_{t}))\right)^{2}}{2}\right)dt\right]\\
&+\sum_{k=1}^{d}\sum_{j=1}^{d'}\sum_{i=1}^{d}\E\left[\int_{0}^{T}(X_{t}^{\theta}-X_{t}^{\varphi})_{k}\left(\sigma_{x_k}^{i,j}(t,X_{t}^{\theta},\theta_{t})-\sigma_{x_{k}}^{i,j}(t,X_{t}^{\theta},\varphi_{t})\right)\left(Z^{\theta}_t\right)^{i,j}dt\right]\\
&+\frac{1}{2}\E\left[\int_{0}^{T}|X_{t}^{\varphi}-X_{t}^{\theta}|^{2}+|D_xf(t,X_{t}^{\varphi},\varphi_{t})-D_xf(t,X_{t}^{\varphi},\theta_{t})|^{2}dt\right]\\
\leq&\frac{d\|Y^{\theta}\|_{\mathbb{H}^{\infty}}}{2^2}\E\left[\int_{0}^{T}\left|X_{t}^{\varphi}-X_{t}^{\theta}\right|^{2}+\sum_{j=1}^{d}\sum_{i=1}^{d}|b^{i}_{x_{j}}(t,X_{t}^{\theta},\theta_{t})-b_{x_{j}}^{i}(t,X_{t}^{\theta},\varphi)|^{2}dt\right]\\
&+\sum_{k=1}^{d}\sum_{j=1}^{d'}\sum_{i=1}^{d}\E\left[\int_{0}^{T}(X_{t}^{\theta}-X_{t}^{\varphi})_{k}\left(\sigma_{x_k}^{i,j}(t,X_{t}^{\theta},\theta_{t})-\sigma_{x_{k}}^{i,j}(t,X_{t}^{\theta},\varphi_{t})\right)\left(Z^{\theta}_t\right)^{i,j}dt\right]\\
&+\frac{1}{2}C_{K,T}\E\left[\int_{0}^{T}|\varphi_{t}-\theta_{t}|^{2}dt\right]\\
=&\frac{d\|Y^{\theta}\|_{\mathbb{H}^{\infty}}}{2^2}\E\left[\int_{0}^{T}\left|X_{t}^{\varphi}-X_{t}^{\theta}\right|^{2}+\left|D_{x}b(t,X_{t}^{\theta},\theta_{t})-D_{x}b(t,X_{t}^{\theta},\varphi_{t})\right|^{2}dt\right]+C_{K,T}\E\left[\int_{0}^{T}|\varphi_{t}-\theta_{t}|^{2}dt\right]\\
&+\sum_{k=1}^{d}\sum_{j=1}^{d'}\sum_{i=1}^{d}\E\left[\int_{0}^{T}(X_{t}^{\theta}-X_{t}^{\varphi})_{k}\left(\sigma_{x_k}^{i,j}(t,X_{t}^{\theta},\theta_{t})-\sigma_{x_{k}}^{i,j}(t,X_{t}^{\theta},\varphi_{t})\right)\left(Z^{\theta}_t\right)^{i,j}dt\right]\\
\leq& C_{K,T,C_{\mathbb H^\infty,\,\text{BMO}}}\E\left[\int_{0}^{T}\left|\theta_{t}-\varphi_{t}\right|^{2}dt\right]\\
&+\sum_{k=1}^{d}\sum_{j=1}^{d'}\sum_{i=1}^{d}\E\left[\int_{0}^{T}(X_{t}^{\theta}-X_{t}^{\varphi})_{k}\left(\sigma_{x_k}^{i,j}(t,X_{t}^{\theta},\theta_{t})-\sigma_{x_{k}}^{i,j}(t,X_{t}^{\theta},\varphi_{t})\right)\left(Z^{\theta}_t\right)^{i,j}dt\right].
\end{align*}
where we applied Assumptions \ref{Ass_2nd_derivative_sapce} and \ref{additional_assumptions} and applied the bound \eqref{forward_bound} in Lemma \ref{lemma bsde stability}. Now for the terms involving $\left(Z^{\theta}_t\right)^{i,j}$ we proceed as follows:
\begin{align*}
&\sum_{k=1}^{d}\sum_{j=1}^{d'}\sum_{i=1}^{d}\E\left[\int_{0}^{T}(X_{t}^{\theta}-X_{t}^{\varphi})_{k}\left(\sigma_{x_k}^{i,j}(t,X_{t}^{\theta},\theta_{t})-\sigma_{x_{k}}^{i,j}(t,X_{t}^{\theta},\varphi_{t})\right)\left(Z^{\theta}_t\right)^{i,j}dt\right]\\
&\leq\sum_{k=1}^{d}\sum_{j=1}^{d'}\sum_{i=1}^{d}\E\left[\int_{0}^{T}|d\left\langle(X_{t}^{\theta}-X_{t}^{\varphi})_{k}\left(\sigma_{x_k}^{i,j}(t,X_{t}^{\theta},\theta_{t})-\sigma_{x_{k}}^{i,j}(t,X_{t}^{\theta},\varphi_{t})\right)\circ W^{j},\left(Z^{\theta}_t\right)^{i,j}\circ  W^{j}\right\rangle|\right]\\
&\leq\sum_{k=1}^{d}\sum_{j=1}^{d'}\sum_{i=1}^{d}\|\left(Z^{\theta}_t\right)^{i,j}\circ W\|_{\text{BMO}} \E\left[\left(\int_{0}^{T}(X_{t}^{\theta}-X_{t}^{\varphi})_{k}^{2}\left(\sigma_{x_k}^{i,j}(t,X_{t}^{\theta},\theta_{t})-\sigma_{x_{k}}^{i,j}(t,X_{t}^{\theta},\varphi_{t})\right)^{2}dt\right)^{\frac{1}{2}}\right]\\
&\leq C_{\mathbb H^\infty,\,\text{BMO}}\sum_{k=1}^{d}\sum_{j=1}^{d'}\sum_{i=1}^{d}\E\left[\left(\sup_{t\in[0,T]}(X_{t}^{\theta}-X_{t}^{\varphi})_{k}^{2}\right)^{\frac{1}{2}}\left(\int_{0}^{T}\left(\sigma_{x_k}^{i,j}(t,X_{t}^{\theta},\theta_{t})-\sigma_{x_{k}}^{i,j}(t,X_{t}^{\theta},\varphi_{t})\right)^{2}dt\right)^{\frac{1}{2}}\right]\\
&\leq\frac{C_{\mathbb H^\infty,\,\text{BMO}}}{2}\sum_{k=1}^{d}\sum_{j=1}^{d'}\sum_{i=1}^{d}\E\left[\sup_{t\in[0,T]}(X_{t}^{\theta}-X_{t}^{\varphi})_{k}^{2}+\int_{0}^{T}\left(\sigma_{x_k}^{i,j}(t,X_{t}^{\theta},\theta_{t})-\sigma_{x_{k}}^{i,j}(t,X_{t}^{\theta},\varphi_{t})\right)^{2}dt\right]\\
&\leq\frac{C_{\mathbb H^\infty,\,\text{BMO}}}{2}\E\left[\sup_{t\in[0,T]}|X_{t}^{\theta}-X_{t}^{\varphi}|^2+\int_{0}^{T}\left|D_{x}\sigma(t,X_{t}^{\theta},\theta_{t})-D_x\sigma(t,X_{t}^{\theta},\varphi_{t})\right|^{2}dt\right]\\
&\leq C_{C_{\mathbb H^\infty,\,\text{BMO}},K}\E\left[\int_0^T\left|\theta_t - \varphi_{t}\right|^{2}dt\right]
\end{align*}
Therefore, 
\begin{align*}
E\bigg[\int_0^{T}(X_{t}^{\varphi}&-X_{t}^{\theta})^{\top}\left(D_{x}\mathcal{H}(t,X_{t}^{\theta},Y_{t}^{\theta},Z_{t}^{\theta},\theta_{t})-D_{x}\mathcal{H}(t,X_{t}^{\theta},Y_{t}^{\theta},Z_{t}^{\theta},\varphi_{t})\right)dt\bigg]\\
&\leq C_{T,K,C_{\mathbb H^\infty,\,\text{BMO}}}\E\left[\int_0^T\left|\theta_t-\varphi_t\right|^{2}dt\right].
\end{align*}
and substituting this into (\ref{initial_bound}) and applying Lemma \ref{lemma bsde stability} one more time for the term $\E\left[
|X_T^{\theta}-X_T^\alpha|\right]$ we have the following
\begin{align*}
J(\varphi)-J(\theta)\leq&\E\left[\int_0^T\mathcal{H}(t,X_{t}^{\theta},Y_{t}^{\theta},Z_{t}^{\theta},\varphi_{t})-\mathcal{H}(t,X_{t}^{\theta},Y_{t}^{\theta},Z_{t}^{\theta},\theta_{t})dt\right]\\
&+C_{T,K,C_{\mathbb H^\infty,\,\text{BMO}}}\E\left[\int_0^T\left|\theta_t-\varphi_t\right|^2dt\right]\\
&+\frac{1}{2}\E\left[\int_{0}^{T}\left(X_{t}^{\varphi}-X_{t}^{\theta}\right)^{\top}D_{xx}^{2}\mathcal{H}(t,C_{t},Y_{t}^{\theta},Z_{t}^{\theta},\varphi_{t})\left(X_{t}^{\varphi}-X_{t}^{\theta}\right)dt\right].
\end{align*}
For the term $\E\left[\int_{0}^{T}\left(X_{t}^{\varphi}-X_{t}^{\theta}\right)^{\top}D_{xx}^{2}\mathcal{H}(t,C_{t},Y_{t}^{\theta},Z_{t}^{\theta},\varphi_{t})\left(X_{t}^{\varphi}-X_{t}^{\theta}\right)dt\right]$, we proceed in the same way.
\begin{align*}
\E&\left[\int_{0}^{T}\left(X_{t}^{\varphi}-X_{t}^{\theta}\right)^{\top}D_{xx}^{2}\mathcal{H}(t,C_{t},Y_{t}^{\theta},Z_{t}^{\theta},\varphi_{t})\left(X_{t}^{\varphi}-X_{t}^{\theta}\right)dt\right]\\
=&\E\biggl[\int_{0}^{T}\sum_{l=1}^{d}\left(X_{t}^{\varphi}-X_{t}^{\theta}\right)_{l}\cdot\\
&\sum_{k=1}^{d}\left(X_{t}^{\varphi}-X_{t}^{\theta}\right)_{k}\left(\sum_{i=1}^{d}b_{x_{k}x_{l}}^{i}(t,C_t,\varphi_{t})\left(Y_t^\theta\right)^i + \sum_{j=1}^{d'}\sum_{i=1}^{d}\sigma_{x_k,x_l}^{i,j}(t,C_t,\varphi_t)\left(Z_t^{\theta}\right)^{ij}\right)dt\bigg]\\
\leq&\sum_{l=1}^{d}\sum_{k=1}^{d}\sum_{i=1}^{d}\E\left[\int_{0}^{T}\left(X_{t}^{\varphi}-X_{t}^{\theta}\right)_{l}\left(X_{t}^{\varphi}-X_{t}^{\theta}\right)_{k}b^{i}_{x_{k}x_{l}}(t,C_t,\varphi_{t})Y_t^{i}(\theta)dt\right]\\
&+\sum_{l=1}^{d}\sum_{k=1}^{d}\sum_{i=1}^{d}\sum_{j=1}^{d'}\E\left[\int_{0}^{T}\left(X_{t}^{\varphi}-X_{t}^{\theta}\right)_{l}\left(X_{t}^{\varphi}-X_{t}^{\theta}\right)_{k}\sigma_{x_k,x_l}^{i,j}(t,C_t,\varphi_t)\left(Z_t^{\theta}\right)^{ij}dt\right]\\
\leq&C_{K,D}\E\left[\int_0^{T}|X_t^{\varphi}-X_t^{\theta}|^{2}dt\right]\\
&+K\sum_{l=1}^{d}\sum_{k=1}^{d}\sum_{i=1}^{d}\sum_{j=1}^{d'}\E\left[\int_0^{T}\left|d\left\langle\left(X_{t}^{\varphi}-X_{t}^{\theta}\right)_{l}\left(X_{t}^{\varphi}-X_{t}^{\theta}\right)_{k}\circ W^{j},\left(Z_t^{\theta}\right)^{ij}\circ W^{j}\right\rangle_{s}\right|\right]\\
\leq& C_{K,C_{\mathbb H^\infty,\,\text{BMO}}}\E\left[\int_{0}^{T}|\varphi_t-\theta_t|^{2}dt\right]\\
&+K\sum_{l=1}^{d}\sum_{k=1}^{d}\sum_{i=1}^{d}\sum_{j=1}^{d'}\|Z^{i,j}\circ W^{j}\|_{\text{BMO}}E\left[\left(\int_{0}^{T}(X_{t}^{\varphi}-X_{t}^{\theta})_{k}(X_{t}^{\varphi}-X_{t}^{\theta})_{l}dt\right)^{\frac{1}{2}}\right]\\
\leq& C_{K,C_{\mathbb H^\infty,\,\text{BMO}}}\E\left[\int_{0}^{T}|\varphi_t-\theta_t|^{2}dt\right]+C_{d,d',K,C_{\mathbb H^\infty,\,\text{BMO}},T}\E\left[\sup_{t\in[0,T]}|X_{t}^{\varphi}-X_{t}^{\alpha}|^{2}\right]\\
\leq& C_{K,C_{\mathbb H^\infty,\,\text{BMO}},T,d,d'}\E\left[\int_0^{T}|\alpha_{t}-\theta_{t}|^{2}dt\right]\,.
\end{align*}
\end{proof}

\section*{Acknowledgements}
The authors are grateful for many fruitful discussions with B. Kerimkulov and L. Szpruch which greatly helped clarify many aspects of the problem and led to the authors being able to present more complete set of results.

\bibliographystyle{abbrv}
\bibliography{bibliography}
\end{document}